\documentclass[11pt]{article}

\usepackage[utf8]{inputenc}
\usepackage[T1]{fontenc}
\usepackage[main=english,ngerman]{babel}
\usepackage{lmodern}
\usepackage{tcolorbox}
\usepackage{enumerate}
\usepackage{multirow}
\usepackage{nicefrac}
\usepackage{caption}
\usepackage{graphicx}
\usepackage{tikz}
\usetikzlibrary{patterns}
\usetikzlibrary{arrows}
\usepackage{tikz-3dplot}
\usepackage{subfigure}
\usepackage{pgfplots}
\usepackage{hyperref}
\usepackage{picture}
\usepackage[absolute,overlay]{textpos}
\usepackage{algorithm}
\usepackage{algpseudocode}
\usepackage{cite}
\pgfplotsset{
  log x ticks with fixed point/.style={
      xticklabel={
        \pgfkeys{/pgf/fpu=true}
        \pgfmathparse{exp(\tick)}%
        \pgfmathprintnumber[fixed relative, precision=3]{\pgfmathresult}
        \pgfkeys{/pgf/fpu=false}
      }
  },
  log y ticks with fixed point/.style={
      yticklabel={
        \pgfkeys{/pgf/fpu=true}
        \pgfmathparse{exp(\tick)}%
        \pgfmathprintnumber[fixed relative, precision=3]{\pgfmathresult}
        \pgfkeys{/pgf/fpu=false}
      }
  }
}

\usepackage[top=2.5cm,bottom=2.5cm]{geometry}
\usepackage{float}

\usepackage{amsmath,amssymb,amsthm}
\usepackage{mathtools}
\usepackage[mathscr]{euscript}
\usepackage{xpatch}

\providecommand{\keywords}[1]
{
  \small	
  \textbf{\textit{Keywords---}} #1
}

\newtheorem{definition}{Definition}
\theoremstyle{remark}
\newtheorem{remark}{Remark}
\theoremstyle{theorem}

\theoremstyle{theorem}
\newtheorem{lemma}{Lemma}

\newtheorem{assumption}{Assumption}

\usepackage[rightcaption]{sidecap}

\renewcommand{\o}{\Omega}

\newcommand{\p}{\textbf}
\newcommand{\f}{\boldsymbol}

\graphicspath{{Images/}}

\begin{document}
	
	\title{Scaled boundary isogeometric analysis with $C^1$ coupling for Kirchhoff plate theory}
 \author{Jeremias Arf$^1$\footnote{Corresponding author: arf@rhrk.uni-kl.de \qquad \qquad \qquad \qquad \qquad \qquad Preprint submitted to CMAME}, Mathias Reichle$^2$, Sven Klinkel$^2$ and Bernd Simeon$^1$ }
 \date{ $^1$ Departement of Mathematics, RPTU Kaiserslautern-Landau, Germany\\
    $^2$ Chair of Structural Analysis and Dynamics, RWTH Aachen, Germany }
	\maketitle
	
  \begin{abstract}
      Although isogeometric analysis exploits smooth B-spline and NURBS basis functions  for the definition of discrete function spaces as well as for the geometry representation, the global smoothness in so-called multipatch parametrizations is an issue. 
      Especially, if strong $C^1$ regularity is required, the  introduction of function spaces with good convergence properties is not straightforward. However, in $2D$ there is the special class  of analysis-suitable $G^1$ (AS-$G^1$) parametrizations that are suitable for patch coupling.
      In this contribution we  show that the concept of scaled boundary isogeometric analysis fits to the AS-$G^1$ idea and the former is appropriate to define $C^1$-smooth basis functions.  The proposed method is applied to   Kirchhoff plates and its capability is demonstrated utilizing several numerical examples. Its applicability to non-trivial and trimmed shapes is demonstrated. 
  \end{abstract}

\section{Introduction}
	
Isogeometric Analysis (IGA) is a concept  which was introduced by Hughes et al.  \cite{IGA2} and developed to a widely used and very successful approach within numerical analysis and computational mathematics. 
It connects the fields of geometry representation in Computer Aided Design (CAD) with the framework of Finite Elements (FE). For both, geometry parametrization and  discretization space definition, one utilizes B-splines and NURBS (Non-Uniform Rational B-Splines).
In IGA we find  different built-in features that allow for interesting applications. 
For example, at least in the so-called single patch case, smoothness of test and ansatz functions is no issue anymore since the underlying B-spline and NURBS basis functions can be chosen with high global regularity. This is an advantage compared to classical FE ansatz spaces which are mostly  continuous or discontinuous mappings. Clearly, also $C^1$-smooth FE spaces can be defined without the ideas of IGA. But, for this regularity requirement and in  particular if we look for $C^k$ basis functions where $k>1$, the definitions become complicated.  Whereas utilizing splines, a regularity increase can be implemented efficiently. This comes along with the $k$-refinement ansatz. The latter means a simultaneous degree  and regularity  elevation step. Such a change of the discretization spaces has no pendant in standard FE theory.
The possibility to vary regularities in IGA implies two important aspects. On the one hand, IGA is a suitable discretization approach if one deals with high-order problems. On the other hand,  increasing the regularity without changing the degree leads to a significant reduction of degrees of freedom although the approximation behaviour changes only slightly. Especially, as long as  $p>r$, where $p$ stands for the B-spline degree and $r$ for regularity,  the observed convergence rates   are basically the same. For more information on IGA we recommend \cite{IGA1,IGA2,IGA3}. \\
Unfortunately, if one considers complex geometries,  multipatch IGA meshes are necessary, i.e. the computational domain is decomposed in several subdomains that are associated to isogeometric spaces.  These patch-wise spaces have to be coupled in order to get high regular basis functions on the whole domain. And this coupling is  in general not trivial. Basically, when coupling the patches one faces similar problems like for the coupling of classical FE mesh elements. This means,  the smoother the basis functions across patch interfaces, the harder the coupling procedure gets. 
Even worse, for different  geometries   locking effects arise, i.e. bad approximation results, if  we want strong global $C^r$-smoothness with $r \geq 1$; see \cite{Collin2016AnalysissuitableGM}. \\
Consequently, in literature one can find the concept of weak patch-coupling to define proper multipatch IGA spaces; see \cite{Dittmann2019,Dornisch2015,Chasapi2020}. 
But in $2D$, if one restricts oneself to a special class of parametrizations, the so-called 
\emph{analysis-suitable $G^1$ parametrizations}, the strong $C^1$ coupling of patches is  possible; cf. \cite{Collin2016AnalysissuitableGM}. Besides, for the latter geometries the problem of $C^1$-locking can be avoided. 
Further we note the framework of \emph{polar splines}, see e.g. \cite{toshniwal2017}, that are useful to handle  disk like domains  and to define globally smooth basis functions. \\
In this article, we combine the approach and the results from the mentioned reference \cite{Collin2016AnalysissuitableGM} to show that  scaled boundary IGA (SB-IGA) is suitable to introduce $C^1$-smooth basis functions. In SB-IGA we assume the domain to be defined by a scaling of its boundary w.r.t. to some given scaling center; see \cite{Arioli2019,Klinkel2019,Chasapi2018}. At first glance, this requires star-convex domains and introduces a singular parametrization mapping. Nevertheless, it fits to the fact that in CAD the computational domains are often represented via its boundaries and furthermore we explain in the subsequent parts why it is still useful for non-star-convex geometries.  To demonstrate the $C^1$ regularity and capabilities we apply our approach to the problem class of Kirchhoff plates. The related strong PDE formulation is of fourth order and hence in the classical weak formulation second derivatives appear ; see \cite{liu2018,Dieringer2011}.  Furthermore, we briefly explain, why the proposed scaled boundary meshes (SB-meshes) are convenient when dealing with trimming. Latter domain modification is often applied in IGA and central if topologically complicated domains appear. We refer to \cite{Marussig2018} for details regarding trimming.  More precisely, the structure of the proposal is the following.\\
We start with a very brief discussion of B-splines and SB-IGA in Sec. \ref{subsection:B-splines_and_SB-IGA} to clarify the mathematical notions. We proceed with two sections, Sec. \ref{section:two-patch coupling} and Sec. \ref{sec:Planar SB-IGA}, dedicated to the coupling problem. We discuss the standard two-patch case as introduced in \cite{Collin2016AnalysissuitableGM} and in a second section we show the application to  SB-meshes. Afterwards, in Sec. \ref{section:generelizations}, we briefly explain generalization possibilities, where we emphasize the application to trimmed geometries. In Sec. \ref{section:numerical_examples} we look at several numerical examples in the context of Kirchhoff plate theory. In  Sec. \ref{Sec:Stability} we discuss the problem of stability coming along with the singular parametrizations.
We close with a short conclusion in  Sec. \ref{section:conclusion}.

%\section{Theory}
\section{B-splines and SB-IGA}
\label{subsection:B-splines_and_SB-IGA}
In this section we want to introduce briefly some basic notion and clarify the SB-IGA ansatz.\\
First, we state a short overview of B-spline functions, B-spline spaces, respectively.\\
Exploiting \cite{IGA1,IGA3} for a short
exposition, we call an  increasing sequence of real numbers $\Xi \coloneqq \{ \xi_1 \leq  \xi_2  \leq \dots \leq \xi_{n+p+1}  \}$ for some $p \in \mathbb{N}$   \emph{knot vector}, where we assume  $0=\xi_1=\dots=\xi_{p+1}, \ \xi_{n+1}=\dots=\xi_{n+p+1}=1$, and call such knot vectors $p$-open. 
Further the multiplicity of the $j$-th knot is denoted by $m(\xi_j)$.
Then  the univariate B-spline functions $\widehat{B}_{j,p}(\cdot)$ of degree $p$ corresponding to a given knot vector $\Xi$ is defined recursively by the \emph{Cox-DeBoor formula}:
\begin{align}
\widehat{B}_{j,0}(\zeta) \coloneqq \begin{cases}
1, \ \ \textup{if}  \ \zeta \in [\xi_{j},\xi_{j+1}), \\
0, \ \ \textup{else},
\end{cases}
\end{align}
\textup{and if }  $p \in \mathbb{N}_{\geq 1} \ \textup{we set}$ 
\begin{align}
\widehat{B}_{j,p}(\zeta)\coloneqq \frac{\zeta-\xi_{j}}{\xi_{j+p}-\xi_j} \widehat{B}_{j,p-1}(\zeta)  +\frac{\xi_{j+p+1}-\zeta}{\xi_{j+p+1}-\xi_{j+1}} \widehat{B}_{j+1,p-1}(\zeta),
\end{align}
where one puts $0/0=0$ to obtain  well-definedness. The knot vector $\Xi$ without knot repetitions is denoted by $\{ \psi_1, \dots , \psi_N \}$. \\
The multivariate extension of the last spline definition is achieved by a tensor product construction. In other words, we set for a given tensor knot vector   $\boldsymbol{\Xi} \coloneqq \Xi_1 \times   \dots \times \Xi_d $, where the $\Xi_{l}=\{ \xi_1^{l}, \dots , \xi_{n_l+p_l+1}^{l} \}, \ l=1, \dots , d$ are $p_l$-open,   and a given \emph{degree vector}   $\f{p} \coloneqq (p_1, \dots , p_d)$ for the multivariate case
\begin{align}
\widehat{B}_{\p{i},\f{p}}(\boldsymbol{\zeta}) \coloneqq \prod_{l=1}^{d} \widehat{B}_{i_l,p_l}(\zeta_l), \ \ \ \ \forall \, \p{i} \in \mathit{\mathbf{I}}, \ \  \boldsymbol{\zeta} \coloneqq (\zeta_1, \dots , \zeta_d),
\end{align}
with  $d$ as  the underlying dimension of the parametric domain $\widehat{\Omega}= (0,1)^d$ and $\textup{\p{I}}$ the multi-index set $\textup{\p{I}} \coloneqq \{ (i_1,\dots,i_d) \  | \  1\leq i_l \leq n_l, \ l=1,\dots,d  \}$.\\
B-splines  fulfill several properties and for our purposes the most important ones are:
\begin{itemize}
	\item If  for all internal knots the multiplicity satisfies $1 \leq m(\xi_j) \leq m \leq p , $ then the B-spline basis functions $\widehat{B}_{i,p}$ are globally $C^{p-m}$-continuous. Therefore we define  in this case the regularity integer $r \coloneqq p-m$. Obviously, by the product structure, we get splines $\widehat{B}_{\p{i},\f{p}}$ which are $C^{r_l}$-smooth  w.r.t. the $l$-th coordinate direction if the internal multiplicities fulfill $1 \leq m(\xi_j^l) \leq m_l  \leq p_l, \ r_l \coloneqq p_l-m_l, \ \forall l \in 1, \dots , d$ in the multivariate case.   We write in the following $\f{r} \coloneqq (r_1,\dots ,r_d), $ for the regularity vector to indicate the smoothness. In case of $r_i <0$ we have discontinuous splines w.r.t. the $i$-th coordinate direction. To later to emphasize the regularity of the splines we introduce a upper index $r$ and write in the following $\widehat{B}_{{i},{p}}^{{r}}, \ \widehat{B}_{\p{i},\f{p}}^{\f{r}}$, respectively.
	\item For univariate splines $\widehat{B}_{i,p}^r, \, p \geq 1, \, r \geq 0$ we have
	\begin{align}
	\label{eq:soline_der}
	\partial_{\zeta} \widehat{B}_{i,p}^r(\zeta) = \frac{p}{\xi_{i+p}-\xi_i}\widehat{B}_{i,p-1}^{r-1}(\zeta) +  \frac{p}{\xi_{i+p+1}-\xi_{i+1}}\widehat{B}_{i+1,p-1}^{r-1}(\zeta),
	\end{align} 
	with $\widehat{B}_{1,p-1}^{r-1}(\zeta)\coloneqq \widehat{B}_{n+1,p-1}^{r-1}(\zeta) \coloneqq 0$.
	\item  The support of the spline $\widehat{B}_{i,p}^r $ is contained in the interval $[\xi_i,\xi_{i+p+1}]$.
 \item We have the partition of unity property $\sum_i \widehat{B}_{i,p}^r=1$.
\end{itemize}
The space spanned by all univariate splines $\widehat{B}_{i,p}^r$ corresponding to  given knot vector and degree $p$ and global regularity $r$  is denoted by $$S_p^r \coloneqq \textup{span}\{ \widehat{B}_{i,p}^r \ | \ i = 1,\dots , n \}.$$ 
For the multivariate case one can define the spline space as the product space $$S_{p_1, \dots , p_d}^{r_1,\dots,r_d} \coloneqq S_{p_1}^{r_1} \otimes \dots \otimes S_{p_d}^{r_d} = \textup{span} \{\widehat{B}_{\p{i},\f{p}}^{\f{r}} \ | \  \p{i} \in \mathit{\mathbf{I}}  \}$$ of proper univariate spline spaces. To obtain more flexibility it could be useful to introduce a strictly positive weight function $W = \sum_{\p{i}} w_{\p{i}} \, \widehat{B}_{\p{i},\f{p}}^{\f{r}} \in S^{r_1,\dots,r_d}_{p_1,\dots,p_d}$ and use  NURBS functions  $\widehat{N}_{\p{i},\f{p}}^{\f{r}} \coloneqq \frac{w_{\p{i}} \, \widehat{B}_{\p{i},\f{p}}^{\f{r}}}{W}$, the NURBS spaces ${N}_{p_1,\dots,p_d}^{r_1,\dots,r_d} \coloneqq \frac{1}{W}  S_{p_1, \dots , p_d}^{r_1,\dots,r_d}, $  respectively. In particular, NURBS can be seen as  generalized B-splines.   Later we use the abbreviations  $\widehat{N}_{{i},{p}}^{{r}} \coloneqq \frac{\widehat{B}_{{i},{p}}^{{r}}}{W}$ and ${N}_{{p}}^{{r}} \coloneqq \textup{span}\{\widehat{N}_{{i},{p}}^{{r}} \ | \ i=1,\dots  \}$ for the $1D$ case.\\

To define discrete spaces based on splines we require a  mapping $\p{F} \colon \widehat{\Omega} \coloneqq (0,1)^d \rightarrow \mathbb{R}^d$ which parametrizes the computational domain.
The knots stored in the knot vector $  \boldsymbol{\Xi} $, corresponding to  the underlying  splines, determine a mesh in the parametric domain $\widehat{\Omega} $, namely  $\widehat{M} \coloneqq \{ K_{\p{j}}\coloneqq (\psi_{j_1}^1,\psi_{j_1+1}^1 ) \times \dots \times (\psi_{j_{d}}^{d},\psi_{j_{d}+1}^{d} ) \ | \  \p{j}=(j_1,\dots,j_{d}), \ \textup{with} \ 1 \leq j_i <N_i\},$ and
with ${\boldsymbol{\Psi}}= \{\psi_1^1, \dots, \psi _{N_1}^1\}  \times \dots \times \{\psi_1^{d}, \dots, \psi _{N_{d}}^{d}\}$  \  \textup{as  the knot vector} \ ${\boldsymbol{\Xi}}$ \ 
\textup{without knot repetitions}.   
The image of this mesh under the mapping $\p{F}$, i.e. $\mathcal{M} \coloneqq \{{\p{F}}(K) \ | \ K \in \widehat{M} \}$, gives us a mesh structure in the physical domain. By inserting knots without changing the parametrization  we can refine the mesh, which is the concept of $h$-refinement; see \cite{IGA2,IGA1}.
For a mesh $\mathcal{M}$ we can introduce the global mesh size through $h \coloneqq \max\{h_{{K}} \ | \ {K} \in \widehat{M} \}$, where for ${K} \in \widehat{M}$ we denote with $h_{{K}} \coloneqq \textup{diam}({K}) $ the \emph{element size} and $\widehat{M}$ is the underlying parametric mesh.
\begin{comment}
\begin{assumption}{(Regular mesh)}\\
There exists a constant $c_u $ independent from the mesh size such that $h_{\mathcal{K}} \leq h \leq c_u \, h_{\mathcal{K}}$ for all mesh elements $\mathcal{K} \in \mathcal{M}$. 
Let the parametrization mapping $\p{F}$ be diffeomorphic and let the restrictions of $\p{F}$ to mesh elements be smooth.
\end{assumption}
\end{comment}

\vspace{0.2cm}

The underlying idea of SB-IGA takes into account that in CAD applications the computational   domain is often represented by means of its boundary. As long as the region of interest is star-shaped we can choose a scaling center $\f{x}_0 \in \mathbb{R}^d$ and the domain is then defined by a scaling of the boundary w.r.t. to $\f{x}_0$. In the planar case, which is the one we focus on here, and in view of the isogeometric analysis  we have some  boundary NURBS curve $\gamma(\zeta) = \textstyle\sum_{i}\p{C}_i \ \widehat{N}_{i,p}^{r}(\zeta), \ \p{C}_i \in \mathbb{R}^2 $ and define the SB-parametrization of $\o$ through
\begin{equation*}
\p{F} \colon \widehat{\o} \coloneqq (0,1)^2  \rightarrow \o \ , \ (\zeta,\xi) \mapsto \xi \ \big( \gamma(\zeta)-\f{x}_0\big)+\f{x}_0 \ \  \  \textup{(see Fig. \ref{Fig:SB_illustration_1})}.
\end{equation*}

\begin{figure}
	\centering
	\begin{tikzpicture}
		\draw [fill=gray, opacity=0.2]  plot[smooth,dashed] coordinates {(5,0) (6,-0.5)  (8,0.4) (6,2.3) (4.1,1.5)  (5,0) };
	\draw[->, thick] (0,0) to (2.6,0);
	\draw[->, thick] (0,0) to (0,2.6);
	\draw[fill=gray, opacity=0.2] (0,0) rectangle (2,2);
	\draw[red, very thick] (0,0) --(2,0); 
	
	\draw[blue, very thick] (0,2) --(2,2);
	\draw[green, very thick] (0,0) --(0,2);  
	
	\draw[green, very thick] (2,0) --(2,2); 
	
		\draw [very thick, blue]  plot[smooth,dashed] coordinates {(5,0) (6,-0.5)  (8,0.4) (6,2.3) (4.1,1.5)  (5,0) };
			\draw[green, very thick] (5.9,0.8)--(5,0);
		\draw[fill=red] (5.9,0.8) circle(2pt);
		
	\draw[->, very thick, blue] (0.98,2) to (1.02,2); 
	\draw[->, very thick, red] (0.98,0) to (1.02,0); 
	
	\draw[->, very thick, green] (2,0.98) to (2,1.02); 
	
	\draw[->, very thick, green] (0,0.98) to (0,1.02); 
	
	\draw[->, very thick, green] (2,0.98) to (2,1.02); 
	
	\draw[->, very thick, green] (5.495,0.44) to (5.405,0.36); 
	
	\draw[->, very thick, blue] (5,1.98+0.15) to (5.1,2.02+0.15); 
	\draw[->, very thick, blue] (6.8,1.4-1.5-0.2) to (6.7,1.3-1.435-0.2); 
	
	\node[below] at (6.1,0.75) {$\f{x}_0$};
	\node[below] at (2.35,-0.05) {$\zeta$};
	\node[left] at (-0.05,2.35) {$\xi$};
	
	\node[below] at (0,-0.08) {$0$};
		\node[left] at (-0.08,0) {$0$};
		\node[left] at (-0.08,2) {$1$};
		\node[below] at (2,-0.08) {$1$};
		\draw (-0.11,0) -- (0,0);
			\draw (0,-0.11) -- (0,0);
		\draw (-0.11,2) -- (0,2);
			\draw (2,-0.11) -- (2,0);
			\draw[->,thick, out=40,in=140] (2.5,2.5) to (4,2.5);
			
			\node[below] at (3.25,2.6) {$\p{F}$};
			\node[above, blue, thick] at (7.2,1.4) {\large $\gamma$};
			
			\node at (1,1) {\large $\widehat{\Omega}$};
				\node at (7.4,0.55) {\large ${\Omega}$};
				
				\draw[->,thick] (9,0) --(10,0);
				\draw[->,thick] (9,0) --(9,1);
				\node at (9.8,-0.2) {$x$};
				\node at (8.75,0.8) {$y$};
	\end{tikzpicture}
	\caption{ \small Within SB-IGA a boundary description is used. However, a suitable scaling center is required. }
 \label{Fig:SB_illustration_1}
\end{figure}
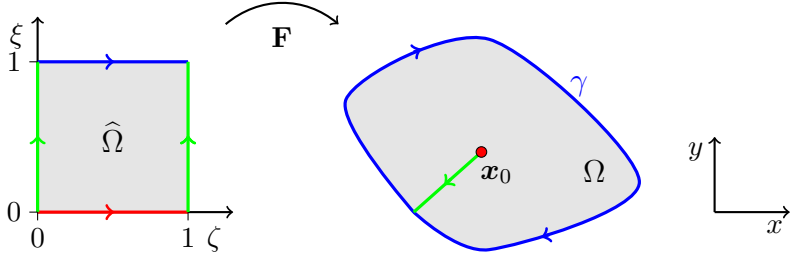

Depending on the situation it might be useful to allow for more flexibility and we replace in this article the prefactor $\xi$ by a degree $1$ polynomial  of the form $ q(\xi) \coloneqq c_1 \xi + c_2, \ c_1 >0, \ c_2\geq 0 $, i.e.  \begin{equation} \label{eq:SB-IGA_param}
    \p{F}(\zeta,\xi) = q(\xi)  \big( \gamma(\zeta)-\f{x}_0\big)+\f{x}_0.
\end{equation} By the linearity w.r.t. the second parameter $\xi$  we can assume  for $\o \subset \mathbb{R}^2$ that  $\p{F} \in \big[ {N}_p^r \otimes S_{p}^r  \big]^2$.  In particular, the weight function depends only on $\zeta$.

  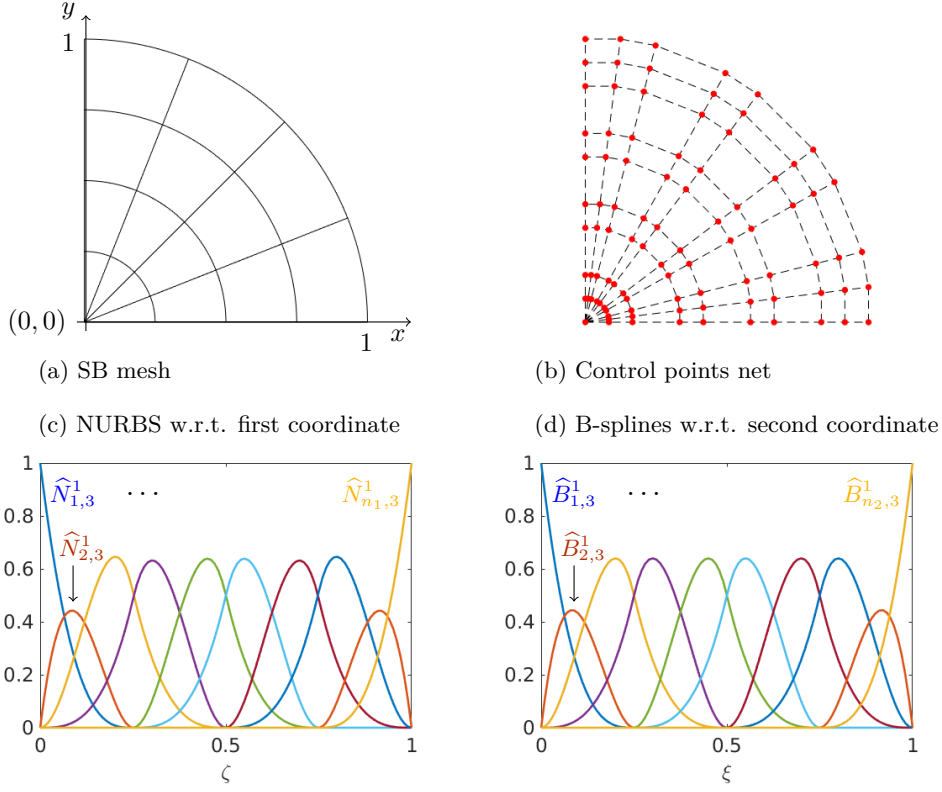
\begin{figure}[H]	
  \centering
	\begin{tikzpicture}[scale=1.2]    
	    \node (eins) at (0.3,0.7) {\includegraphics[width=0.42\linewidth]{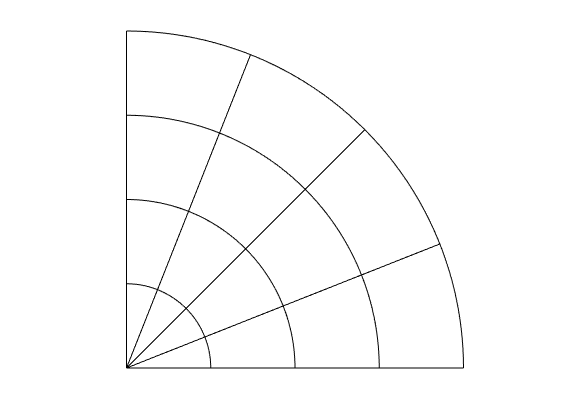}};
         \node (zwei) at (-0.45-0.3+6.6,0.7) {\includegraphics[width=0.42\linewidth]{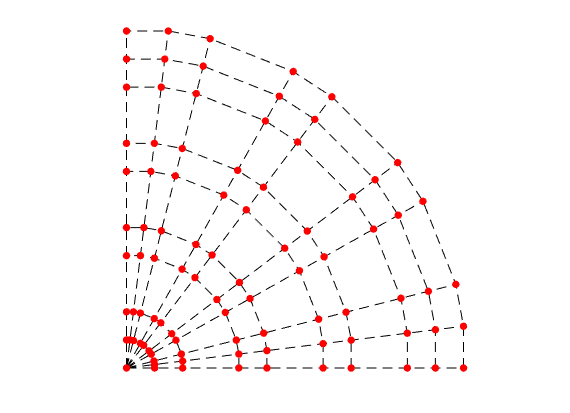}};
         \node[right] at (-1.8,-1.35) {\footnotesize (a) SB mesh};
          \node[right] at (-1.8,-1.95) {\footnotesize (c) NURBS w.r.t. first coordinate};
               \node[right] at (-1.5+5.2,-1.35) {\footnotesize (b) Control points net};
          \node[right] at (-1.5+5.2,-1.95) {\footnotesize (d) B-splines w.r.t. second coordinate};
        \node (eins) at (0.3,-4) {\includegraphics[width=0.42\linewidth]{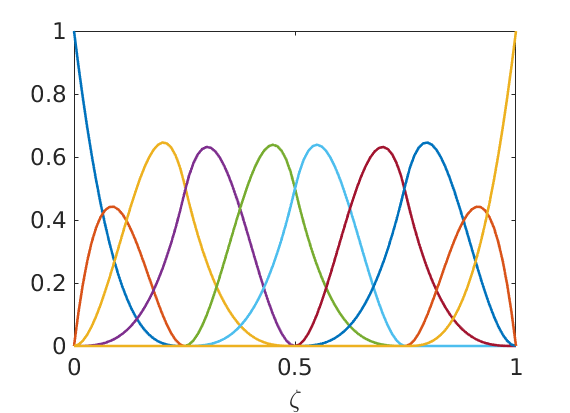}};
         \node (eins) at (0.3-0.45+6,-4) {\includegraphics[width=0.42\linewidth]{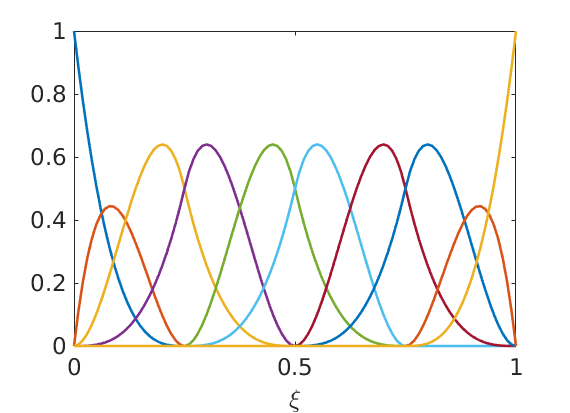}};
         \node[blue] at (-1.3,-2.7) {\footnotesize $\widehat{N}_{1,3}^1$}; 
          \node[orange!50!red!70!black] at (-1.2,-3.3) {\footnotesize $\widehat{N}_{2,3}^1$};
          \draw[->] (-1.3,-3.5) -- (-1.3,-3.9); 
          \node[orange!60!yellow] at (2,-2.7) {\footnotesize $\widehat{N}_{n_1,3}^1$}; 
          \node at (-0.5,-2.7) { $\dots$}; 

           \node[blue] at (-1.3+5.55,-2.7) {\footnotesize $\widehat{B}_{1,3}^1$}; 
          \node[orange!50!red!70!black] at (-1.2+5.55,-3.3) {\footnotesize $\widehat{B}_{2,3}^1$};
          \draw[->] (-1.3+5.55,-3.5) -- (-1.3+5.55,-3.9); 
          \node[orange!60!yellow] at (2+5.55,-2.7) {\footnotesize $\widehat{B}_{n_2,3}^1$}; 
          \node at (-0.5+5.55,-2.7) { $\dots$}; 
          \draw[->,shift={(-1.158,-1)}] (0,0.1) -- (0,3.6);
          \draw[->,shift={(-1.158,-0.8)}] (-0.1,0) -- (3.6,0);
          \node at (-1.7,-0.8) {\small $(0,0)$};
            \node at (-1.35,2.65) {\small $y$};
             \node at (2.3,-0.95) {\small $x$};
             \node at (1.95,-1.03) {\small $1$};
               \node at (-1.35,2.3) {\small $1$};
	\end{tikzpicture}
\caption{\small The mesh and corresponding control net for a simple SB parametrization. Here we have $p=3, r=1$ for the underlying NURBS definition. One notes the usage of B-splines without weight function for the scaling direction $\xi$. }\label{Fig:Controls}
\end{figure}
  In other words, there are so-called control points $\p{C}_{i,j} \in  \mathbb{R}^2$ associated to the NURBS $(\zeta,\xi) \mapsto \widehat{N}_{i,p}^r(\zeta) \widehat{B}_{j,p}^r(\xi) \in N_p^r \otimes S_p^r$ which define $\p{F}$, namely $$\p{F}(\zeta,\xi) = \sum_{i=1}^{n_1} \sum_{j=1}^{n_2}  \p{C}_{i,j} \  \widehat{N}_{i,p}^r(\zeta) \widehat{B}_{j,p}^r(\xi) . $$ For reasons of simplification we  suppose equal degree and regularity w.r.t. each coordinate direction. Due to the SB ansatz we obtain in the physical domain $\o$ layers of control points and it is $\p{C}_{1,1}=\p{C}_{2,1}= \dots = \p{C}_{n_1,1}$; cf. Fig. \ref{Fig:Controls}.   
  \begin{remark}
  The structure of the domain and the control points in Fig. \ref{Fig:Controls} are reminiscent of the mentioned \emph{polar splines} framework; see \cite{toshniwal2017}.
      In fact, the idea of degenerating an edge of the standard parametric domains is analogous to the SB-IGA ansatz followed in this article. In some sense, one can interpret the polar spline approach as a special case of a scaled boundary representation. But we emphasize that there are also several differences between the mentioned polar spline approach and the content of this publication. First of all, we do not work with periodic spline functions and the computational domain boundary domain can be  a non-smooth curve. Secondly, our  treatment of the scaling center basis functions differs, namely, we use always the original B-spline functions from the non-degenerate parametric domain, whereas in \cite{toshniwal2017} triangular Bernstein polynomials are considered. In particular, we do not construct polar spline basis functions. Besides, in the subsequent parts we look at the coupling of different SB-parametrizations as well as at the coupling of different star-shaped blocks. This   leads in general  to  geometries that can not be treated directly with polar splines. 
  \end{remark}
   The isogeometric test functions, starting point for discretization methods,  are then defined as the push-forwards of the NURBS, namely $$\mathcal{V}_h= \mathcal{V}_h(r,p) \coloneqq \{ \phi \ | \ {\phi} \circ \p{F}  \in N_{p}^r \otimes {S}_p^r  \}.$$
 If the domain boundary $\partial\o$ is composed of different curves $\gamma^{(m)}$, one defines  parametrizations for each curve as written above and we are in the field of multipatch geometries; cf. Fig. \ref{Fig:SB_illustration_2}. To be more precise, for a $n$-patch geometry we have 	
\begin{align}
\bigcup_{m=1,\dots,n} \overline{\o_{m}} = \overline{\Omega}, \ \ \o_k \cap \o_l = \emptyset \ \textup{if} \ k \neq l , \  \ \ \p{F}_m \colon  \widehat{\Omega} \rightarrow \o_m, \ \p{F}_m  \in \big[N_{p}^r \otimes {S}_p^r \big]^2 \  
\end{align} and $\p{F}_m$ is defined analogous to  \eqref{eq:SB-IGA_param}. Unfortunately,  especially if the curves meet not $G^1$ but high regularity of the corresponding IGA test functions is required, then the coupling gets involved. IGA spaces in the multipatch framework are straight-forwardly defined as $$\mathcal{V}^M_{h} \coloneqq \{ \phi  \colon \o \rightarrow \mathbb{R}\ | \phi_{|\o_m} \in \mathcal{V}_h^{(m)}, \ \forall m \},$$  where $\mathcal{V}_h^{(m)}$ denotes the IGA space  corresponding to the $m$-th patch, to $\p{F}_m$, respectively. For all the patch coupling considerations we suppose the next assumption.
\begin{assumption}[Regular patch coupling] \color{white}{sdakjhjkh} 
\label{Assumption:coupling}\color{black}
    \begin{itemize}
        \item In each patch we use  NURBS and B-splines with the same degree $p$ and regularity $r$.
        \item The control points at interfaces match, meaning the control points of meeting patches coincide at the respective interface.
    \end{itemize}
\end{assumption}
Thus it is justified to write for the set of parametric basis functions in the $m$-th patch $$\{ \widehat{N}_{i,p}^r \cdot \widehat{B}_{j,p}^r \ | \ 1 \leq i \leq n_1^{(m)}, \ 1 \leq j \leq n_2\},$$ for proper $n^{(m)}_1, \ n_2 \in \mathbb{N}_{>1}.$\\
The main part  of the article is dedicated to the $C^1$ coupling of such SB-IGA patches, i.e. face spaces of the form
\begin{equation}
\mathcal{V}_h^{M,1} \coloneqq \mathcal{V}_h^{M} \cap C^1(\overline{\o}),
\end{equation} 
where the singularity of the $\p{F}_i$ at $\f{x}_0$ for $c_2=0$ requires a special attention.

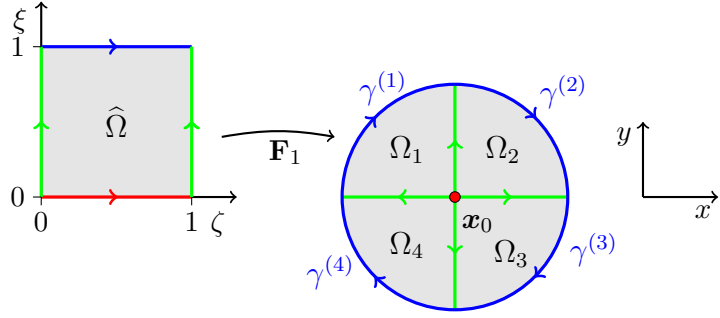
\begin{figure}
	\centering
	\begin{tikzpicture}
%	Draw the sections
	\filldraw[fill opacity=0.2,fill=gray] (-1.5,0) -- (0,0) arc (0:90:1.5) -- cycle;

\filldraw[fill opacity=0.2,fill=gray] (-1.5,0) -- (-1.5,1.5) arc (90:180:1.5) -- cycle;

\filldraw[fill opacity=0.2,fill=gray] (-1.5,0) -- (-3,0) arc (180:270:1.5) -- cycle;

\filldraw[fill opacity=0.2,fill=gray] (-1.5,0) -- (-1.5,-1.5) arc (270:360:1.5) -- cycle;

	\draw[very thick, green] (-1.5,0) -- (-3,0);
	\draw[very thick, green] (-1.5,0) -- (-1.5,1.5); 
	\draw[very thick, green] (-1.5,0) -- (-1.5,-1.5); 
	\draw[very thick, green] (-1.5,0) -- (0,0); 
	
	\draw[very thick, green,<-] (-0.74,0) -- (-0.76,0);
	\draw[very thick, green,->] (-1.5,0.74) -- (-1.5,0.76); 
	\draw[very thick, green,->] (-1.5,-0.74) -- (-1.5,-0.76); 
	\draw[very thick, green, <-] (-0.76-1.5,0) -- (-0.74-1.5,0);

	\draw[very thick, blue] (0,0) arc (0:90:1.5) ;
	\draw[very thick, blue] (-1.5,1.5) arc (90:180:1.5) ;
	\draw[very thick, blue] (-3,0) arc (180:270:1.5) ;
	\draw[very thick, blue] (-1.5,-1.5) arc (270:360:1.5) ;
 
	\draw[fill=red] (-1.5,0) circle (2pt);

  % Parametric domain 
  	\draw[->, thick] (0-7,0) to (2.6-7,0);
  \draw[->, thick] (0-7,0) to (0-7,2.6);
  	\draw[fill=gray, opacity=0.2] (0-7,0) rectangle (2-7,2);
  \draw[red, very thick] (0-7,0) --(2-7,0); 
  
  \draw[blue, very thick] (0-7,2) --(-7+2,2);
  \draw[green, very thick] (0-7,0) --(0-7,2);  
  
  \draw[green, very thick] (2-7,0) --(2-7,2);

  	\node[below] at (6.1-7.3,-0.1) {$\f{x}_0$};
  \node[below] at (2.35-7,-0.05) {$\zeta$};
  \node[left] at (-0.05-7,2.35) {$\xi$};
  
  \node[below] at (0-7,-0.08) {$0$};
  \node[left] at (-0.08-7,0) {$0$};
  \node[left] at (-0.08-7,2) {$1$};
  \node[below] at (2-7,-0.08) {$1$};
  \draw (-0.11-7,0) -- (0-7,0);
  \draw (0-7,-0.11) -- (0-7,0);
  \draw (-0.11-7,2) -- (0-7,2);
  \draw (2-7,-0.11) -- (2-7,0);

  \draw[->, very thick, blue] (0.98-7,2) to (1.02-7,2); 
  \draw[->, very thick, red] (0.98-7,0) to (1.02-7,0); 
  
  \draw[->, very thick, green] (2-7,0.98) to (2-7,1.02); 
  
  \draw[->, very thick, green] (0-7,0.98) to (0-7,1.02); 
  
  \draw[->, very thick, green] (2-7,0.98) to (2-7,1.02); 
  	\node at (1-7,1) {\large $\widehat{\Omega}$};
  	
  		\node at (-1.5-0.85*0.75,0.85*0.75) {\large ${\Omega}_1$};

	\node at (-1.5+0.85*0.75,0.85*0.75) {\large ${\Omega}_2$};
	
		\node at (-1.5+1*0.75,-1*0.75) {\large ${\Omega}_3$};
		
		\node at (-1.5-0.85*0.75,-0.85*0.75) {\large ${\Omega}_4$};

	% boundary curves	
		\draw[->, very thick, blue] (-2.25-0.31-0.01,0.75+0.31-0.01) to (-2.25-0.31+0.01,0.75+0.31+0.01);

			\draw[->, very thick, blue] (-2.25-0.31-0.01+2.13,0.75+0.31+0.01) to (-2.25-0.31+0.01+2.13,0.75+0.31-0.01);

			\draw[->, very thick, blue] (-2.25-0.31+0.01,0.75+0.31-0.01-2.13) to (-2.25-0.31-0.01,0.75+0.31+0.01-2.13); 
			
				\draw[->, very thick, blue] (-2.25-0.31+0.01+2.13,0.75+0.31+0.01-2.12) to (-2.25-0.31-0.01+2.13,0.75+0.31-0.01-2.12); 
				
	\node[left, blue] at (-2,1.45) {\large $\gamma^{(1)}$};
	\node[left, blue] at (0.4,1.4) {\large $\gamma^{(2)}$};
	\node[left, blue] at (0.8,-0.7) {\large $\gamma^{(3)}$};
	\node[left, blue] at (-2.7,-1) {\large $\gamma^{(4)}$};
	
		\draw[->,thick, out=10,in=170] (2.5-7.1,2.5-1.7) to (4-7.1,2.5-1.7);
		
		\node[below] at (3.25-7,2.6-1.7) {$\p{F}_1$};

  \draw[->,thick,shift={(-8,0)}] (9,0) --(10,0);
 \draw[->,thick,shift={(-8,0)}] (9,0) --(9,1);
\node[shift={(-8,0)}] at (9.8,-0.2) {$x$};
\node[shift={(-8,0)}] at (8.75,0.8) {$y$};
	\end{tikzpicture}
	\caption{ \small The boundary can be determined by the concatenation of several curves. In this situation the patch-wise defined discrete function spaces have to be coupled in order to obtain the wanted global smoothness.}
 \label{Fig:SB_illustration_2}
\end{figure}

\section{Classical planar two-patch coupling}
\label{section:two-patch coupling}
For reasons of simplification, we restrict ourselves until Sect. \ref{section:generelizations} to the situation of two-patch geometries. The aspects of coupling can be straightforwardly generalized to  three or more patches. 
 %\subsection{Classical planar $C^1$ two-patch coupling}
First, before we come back to SB-IGA we	turn towards  the simple case of a classical two-patch parametrization, meaning there are no singular points for the $\p{F}_m$. For such a situation we explain the needed conditions for a $C^1$ coupling and exploit the notion of analysis-suitable $G^1$ planar parametrizations. Here we mainly focus on the results and the framework of \cite{Collin2016AnalysissuitableGM}.\\
The setting is now the following. Assume that we have two mappings $\p{F}^{(S)},$ $ \ S \in \{L,R \}$ corresponding to the left and right side of a two-patch situation as displayed in Fig. \ref{Fig:1}.

	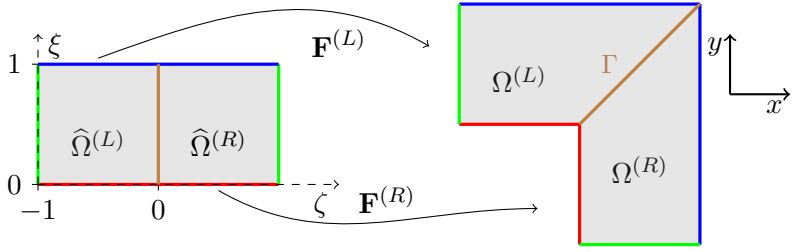
\begin{figure}[h!]
	\centering
		\begin{tikzpicture}[scale = 0.8]
		\draw[fill = gray , opacity = 0.2] (-5,-1.5) -- (-3,-1.5) -- (-3,0.5) -- (-5,0.5) -- (-5,-1.5);
		\draw[very thick,red] (-5,-1.5) -- (-3,-1.5); 
		\draw[very thick,green] (-3,-1.5) -- (-3,0.5); 
		
		\draw[very thick,blue] (-3,0.5) -- (-5,0.5);			
		\draw[->, out = 20, in = 150]  (-6,0.6) to (-0.5,0.8);
		\draw[->, out = -30, in = -180]  (-4,-1.6) to (1.3,-1.9);
		\node[above] at (-2,0.45) { ${\p{F}}^{(L)}$};
		\node[above] at (-1.2,-2.15) { $\p{F}^{(R)}$};
		\node at (-6,-0.8) { $\widehat{\Omega}^{(L)}$};
		\node at (-4,-0.8) { $\widehat{\Omega}^{(R)}$};
		\node at (1,0.25) { $\Omega^{(L)}$};
		\node at (3,-1.25) { $\Omega^{(R)}$};
		
		\draw[fill = gray , opacity = 0.2] (-7,-1.5) -- (-5,-1.5) -- (-5,0.5) -- (-7,0.5) -- (-7,-1.5);
		\draw[very thick,red] (-7,-1.5) -- (-5,-1.5); 
	%	\draw[very thick,green] (-5,-1.5) -- (-5,0.5); 
		\draw[very thick,green] (-7,0.5) -- (-7,-1.5); 
		\draw[very thick,blue] (-5,0.5) -- (-7,0.5);
		
		\draw[dashed,->] (-7,-1.5) to (-7,1);
		\draw[dashed,->] (-7,-1.5) to (-2,-1.5);
		\draw (-7,-1.5) -- (-7,-1.6);
		\draw (-5,-1.5) -- (-5,-1.6);
		\draw (-7,-1.5) -- (-7.1,-1.5);
		\draw (-7,0.5) -- (-7.1,0.5);
		\node[below] at (-7,-1.6) {  $-1$};
		\node[below] at (-5,-1.6) {  $0$};
		\node[left] at (-7.1,-1.5) {  $0$};
		\node[left] at (-7.1,0.5) {  $1$};
		\node[right] at (-7,0.8) {  $\xi$};
		\node[below] at (-2.3,-1.45) {  $\zeta$};
		\node[very thick,brown] at (2.5,0.5) { \small $\Gamma$};
		
		\draw[fill = gray, opacity = 0.2] (-0,1.5) -- (4,1.5) -- (2,-0.5) -- (0,-0.5)--(-0,1.5);
		\draw[fill = gray, opacity = 0.2] (2,-0.5) -- (4,1.5) -- (4,-2.5) -- (2,-2.5)--(2,-0.5);
		\draw[very thick,blue] (-0,1.5) -- (4,1.5);
		\draw[very thick,blue] (4,1.5) -- (4,-2.5);
		\draw[very thick,red] (-0,-0.5) -- (2,-0.5);
		\draw[very thick,red] (2,-0.5) -- (2,-2.5);
		\draw[very thick,green] (0,-0.5) -- (0,1.5);
		\draw[very thick,green] (2,-2.5) -- (4,-2.5);
		\draw[very thick,brown] (2,-0.5) -- (4,1.5);
		\draw[very thick,brown] (-5,0.5) -- (-5,-1.5);
    \draw[->,thick,shift={(-4.5,0)}] (9,0) --(10,0);
 \draw[->,thick,shift={(-4.5,0)}] (9,0) --(9,1);
\node[shift={(-3.65,0)}] at (9.8,-0.2) {$x$};
\node[shift={(-3.6,0)}] at (8.75,0.8) {$y$};
		\end{tikzpicture}
	%	\caption{test}
	\caption{ \small A planar non-degenerate two-patch domain.}
		\label{Fig:1}
\end{figure}

In more detail,  we have 
\begin{align*}
\p{F} &\in C^0(\overline{\tilde{\Omega}}; \mathbb{R}^2), \ \overline{\tilde{\Omega}} \coloneqq \overline{\widehat{\Omega}^{(L)} \cup \widehat{\Omega}^{(R)}}, \ \  \textup{as global parametrization, i.e.}  \\ \p{F}_{|\widehat{\Omega}^{(S)}}&= \p{F}^{(S)} \colon \widehat{\Omega}^{(S)} \rightarrow {\Omega}^{(S)}, \ \ \textup{with} \ \
 \p{F}^{(S)} \in C^1(\overline{\widehat{\Omega}^{(S)}}; \mathbb{R}^2) \ \
 \textup{and}\\ \widehat{\Omega}^{(L)} &= (-1,0) \times (0,1), \ \ \ \widehat{\Omega}^{(R)} = (0,1) \times (0,1).
\end{align*}

One notes the colors in the mentioned Fig.  \ref{Fig:1} that should determine the   parametrization orientation used here and in the following. Further we assume that the $\p{F}^{(S)}$ are diffeomorphisms and write $\widehat{\Gamma} = \p{F}^{-1}({\Gamma}) \coloneqq \{ (0,\xi) \ | \xi \in [0,1]  \}$ for the interface. And, using the notation from the previous section, we can write $\p{F}_1 \coloneqq \p{F}^{(L)}(\cdot -1,\cdot), \ $%$\in \big[N_{p}^r \otimes {S}_p^r \big]^2$, \  
$\p{F}_2 \coloneqq \p{F}^{(R)}$, i.e. there is a shift in the parametric domain.
In view of isogeometric analysis we suppose  two functions $g^{(L)} \in C^1(\o^{(L)}),$ $g^{(R)} \in C^1(\o^{(R)})$ that can be combined to a continuous mapping $g \colon \o \rightarrow \mathbb{R}$ defined by $g_{|\o^{(S)}} = g^{(S)}$ and we want to know which properties of the $g^{(S)}$ lead to a $C^1$-regular $g$. The latter is the case if and only if the graph $ \mathcal{G} \coloneqq \{ (\p{F}(\f{\zeta}),g( \p{F}(\f{\zeta})) ) \ | \ \f{\zeta} \in \tilde{\Omega}  \} $ as a surface in $3D$ has  well-defined tangent planes along the patch interface. This condition can be formulated by means of auxiliary coefficient functions $\alpha^{(S)}, \ \beta $ as written down in the next lemma which uses Proposition 2 and  Definition 2 of \cite{Collin2016AnalysissuitableGM}.    
		\label{lemma:G1_cond}
		
\begin{lemma}[{$C^1$ regularity}]
\label{Lemma:C1_condition}
				Let $\hat{g} \colon \tilde{\Omega} \rightarrow \mathbb{R}$ be a continuous function with $\hat{g}^{(S)}  \coloneqq \hat{g}_{| \widehat{\Omega}^{(S)}} \in C^1(\overline{\widehat{\Omega}^{(S)}})$
	        	and let $g \coloneqq \hat{g} \circ \p{F}^{-1}  \colon  \Omega \rightarrow \mathbb{R}$. \\			
		    	Then  $ g \in C^1(\Omega)$ \  if and only if there exist mappings
		    	\  $   \alpha^{(L)}, \ \alpha^{(R)} , \ \beta \colon [0,1] \rightarrow \mathbb{R}$ \ s.t.  $\forall \xi \in [0,1]$:
			\begin{align}
			&\alpha^{(L)}(\xi) \, \alpha^{(R)}(\xi) >0  \ \ \   \textup{and}  \nonumber \\
			\alpha^{(R)}(\xi) \, \begin{bmatrix}
			\partial_{\zeta}\p{F}^{(L)}(0,\xi) \\ \label{C1condition_eq_2}\partial_{\zeta}\hat{g}^{(L)}(0,\xi)
			\end{bmatrix}-
			& \alpha^{(L)}(\xi) \, \begin{bmatrix}
			\partial_{\zeta}\p{F}^{(R)}(0,\xi) \\\partial_{\zeta}\hat{g}^{(R)}(0,\xi)
			\end{bmatrix}+
			\beta(\xi) \, \begin{bmatrix}
			\partial_{\xi}\p{F}^{(R)}(0,\xi) \\\partial_{\xi}\hat{g}^{(R)}(0,\xi)
			\end{bmatrix}= \p{0}.
			\end{align}
\end{lemma}
\begin{proof}
		This follows directly by Definition 2 and Proposition 2 in \cite{Collin2016AnalysissuitableGM}.
\end{proof}

Note that in \eqref{C1condition_eq_2} the last term of the sum on the left side does not change if  $R$ is replaced by $L$, due to the assumed continuity of $\p{F}$ and $g$. In the  IGA context, one considers  cases like $\hat{g}^{(L)}(\cdot -1,\cdot) \in N_{p,p}^{r,r} $ and $\hat{g}^{(R)} \in N_{p,p}^{r,r}$. % Although we could choose also NURBS basis functions we first concentrate on the case of constant weight function, i.e. pure B-spline bases.% Besides, we assume the parametrization to be a B-spline parametrization, too. 
          \    More precisely,  the standard two-patch isogeometric spaces without coupling conditions are defined in the parametric domain through
$$\tilde{\mathcal{V}}_h^M \coloneqq \{  \hat{\phi} \colon \tilde{\Omega} \rightarrow \mathbb{R} \ | \ \hat{\phi}_{|\widehat{\Omega}^{(L)}}(\cdot-1,\cdot) \in N_{p,p}^{r,r}, \ \  \hat{\phi}_{|\widehat{\Omega}^{(R)}} \in N_{p,p}^{r,r} \}$$
and in the physical domain ${\mathcal{V}}_h^M  \coloneqq \tilde{\mathcal{V}}_h^M  \circ \p{F}^{-1}$. %Later we will come back to the NURBS basis functions related to  SB-IGA.

 Although there is  a clear criterion for $C^1$ regularity, the actual calculation of test functions meeting the conditions within the scope of numerical methods is in general not trivial. Especially if test functions are defined separately for each patch, like for the isogeometric spaces, and a suitable global $C^1$-smooth linear combination  of the test functions is  sought one can observe bad approximations properties and a loss of convergence in different situations. This is a reason why in \cite{Collin2016AnalysissuitableGM} a class a special parametrizations are introduced leading to optimal convergence of the isogeometric spline spaces under $h$-refinement, namely the so-called analysis-suitable $G^1$ parametrizations.

\begin{definition}[Analysis-suitable $G^1$ parameterizations; see \cite{Collin2016AnalysissuitableGM}]
	\label{Def: ASG1}
	Assume there are polynomial functions $\alpha^{(S)}, \ \beta^{(S)} \colon [0,1] \rightarrow \mathbb{R}$ of degree at most $1$ s.t.
	\begin{align}
	\alpha^{(R)}(\xi) \, \partial_{\zeta}\p{F}^{(L)}(0,\xi)- \alpha^{(L)}(\xi) 	\, \partial_{\zeta}\p{F}^{(R)}(0,\xi)+ \beta(\xi) \, 	\partial_{\xi}\p{F}^{(R)}(0,\xi) = \p{0}, \ \forall \xi,
	\end{align}
	with 
 $ \beta = \alpha^{(L)} \, \beta^{(R)} - \alpha^{(R)} \, \beta^{(L)}$. \\
 Then the parametrization $\p{F}$  is called analysis-suitable $G^1$.
\end{definition}

As already mentioned, the function spaces of interest are $C^1$-regular spline spaces
\begin{align}
\label{eq:definition_coupled_multipatch}
{\mathcal{V}}_h^{M,1} \coloneqq {\mathcal{V}}_h^M \cap C^1(\o).
\end{align}	

\begin{lemma}
	\label{Lemma:AsG1_convergence_prop}
	Let the parametrization be AS-$G^1$ and let $p>r+1>1$ for the underlying B-splines. Further, assume the patch coupling to be regular in the sense of Assumption \ref{Assumption:coupling} and let us assume only B-spline basis functions in the parametric domain, i.e. constant weight functions $W$.\\ Then in numerical applications the asymptotic convergence behaviour of the coupled spaces $\mathcal{V}_h^{M,1}$ conforms to the  optimal approximation rates. 
	In other words, AS-$G^1$ parametrizations do not suffer from order reductions,  $C^1$ locking, respectively. 
\end{lemma} 
\begin{proof}
	Compare Theorem 1 in \cite{Collin2016AnalysissuitableGM}.
\end{proof}

An important point is the appearance of $C^1$ locking if we choose $p=r+1$ even if $\p{F}$ is AS-$G^1$ as shown by Theorem 2 in \cite{Collin2016AnalysissuitableGM}.\\
After the consideration of the regular two-patch case in $2D$ we face now the coupling problem for SB-parametrizations.

\section{Planar SB-IGA with $C^1$ coupling}
\label{sec:Planar SB-IGA}
Here we show that planar SB-IGA parametrizations are quasi analysis-suitable, except at the scaling center. This is important to obtain good convergence properties for the $C^1$-coupled test functions. The problem with the singular point is addressed in the subsequent second subsection.

\subsection{SB-parametrization as quasi AS-$G^1$}
Analogous to the classical planar two-patch case we first look at two scaled boundary patches as displayed in Fig. \ref{Fig:2}. This means, we have constants $c_1 >0, \ c_2\geq 0$ with two boundary curves $\gamma^{(S)}$ and a common scaling center $\f{x}_0$ and the patch parametrizations are of the form
\begin{equation}
\label{SB_IGA_param_2_patch}
\p{F}^{(S)}(\zeta, \xi) =   (c_1 \xi  + c_2) \ (\gamma^{(S)}(\zeta) - \f{x}_0)+ \f{x}_0.
\end{equation}

\begin{assumption}
	\label{assumption:SB-IGA_2_patch}
	Here and in the rest of the article we assume that the boundary curves $\gamma^{(S)}$ are parametrized in a strong $C^1$ sense and it is $$\f{0} \neq \partial_{\zeta} \gamma^{(L)}(0), \ \ \f{0} \neq \partial_{\zeta} \gamma^{(R)}(0), \ \ \gamma^{(L)}(0)=\gamma^{(R)}(0) .$$% \ \partial_{\zeta} \gamma^{(R)}(0) \nparallel (\gamma^{(R)}(0)-\f{x}_0), \ \partial_{\zeta} \gamma^{(R)}(0) \neq, \ \gamma^{(L)}(0)=\gamma^{(R)}(0) .$$  
	Moreover, we assume that the boundary curves are chosen in such a way that for every $\delta >0$, it is $\p{F}^{(S)}_{|\widehat{\Omega}^{(S)} \cap \tilde{\Omega}_{\delta}} \in C^1(\overline{\widehat{\Omega}^{(S)} \cap \tilde{\Omega}_{\delta}}), \ \tilde{\Omega}_{\delta} \coloneqq \tilde{\Omega} \textbackslash \{ (\zeta,\xi) \ | \ \zeta \in [-1,1], \ \xi \in [0,\delta] \}$. Further, the restriction $\p{F}^{(S)}_{|\widehat{\Omega}^{(S)} \cap \tilde{\Omega}_{\delta}}$ defines a diffeomorphism.
\end{assumption}

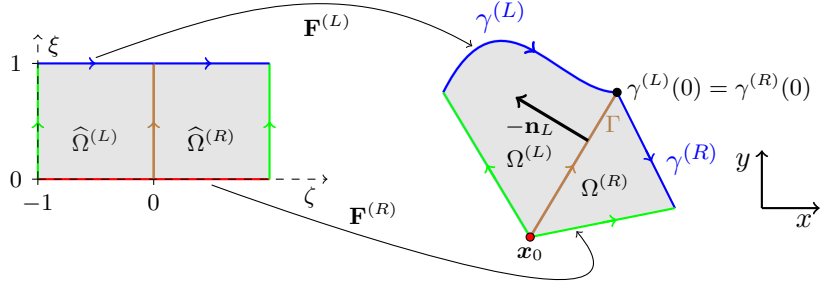
\begin{figure}
	\centering
	\begin{tikzpicture}[scale = 0.77]
	\draw[very thick,blue, fill = gray , opacity = 0.2]  (0,0) .. controls(1,2) and (2,0) .. (3,0) to (1.5,-2.5) to (0,0);
	\draw[blue,thick]  (0,0) .. controls(1,2) and (2,0) .. (3,0) ;
	\draw[->, very thick, blue] (1.4+0.18,0.5+0.185+0.015) -- (1.41+0.18,0.49+0.187+0.015);
	\draw[thick,green] (0,0) to (1.5,-2.5);
	\draw[thick,green,<-] (0.75,-1.25) to (0.75+0.15,-1.25-0.25);
	\draw[thick,brown] (3,0) to (1.5,-2.5);
	\node[above, blue] at (1,0.85) {$\gamma^{(L)}$};  
	\node[above, blue] at (4.3,-1.6) {$\gamma^{(R)}$};
	\draw[fill = gray , opacity = 0.2] (-5,-1.5) -- (-3,-1.5) -- (-3,0.5) -- (-5,0.5) -- (-5,-1.5);
	\draw[thick,red] (-5,-1.5) -- (-3,-1.5); 
	\draw[green,thick] (-3,-1.5) -- (-3,0.5); 
	\draw[brown,thick ] (-5,0.5) -- (-5,-1.5); 
	\draw[blue, thick ] (-3,0.5) -- (-5,0.5);
	
	\draw[blue, thick,->] (-4-0.1,0.5) -- (-4,0.5);
	\draw[blue, thick,->] (-6-0.1,0.5) -- (-6,0.5);
	\draw[green, thick,->] (-7,-0.6) -- (-7,-0.5);
	\draw[brown, thick,->] (-5,-0.6) -- (-5,-0.5);
	\draw[green, thick,->] (-3,-0.6) -- (-3,-0.5);
	
	\draw[->, out = 20, in = 150]  (-6,0.6) to (0.5,0.8);
	\draw[->, out = -20, in = -50]  (-4,-1.6) to (2.3,-2.4);
	\node[above] at (-2,0.8) {\footnotesize $\p{F}^{(L)}$};
	\node[above] at (-1.2,-2.45) {\footnotesize $\p{F}^{(R)}$};
	\node at (-6,-0.8) {\footnotesize $\widehat{\Omega}^{(L)}$};
	\node at (-4,-0.8) {\footnotesize $\widehat{\Omega}^{(R)}$};
	\node at (1.5,-1.1) {\footnotesize $\Omega^{(L)}$};
 \node at (1.5,-0.6) {\footnotesize $-\p{n}_L$};
 \node[brown] at (2.93,-0.55) {\footnotesize $\Gamma$};
	\node[right] at (3,0.1) {\footnotesize $\gamma^{(L)}(0)= \gamma^{(R)}(0)$};
	\node[below] at (1.5,-2.5) {\footnotesize $\f{x}_0$};
	
	\draw[fill = gray , opacity = 0.2] (-7,-1.5) -- (-5,-1.5) -- (-5,0.5) -- (-7,0.5) -- (-7,-1.5);
	\draw[thick, red] (-7,-1.5) -- (-5,-1.5); 
	\draw[brown, thick] (-5,-1.5) -- (-5,0.5); 
	\draw[thick, green] (-7,0.5) -- (-7,-1.5); 
	\draw[thick, blue] (-5,0.5) -- (-7,0.5);
	
	\draw[blue, fill = gray , opacity = 0.2]  (3,0) -- (4,-2) -- (1.5,-2.5) -- (3,0) ;
	\draw[blue,thick]  (3,0) -- (4,-2) ;
	\draw[blue,thick,->]  (3.5,-1) -- (3.5+0.1,-1-0.2) ;
	\draw[thick, green] (4,-2) -- (1.5,-2.5);
	\draw[thick, green,<-] (2.75+0.25,-2.25+0.05) -- (2.75,-2.25);
	
	\draw[thick, brown]  (1.5,-2.5) -- (3,0);
    \draw[very thick, ->] (2.5,-2.5/3) -- (2.5-5/4,-2.5/3+3/4); 
	\draw[thick, brown,->]  (2.25-0.15,-1.25-0.25) -- (2.25,-1.25);
	\node at (2.8,-1.6) {\footnotesize $\Omega^{(R)}$};
	\draw[dashed,->] (-7,-1.5) to (-7,1);
	\draw[dashed,->] (-7,-1.5) to (-2,-1.5);
	\draw (-7,-1.5) -- (-7,-1.6);
	\draw (-5,-1.5) -- (-5,-1.6);
	\draw (-7,-1.5) -- (-7.1,-1.5);
	\draw (-7,0.5) -- (-7.1,0.5);
	\node[below] at (-7,-1.6) { \footnotesize $-1$};
	\node[below] at (-5,-1.6) { \footnotesize $0$};
	\node[left] at (-7.1,-1.5) { \footnotesize $0$};
	\node[left] at (-7.1,0.5) { \footnotesize $1$};
	\node[right] at (-7,0.8) { \footnotesize $\xi$};
	\node[below] at (-2.3,-1.45) { \footnotesize $\zeta$};
	\draw[fill = red] (1.5,-2.5) circle(2pt); 
	\draw[fill = black] (3,0) circle(2pt); 
   \draw[->,thick,shift={(-3.5,-2)}] (9,0) --(10,0);
 \draw[->,thick,shift={(-3.5,-2)}] (9,0) --(9,1);
\node[shift={(-2.75,-1.55)}] at (9.8,-0.2) {$x$};
\node[shift={(-2.75,-1.55)}] at (8.75,0.8) {$y$};
	\end{tikzpicture}
	\caption{\footnotesize Two-patch situation for SB-IGA.}
	\label{Fig:2}
\end{figure}

\begin{lemma}[{SB-IGA patch coupling as quasi AS-$G^1$}]
       \label{Lemma:quasi_ASG1}
	For $c_2>0$ the two-patch param. \ref{SB_IGA_param_2_patch} is AS-$G^1$.	And for  $c_2=0$ the parametrization is AS-$G^1$ except at the scaling center $\f{x}_0$. In other words,  the condition in  Definition \ref{Def: ASG1} is fulfilled.
\end{lemma}
\begin{proof}
	We show the assertion for the interesting case $c_2=0$. If $c_2>0$ similar proof steps can be used.\\
	First let $\delta >0$ and $\partial_{\zeta} \gamma^{(R)}(0) \nparallel \partial_{\zeta} \gamma^{(L)}(0)$.
	Obviously, $\p{F}_{|\tilde{\Omega}_{\delta}}$ is globally continuous.
	In view of \eqref{SB_IGA_param_2_patch} and Assumption \ref{assumption:SB-IGA_2_patch} there are  $ d_1, \ d_2 \in \mathbb{R} $ with $d_1 \ \partial_{\zeta} \gamma^{(L)}(0) + d_2  \ \partial_{\zeta} \gamma^{(R)}(0) = \gamma^{(L)}(0) -\f{x}_0$. \\
	By the orientation of the parametrization as elucidated in Fig. \ref{Fig:3} it has to be $d_1 \neq 0, \ d_2  \neq 0 $ and $d_1 \, d_2 < 0$. Namely,  for the set  $\mathcal{S}=\{  a_1 \partial_{\zeta} \gamma^{(L)}(0) + a_2 \partial_{\zeta} \gamma^{(R)}(0) \ | \ a_1>0, \, a_2 >0 \}$ we see  $ \gamma^{(L)}(0)-\f{x}_0 \notin \mathcal{S}$ and $ -\gamma^{(L)}(0)+\f{x}_0 \notin \mathcal{S}$. \\
	Set $\alpha^{(R)} \coloneqq d_1 , \ \alpha^{(L)} \coloneqq -d_2  $ and $\beta \coloneqq \frac{c_1 \xi + c_2}{-c_1}$. \\
	Then: 
	\begin{align*}
	\alpha^{(R)} &	\partial_{\zeta}\p{F}^{(L)}(0,\xi) - \alpha^{(L)} 	\partial_{\zeta}\p{F}^{(R)}(0,\xi) + \beta(\xi) \, 	\partial_{\xi}\p{F}^{(L)}(0,\xi) \\
	&= d_1 \ (c_1 \xi + c_2) \ \partial_{\zeta} \gamma^{(L)}(0) + d_2 \ (c_1 \xi + c_2) \  \partial_{\zeta} \gamma^{(R)}(0) + \frac{c_1 \xi + c_2}{-c_1} \, c_1 \, (\gamma^{(L)}(0) - \f{x}_0) \\
	&=  (c_1 \xi + c_2) \ \big[d_1 \ \partial_{\zeta} \gamma^{(L)}(0) + d_2 \   \partial_{\zeta} \gamma^{(R)}(0) -   \, (\gamma^{(L)}(0) - \f{x}_0) \big] =0.
	\end{align*}
	Then setting $\beta^{(L)} \coloneqq \frac{-\beta}{2 \, \alpha^{(R)}}, \ \beta^{(R)} \coloneqq \frac{\beta}{2 \, \alpha^{(L)}}$ we get $\beta = \alpha^{(L)} \beta^{(R)} - \alpha^{(R)} \beta^{(L)}$. This finishes the proof for $\partial_{\zeta} \gamma^{(R)}(0) \nparallel \partial_{\zeta} \gamma^{(L)}(0)$. \\
	Now let $\partial_{\zeta} \gamma^{(R)}(0) = c \  \partial_{\zeta} \gamma^{(L)}(0) $ for $c \neq 0$. It is easy to check that it has to be $c >0$ otherwise we can not have $\p{F}^{(S)}_{|\widehat{\Omega}^{(S)} \cap \tilde{\o}_{\delta}} \in C^1(\overline{\widehat{\Omega}^{(S)} \cap \tilde{\o}_{\delta}})$. Then with  $\alpha^{(R)} \coloneqq c , \ \alpha^{(L)} \coloneqq 1, \ \beta = 0$ one  gets 
	\begin{align*}
	\alpha^{(R)} &	\partial_{\zeta}\p{F}^{(L)}(0,\xi) - \alpha^{(L)} 	\partial_{\zeta}\p{F}^{(R)}(0,\xi) + \beta \, 	\partial_{\xi}\p{F}^{(L)}(0,\xi) \\
	&= c \ (c_1 \xi + c_2) \ \partial_{\zeta} \gamma^{(L)}(0) -  (c_1 \xi + c_2) \  \partial_{\zeta} \gamma^{(R)}(0)  \\
	&=  c \, (c_1 \xi + c_2) \ \big( \partial_{\zeta} \gamma^{(L)}(0) -  \   \partial_{\zeta} \gamma^{(L)}(0) \big) =0.
	\end{align*}
	
\end{proof}

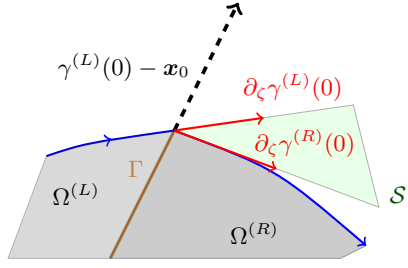
\begin{figure}[h!]
	\centering
	\begin{tikzpicture}[scale=1.7]
	\draw[fill = gray, very thin, opacity = 0.4] (-0.5,-1) to (0,0) .. controls(0.0+0.8,0.0-0.3) .. (1.5,-0.9) to (01.3,-1) to (-0.5,-1);
	
	\draw[blue,thick,->] (0,0) .. controls(0.0+0.8,0.0-0.3) .. (1.5,-0.9);
	\draw[very thick,brown] (0-0.5,-1) -- (0,0);
	
	\draw[blue,thick] (0,0) .. controls(0.0-0.7,0.0-0.1) .. (-1,-0.2);
	\draw[->,line width=0.5mm,dashed] (0,0) --(0.5,1);

 \draw[->,thick,blue,shift={(0,0.13)}]  (-0.5,-0.2) -- (-0.5+0.01,-0.2+0.0025);
	
	%	\draw[very thin,fill=green!30, opacity = 0.4] (0,0) --  (-100:1) arc(-100:50:1) -- cycle;
	\draw[fill=green!30,opacity=0.3] (0,0) -- (2*0.7,2*0.1) -- (2*0.8,-0.3*2) -- (0,0);
	\draw[thick,->,color=red] (0,0) -- (0.7,0.1);
	\draw[thick,->,color=red] (0,0) -- (0.8,-0.3);
	\draw[fill = gray, very thin, opacity = 0.3] (-0.5,-1) to (0,0) .. controls(0.0-0.7,0.0-0.1) .. (-1,-0.2) to (-1.3,-1) to (-0.5,-1);
	\node[left] at (-0.5,-0.5) { \footnotesize $\Omega^{(L)}$};
	\node[left] at (0.9,-0.8) { \footnotesize $\Omega^{(R)}$};
	\node[left,red,thick] at (1.4,0.34) { \footnotesize $\partial_{\zeta}\gamma^{(L)}(0)$};
	\node[left,red,thick] at (1.5,-0.1) { \footnotesize $\partial_{\zeta}\gamma^{(R)}(0)$};
	\node[left,green!30!black,thick] at (1.9,-0.5) { \footnotesize $\mathcal{S}$};
	\node[left] at (0.2,0.5) { \footnotesize $\gamma^{(L)}(0)-\f{x}_0$};
  \node[brown] at (-0.3,-0.3) {\footnotesize $\Gamma$};
	\end{tikzpicture}
	\caption{\footnotesize The orientation of the parametrization leads to $d_1\,d_2 <0$.  }
	\label{Fig:3}
\end{figure}

 Lemma \ref{Lemma:AsG1_convergence_prop} and Lemma \ref{Lemma:quasi_ASG1} suggest that $C^1$ coupling for SB-IGA planar parametrizations leads to adequate approximation results of the underlying isogeometric spaces.   In the paper  \cite{Collin2016AnalysissuitableGM} an important property of the AS-$G^1$ geometries is the large enough interface spaces of traces and transversal directional derivatives, see (27) and (28) in \cite{Collin2016AnalysissuitableGM}. More precisely, for the SB-IGA two-patch case we look at the space $$\widehat{\mathcal{V}}_{h,\Gamma} \coloneqq \{ \xi \mapsto [\phi, \nabla \phi \cdot \p{d}] \circ \p{F}^{(R)}(0,\xi) \ | \ \phi \in \mathcal{V}_h^{M,1} \},$$ where $\p{d}$ is the transversal vector at the interface defined through $$\p{d} \circ \p{F}^{(R)}(0,\xi) = \frac{1}{\alpha^{R}} \Big[ \partial_{\zeta}\p{F}^{(R)}(0,\xi)- \beta^{(R)}(\xi)\partial_{\xi}\p{F}^{(R)}(0,\xi) \Big] \ ;  \ \ \textup{cf. (23) in \cite{Collin2016AnalysissuitableGM}}.$$ Note that we can think of constant $\alpha^{(S)}$ due to the proof of Lemma \ref{Lemma:quasi_ASG1}. Now we want to explain, why for SB-IGA with NURBS boundary curves the corresponding interface space  still contains enough elements, at least away from the scaling center. The problem with the singularity is then faced in the subsequent section.
\begin{lemma}
    Let $p>r+1>1$ and $j>2, \ l>2$. Then $\widehat{B}_{j,p}^{r+1} \times \widehat{B}_{l,p}^{r} \in \widehat{\mathcal{V}}_{h,\Gamma}$. Moreover,
    if $c_2=0$ in \eqref{SB_IGA_param_2_patch} and if the $\p{F}^{(S)} \in C^1(\overline{\widehat{\o}^{(S)}})$ have a smooth inverse, it holds  $S_p^{r+1} \times S_{p}^r \subset \widehat{\mathcal{V}}_{h,\Gamma}$.
\end{lemma}
\begin{proof}
We follow the proof idea from \cite{Collin2016AnalysissuitableGM}.\\
First, let $c_2=0$ and assume a singular parametrization. 
    %Let $\hat{\phi}_0(\zeta, \xi) = \widehat{B}_{j,p}^r( \xi) ,  \ \hat{\phi}_1(\xi) = \widehat{B}_{l,p-1}^{r-1}(\xi), \ j>2, \, l>1 $.
    By the properties of the NURBS, it is straightforward to see that there are functions $ \hat{c}^{(S)} \in N_p^r $ with $\hat{c}^{(S)}(\zeta) = \zeta + \mathcal{O}(\zeta^2)$. Let $j>2, \, l>2$ and 
      define $$ \hat{g}^{(S)}(\zeta,\xi) \coloneqq \widehat{B}_{j,p}^{r+1}( \xi) + \big[ \beta^{(S)}(\xi) \partial_{\xi}\widehat{B}_{j,p}^{r+1}( \xi) + \alpha^{(S)} \widehat{B}_{l,p}^{r}(\xi)\big] \ \hat{c}^{(S)}(\zeta) $$% \in N_p^r \otimes S_p^r.$$
      Obviously, the composed mapping $\hat{g}$, i.e. $\hat{g}_{|\widehat{\o}^{(S)}} = \hat{g}^{(S)}$ is continuous and except of a shift in the parameter $\zeta$ they are feasible NURBS basis functions in $N_p^r \otimes S_p^r$. Moreover, one has $\hat{g}^{(S)}(\zeta,\xi) \in O(\xi^2)$ and in view of the considerations in Sec. \ref{section:approx_scaling} this implies $\hat{g}_{|\widehat{\o}^{(S)}} \circ (\p{F}^{(S)})^{-1} \in C^1(\overline{\o^{(S)}})$. The $C^1$ regularity of $g = \hat{g} \circ \p{F}^{-1}$ across the interface follows by Lemma \ref{Lemma:C1_condition}, meaning $ g \in \mathcal{V}_h^{M,1}$. And using (28)-(29) from  \cite{Collin2016AnalysissuitableGM} we can conclude $[g, \nabla g \cdot \p{d}] \circ \p{F}^{(R)}(0,\xi) = [\widehat{B}_{j,p}^{r+1}(\xi), \widehat{B}_{l,p}^{r}(\xi)]$.\\Basically, an analogous  argumentation yields the second part  of the assertion.
\end{proof}

In view of the Theorem 1 in \cite{Collin2016AnalysissuitableGM}, the space $\mathcal{V}_h^{M,1}$ seems appropriate for approximations away from the scaling center. From now on, we concentrate on the more interesting case with a singular parametrization, i.e. we assume that we have the scaling factor $q(\xi)=\xi$; cf. \eqref{eq:SB-IGA_param}.

\subsection{Approximation in the scaling center}
\label{section:approx_scaling}

 Clearly, a special consideration of the behavior near the scaling center is needed.  But a simple calculation shows that only such splines may cause problems at the scaling center which have non-vanishing values or derivatives in points $(\zeta, 0)$. For this purpose assume w.l.o.g. $\f{x}_0 = \f{0}, \ \p{F}_m(\zeta, \xi) =    \xi   \ \gamma(\zeta)$ and let $$\hat{\phi} \in \textup{span}\{ \widehat{N}_{i,p}^r \cdot \widehat{B}_{j,p}^r  \ | \ j\geq 3, r \geq 1\},$$ i.e. we have  a $C^1$ function with $\hat{\phi}= \partial_{\xi}\hat{\phi}= \partial_{\zeta}\hat{\phi} =0$ on $\{  (\zeta, 0) \  | \ \zeta \in [0,1] \}$; cf. \eqref{eq:soline_der}. \\
Now we check that the push-forward  $\phi \coloneqq \hat{\phi} \circ \p{F}^{-1}_m$ has a well-defined value and derivatives in $\f{x}_0$. The former is obvious and we concentrate on the derivatives. The chain rule yields 
\begin{align}
\label{eq:derivative_limit_1}
    \begin{bmatrix}
{\partial}_{\zeta} \hat{\phi}(\zeta,\xi) \\
\partial_{\xi}  \hat{\phi}(\zeta,\xi)
    \end{bmatrix} = \begin{bmatrix}
        \xi {\partial}_{\zeta}\gamma_1(\zeta)  & \xi {\partial}_{\zeta}\gamma_2(\zeta) \\
         \gamma_1(\zeta)  & \gamma_2(\zeta)
    \end{bmatrix} \begin{bmatrix}
        \partial_{x} \phi \circ \p{F}_m(\zeta,\xi) \\
        \partial_y \phi \circ \p{F}_m(\zeta,\xi)
    \end{bmatrix},
\end{align}
where $\gamma = (\gamma_1,\gamma_2)$.
If we consider the derivatives away from the singular point, we get by  assumption that the matrix on the right-hand side above is invertible, i.e.  $$ 0 \neq d(\zeta) \coloneqq {\partial}_{\zeta}\gamma_1(\zeta) \  \gamma_2(\zeta) -  {\partial}_{\zeta}\gamma_2(\zeta) \ \gamma_1(\zeta).$$
Hence,  it is
\begin{align}
\label{eq:derivativ limit_2}
   \begin{bmatrix}
        \partial_{x} \phi \circ \p{F}_m(\zeta,\xi)\\
        \partial_y \phi\circ \p{F}_m(\zeta,\xi)
   \end{bmatrix}  =  \frac{1}{\xi \ d(\zeta)}  \begin{bmatrix}
        \gamma_2(\zeta)  & -\xi {\partial}_{\zeta}\gamma_2(\zeta) \\
         -\gamma_1(\zeta)  & \xi {\partial}_{\zeta}\gamma_1(\zeta) 
    \end{bmatrix} 
     \begin{bmatrix}
{\partial}_{\zeta}  \hat{\phi} (\zeta,\xi)\\
\partial_{\xi}  \hat{\phi}(\zeta,\xi)
    \end{bmatrix}.
\end{align}
By linearity and the definition of the B-spline basis functions  we can suppose w.l.o.g.  $\hat{\phi}(\zeta,\xi) = \widehat{N}(\zeta) \ \xi^2,$
for a suitable $\widehat{N}$. Note that $ \widehat{B}_{j,p}^r(\xi) \in \mathcal{O}(\xi^2)$ for $\xi \rightarrow 0$ if $j\geq 3$.
For this case we study  the derivatives when  $\xi \rightarrow 0$. With \eqref{eq:derivativ limit_2} one sees
\begin{align}
\partial_{x} \phi\circ \p{F}_m(\zeta,\xi) & = \frac{{\partial}_{\zeta} \widehat{N}(\zeta) \  \xi^2}{\xi \ d(\zeta)} \ \gamma_2(\zeta)- 2\frac{\xi  \ \widehat{N}(\zeta) \ \xi}{\xi \ d(\zeta)} \ {\partial}_{\zeta}\gamma_2(\zeta) \overset{ \xi \rightarrow 0}{\longrightarrow} \ 0 .
\end{align}
But this implies directly that $\phi(x,y)$ has a well-defined $x$-derivative in the scaling center, namely $\partial_{x} \phi = 0$ in $\f{x}_0$. Analogously one gets the well-defined derivative $\partial_{y} \phi(\f{x}_0) = 0$.
Thus we can summarize that in the $m$-th patch the push-forwards of the $C^1$-smooth basis functions $\widehat{N}_{i,p}^r \cdot \widehat{B}_{j,p}^r, \ j \geq 3$ define mappings in $C^1(\overline{\o_m})$.

This means, it is justified to remove in each patch  initially  before patch coupling all the parametric basis functions $\widehat{N}_{i,p}^r  \cdot \widehat{B}_{j,p}^r$ with $j \leq 2$. But clearly, to preserve the approximation ability of SB-IGA test functions, we have to introduce new basis functions in the physical domain that determine the function value and derivatives at the scaling center. In the planar case, three additional test functions are sufficient, where we exploit the isoparametric paradigm to define preliminary test functions $\phi_{i,sc} \in C^0(\o)$ with
\begin{align*}
\phi_{1,sc}(\f{x}_0) = 1, \ \  \partial_x\phi_{1,sc}(\f{x}_0) =  \partial_y\phi_{1,sc}(\f{x}_0)= 0, \\
\phi_{2,sc}(\f{x}_0) = 0, \ \  \partial_x\phi_{2,sc}(\f{x}_0) = 1, \ \ \partial_y\phi_{2,sc}(\f{x}_0)= 0, \\
\phi_{3,sc}(\f{x}_0) = 0, \ \  \partial_x\phi_{3,sc}(\f{x}_0) = 0, \ \ \partial_y\phi_{3,sc}(\f{x}_0)= 1.
\end{align*}
Latter requirements can be easily satisfied if we use the entries of the geometry control points  as coefficients for the parametric pendants $\hat{\phi}_{i,sc}$. To be more precise, if we have
$$\p{F}_m = \sum_{j=1}^{n_2}  \sum_{i=1}^{n_1^{(m)}} \p{C}_{i,j}^{(m)} \widehat{N}_{i,p}^r  \cdot \widehat{B}_{j,p}^r, $$ then we set 
\begin{align}
\label{eq_sc_test_fun_param_1}
\hat{\phi}_{1,sc}^{(m)} \coloneqq  \sum_{j=1}^{p+1}\sum_{i=1}^{n_1^{(m)}} &  \widehat{N}_{i,p}^r  \cdot \widehat{B}_{j,p}^r,   \ \ 
\hat{\phi}_{2,sc}^{(m)} \coloneqq  \sum_{j=1}^{p+1}\sum_{i=1}^{n_1^{(m)}}  (\p{C}_{i,j}^{(m)})_1 \  \widehat{N}_{i,p}^r  \cdot \widehat{B}_{j,p}^r,  \\
&\hat{\phi}_{3,sc}^{(m)} \coloneqq  \sum_{j=1}^{p+1}\sum_{i=1}^{n_1^{(m)}}  (\p{C}_{i,j}^{(m)})_2  \ \widehat{N}_{i,p}^r  \cdot \widehat{B}_{j,p}^r \ .  \label{eq_sc_test_fun_param_3}
\end{align}
 Note that in the two-patch case we have  $\p{F}_1 = \p{F}^{(L)}(\cdot -1,\cdot), \ \p{F}_2 = \p{F}^{(R)}$.
Then the $\phi_{i,sc}$ are defined through $$(\phi_{i,sc})_{| \o_m} \circ \p{F}_m = \hat{\phi}_{i,sc}^{(m)}.$$
For example, the geometry in Fig. \ref{Fig:Controls} (a) with control points in Fig. \ref{Fig:Controls} (b) leads to the three scaling center functions in Fig. \ref{Fig:scaling_center_funs} below.

For a  $C^1$ spline $\textstyle\sum_j  c_j \ \widehat{B}_{j,p}^r (\xi)$ the values in the first mesh interval are completely determined by the terms with $j\leq p+1$. Thus using the partition of unity property we get directly that $ \hat{\phi}_{1,sc}^{(m)}=1, \ \hat{\phi}_{2,sc}^{(m)}=(\p{F}_m)_1, \ \hat{\phi}_{2,sc}^{(m)}=(\p{F}_m)_2 $ in a neighborhood of $ \{(\zeta,0) \ | \ \zeta \in [0,1] \}$ and hence  $ {\phi}_{1,sc}=1, {\phi}_{2,sc}=x, \ {\phi}_{3,sc}=y $ in a neighborhood of $\f{x}_0$. In other words, we can choose the latter three functions for the determination of values and derivatives at $\f{x}_0$, i.e. we add them to the set of basis functions used for the coupling step. Note that the continuity of the global parametrization $\p{F}$ implies the continuity of the composed  ${\phi}_{i,sc}$. 
\begin{figure}[H]	
\begin{minipage}{0.32\textwidth}
				\includegraphics[width=1\linewidth]{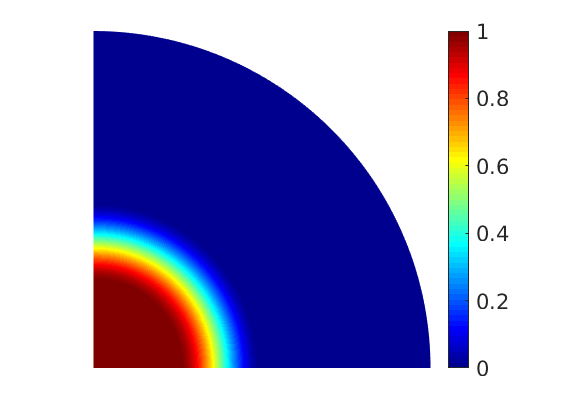}
    \caption*{\footnotesize (a) \footnotesize $\phi_{1,sc}$}
			\end{minipage}	
			\begin{minipage}{0.32\textwidth}
				\includegraphics[width=1\linewidth]{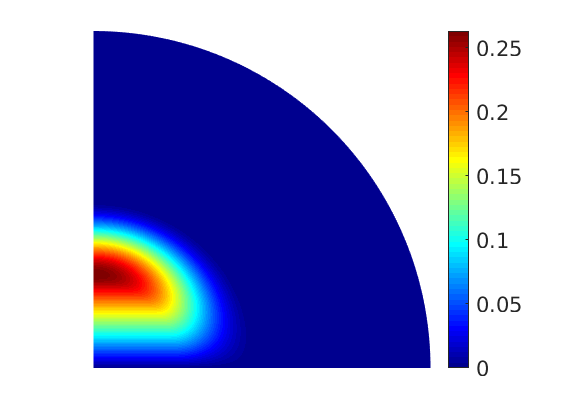}
        \caption*{ \footnotesize (b) \footnotesize  $\phi_{2,sc}$}
			\end{minipage}	
		\begin{minipage}{0.32\textwidth}
			\includegraphics[width=1\linewidth]{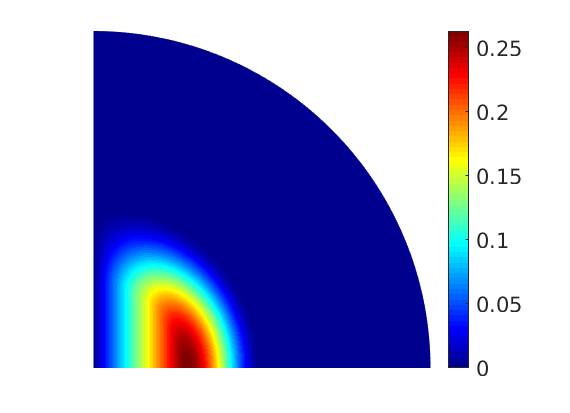}
       \caption*{ \footnotesize (c)   $\phi_{3,sc}$}
		\end{minipage}	
  \caption{ \small The auxiliary scaling center basis functions for the geometry and control net from Fig. \ref{Fig:Controls}}
  \label{Fig:scaling_center_funs}
\end{figure}
From the previous remarks, the principal idea how to determine the globally $C^1$-smooth basis functions is clear. Let $\mathcal{B}$ be the set of all uncoupled basis functions. First, one removes in each patch all the  basis functions $\widehat{N}_{i,p}^r  \cdot \widehat{B}_{j,p}^r, \ j < 3$. Then one adds the three basis functions that determine the Hermite data in the scaling center. After these two steps, we  can be sure that the remaining functions have well-defined derivatives and values in $\f{x}_0$. In particular, we deal now with a modified basis ${\mathcal{B}}^{'}$. Consequently, the actual $C^1$ coupling is then done only with the functions in ${\mathcal{B}}^{'}$. As a result, one obtains some space ${\mathcal{W}}_h^{M,1}$.   It should now we verified that indeed $${\mathcal{W}}_h^{M,1} = {\mathcal{V}}_h^{M,1} (\coloneqq {\mathcal{V}}_h^{M} \cap C^1(\overline{\o})).$$
\begin{lemma}
    It holds ${\mathcal{W}}_h^{M,1} = {\mathcal{V}}_h^{M,1}$.
\end{lemma}
\begin{proof}
It is enough to consider "$\supset$".\\
Assume there is $\phi \in {\mathcal{V}}_h^{M,1} \backslash {\mathcal{W}}_h^{M,1}$. 
    By the definition of the scaling center basis functions and the iso-parametric paradigm the three  polynomials $(x,y) \mapsto x, \ (x,y) \mapsto y, \ (x,y) \mapsto 1$ are elements of $ {\mathcal{W}}_h^{M,1}$. Thus, we can say w.l.o.g. $\phi(\f{x}_0)= \partial_{x}\phi(\f{x}_0) = \partial_y \phi(\f{x}_0)=0$. Writing $\hat{\phi}^{(m)}= \phi \circ \p{F}_m$ this implies on the one hand obviously 
   $ \hat{\phi}^{(m)}(\zeta, 0) = 0 , \forall \zeta$. And, on the other hand by the continuity of the derivative  \eqref{eq:derivative_limit_1} implies ${\partial}_{\zeta}\hat{\phi}^{(m)}(\zeta,0)=\partial_{\xi}\hat{\phi}^{(m)}(\zeta,0)=0, \forall \zeta.$\\
   But then it follows $\phi \in {\mathcal{W}}_h^{M,1}.$
\end{proof}
After we have seen how the scaling center issue can be handled, we have to face now the coupling conditions. This is content of the next section.

\subsection{Enforcing the coupling conditions}

We shortly mention an adaption of the approach of \cite{Collin2016AnalysissuitableGM}  to compute the actual $C^1$-regular test functions. We want to emphasize that there are also other ways to enforce the coupling conditions, e.g. a least squares ansatz. Nevertheless, in view of clarity we restrict ourselves to one coupling procedure.  Furthermore we explain the steps for a two-patch situation like Fig. \ref{Fig:2}.

We apply similar steps as in \cite{Collin2016AnalysissuitableGM} to our context which can be summarized as follows.\\
Initially, we start with the set of all the non-coupled basis functions $$\mathcal{\tilde{B}} \coloneqq \{ \widehat{N}_{i,p}^r(\cdot +1)  \cdot \widehat{B}_{j,p}^r  \ | \ 1 \leq i \leq n_1^{(L)} , \ 1 \leq j \leq n_2\} \cup \{ \widehat{N}_{i,p}^r \cdot \widehat{B}_{j,p}^r  \ | \ 1 \leq i \leq n_1^{(R)} , \ 1 \leq j \leq n_2\} ,$$  ${\mathcal{B}}\coloneqq \mathcal{\tilde{B}} \circ \p{F}^{-1}$ respectively. Here we extend an uncoupled basis function  to $\tilde{\Omega}$ by setting it to zero on the remaining patch. Further, below we write ${\mathcal{\tilde{B}}}^k$ for the set of parametric basis functions after the $k$-th  step of the coupling procedure.  The hat " $\tilde{\cdot}$ " is used if we consider the functions in the parametric domain $\tilde{\o}$ and we drop it if we mean the corresponding obvious push-forwards.
\begin{enumerate}
	\item Remove all basis functions corresponding to $\widehat{N}_{i,p}^r(\cdot +1)  \cdot \widehat{B}_{j,p}^r, \ \widehat{N}_{i,p}^r  \cdot \widehat{B}_{j,p}^r$ with  $j \leq 2$.
	We have now a new basis $\mathcal{\tilde{B}}^{1}$.
\end{enumerate} 
Latter step is done in order to remove the problematic test functions near the scaling center. 
\begin{itemize}
	\item[2.] The remaining test functions in $\mathcal{\tilde{B}}^{1}$ are coupled continuously which leads to a modified basis $\mathcal{\tilde{B}}^{2}$.
\end{itemize}
The $C^0$ coupling is easily achieved due to Assumption \ref{Assumption:coupling}.  This means we once more reduce again the number of basis functions.

\begin{itemize}
	\item[3.]  Incorporate the three scaling center test functions $\phi_{i,sc}$ determined by \eqref{eq_sc_test_fun_param_1}-\eqref{eq_sc_test_fun_param_3}; 	$\mathcal{\tilde{B}}^{2} \longrightarrow \mathcal{\tilde{B}}^{3}.$
\end{itemize}
This is, we can handle values and derivatives at the scaling center. 
\begin{itemize}
	\item[4.] Remove all basis functions that violate possible (problem-dependent) boundary \\ conditions; $\mathcal{\tilde{B}}^{3}$ $ \longrightarrow \mathcal{\tilde{B}}^{4}.$  
\end{itemize}
All  four steps can be summarized by means of a  transformation matrix  $M_0 \in \mathbb{R}^{N \times N_4}, \ N = \#  \mathcal{\tilde{B}}, \ N_4 = \#\mathcal{\tilde{B}}^{4}$ that connects the  new global continuous basis functions with   the  uncoupled ones. \\
Now we are able to face the actual $C^1$ coupling.
\begin{itemize}
	\item[5.]  One computes the normal derivative jumps across the interface, more precisely the derivative jump matrix $(M_J)_{i,j} = \langle [\![ \nabla\phi_i^4 \cdot \f{n}_L ]\!] ,[\![ \nabla\phi_j^4 \cdot \f{n}_L ]\!] \rangle_{L^2(\Gamma)}, $ $\  \mathcal{{B}}^{4} = \{ \phi_1^4, \phi_2^4, \dots , \phi_{N_4}^4 \}$. Here $\f{n}_L$ is the unit outer normal vector of the interface side w.r.t. to the left patch; see Fig. \ref{Fig:2}. And  $[\![  g ]\!]$ stands for the jump value of $g$ across the interface $\Gamma$.
\end{itemize}

\begin{itemize}
	\item[6.] The global $C^1$ test functions are then obtained by the  the null space matrix $M_1 = \textup{null}(M_J)$. Thereby we get the wanted basis function set $\mathcal{{B}}^{6}$.
\end{itemize}

This means if one has a weak linear formulation of some problem, the assembly of  matrices of the form $(A)_{i,j} = b( \phi_i^6, \phi_j^6), $ where $b(\cdot,\cdot)$ is a proper bilinear form and $ \mathcal{{B}}^{6} = \{ \phi^6_1, \phi^6_2, \dots , \phi^6_{N_6} \}$, can be computed from the uncoupled system matrix $(\tilde{A})_{i,j} = b( \phi_i,  \phi_j), \  \mathcal{{B}}= \{ \phi_1, \phi_2, \dots , \phi_{N} \} $ just by matrix multiplications, namely
$$A=M_1^T\ M_0^T \ \tilde{A} \ M_0 \ M_1.$$

\begin{remark}

Applying the above steps we can calculate $C^1$-regular basis functions in the two-patch case. In case of a SB-parametrization which consists of more than two patches the coupling is done for each interface according to the mentioned approach. However, the definition of the additional scaling center test functions  only has to be done once.
\end{remark}

\begin{remark}
Later in the numerics part, we use  Gauss quadrature rules to compute the matrix $M_J$ and other appearing integrals. In particular, we do not need to evaluate the basis functions at the singular point. Thus, the computation of the matrix $M_J$  is well-defined.
\end{remark}

\section{Remarks on generalizations}
    \label{section:generelizations}
	In the first part, the $C^1$-coupling was explained for the two-patch case, but as already mentioned the approach can be generalized to situations with more patches. Hereto one enforces the $C^1$ coupling at each interface, but in case of $c_2=0$, i.e. classical SB-IGA,  the  scaling center test functions $\phi_{i,sc}$ are defined once for the complete multipatch domain.

	\subsection{Non-star-shaped domains}
	\label{subsec:non-star-domains}
	Although we are now able to handle various geometries we are still limited to star-shaped domains. But frequently in applications the computational domain is not star-shaped. Nevertheless, in a special situation the $C^1$-coupling can be generalized  without loosing the (quasi) AS-$G^1$ structure. Namely, lets assume a decomposition of the domain $\o$ into star-shaped subdomains $\o_m$, where the interfaces between the different subdomains are straight lines; see for example Fig. \ref{Fig:non-star-shaped-domain}.\\
	Then we know how to couple the patches within each subdomain and the new interfaces between the star-shaped subdomains again  fit to the AS-$G^1$    framework since the two elements corresponding to that interface have a w.l.o.g. bilinear parametrization (see Fig. \ref{Fig:non-star-shaped-domain}). And due to Proposition 3 in \cite{Collin2016AnalysissuitableGM}  we have that bilinear multipatch parametrizations are AS-$G^1$. This property becomes evident in the numerical examples shown in the last part of the article.

	\begin{figure}[H]
 \begin{minipage}{0.49\linewidth}
 \hspace{2cm}
    \begin{tikzpicture}[scale=1.2]
		\draw [thick,fill=gray,opacity=0.2] plot [smooth, tension=0.6] coordinates {(0,0) (0,2) (1,1.8) (2,2) (2,0) (2.2,-0.5) (3,-0.6) (3.5,2) (4.5,2) (4.5,1) (4,0) (3.7,-1) (3,-1.5) (2,-1.3) (0.4,-1) (0,0)};
		\draw [thick] plot [smooth, tension=0.6] coordinates {(0,0) (0,2) (1,1.8) (2,2) (2,0) (2.2,-0.5) (3,-0.6) (3.5,2) (4.5,2) (4.5,1) (4,0) (3.7,-1) (3,-1.5) (2,-1.3) (0.4,-1) (0,0)};

		\draw[brown,thick] (0,0) to (2,0);
		\draw[brown,thick] (1.3,-1.2) to (2.25,-0.55);
		\draw[brown,thick] (3,-1.5) to (2.7,-0.8);
		%	\draw[brown] (0,0) to (2,0);
		\draw[brown,thick] (3.1,-0.3) to (3.9,-0.3);

		\draw[dashed] (1,1)--  (0,0);
		\draw[dashed] (1,1)--  (2,0);
		\draw[dashed] (1,1)--  (2,2);
		\draw[dashed] (1,1)--  (0,2);
		
		\draw[dashed] (1,-0.5)--  (0,0);
		\draw[dashed] (1,-0.5)--  (2,0);
		\draw[dashed] (1,-0.5)--  (1.3,-1.2);
		\draw[dashed] (1,-0.5)--  (2.25,-0.55);
		
		\draw[dashed] (2.3,-1)--  (3,-1.5);
		\draw[dashed] (2.3,-1)--  (2.7,-0.8);
		\draw[dashed] (2.3,-1)--  (1.3,-1.2);
		\draw[dashed] (2.3,-1)--  (2.25,-0.55);

		\draw[dashed] (3.37,-0.9)--  (3,-1.5);
		\draw[dashed] (3.37,-0.9)--  (2.7,-0.8);
		\draw[dashed] (3.37,-0.9)--  (3.1,-0.3);
		\draw[dashed] (3.37,-0.9)--  (3.9,-0.3);
		
		\draw[dashed] (3.76,1)--  (4.5,2);
		\draw[dashed] (3.76,1)--  (3.5,2);
		\draw[dashed] (3.76,1)--  (3.1,-0.3);
		\draw[dashed] (3.76,1)--  (3.9,-0.3);

		\draw[fill=red] (1,1) circle (1.9pt);
		\draw[fill=red] (1,-0.5) circle (1.9pt);	
		\draw[fill=red] (2.3,-1) circle (1.9pt);	
		\draw[fill=red] (3.37,-0.9) circle (1.9pt);
		\draw[fill=red] (3.76,1) circle (1.9pt);

		\end{tikzpicture} 
 \end{minipage}
 \hspace{0cm}
 \begin{minipage}{0.49\linewidth}
		\centering
		\begin{tikzpicture}[scale=0.7]
		\draw (2,3) -- (3.5,0);
		\filldraw[fill=gray, opacity = 0.2] (2,3) --(3.5,0) -- (2,0) -- (1,1.5) -- (2,3);
		\filldraw[fill=gray, opacity = 0.2] (2,3) --(3.5,0) -- (4,0.5) -- (4,2) -- (2,3);
		\draw[dashed] (2,3) -- (0,0);
		\draw[dashed] (3.5,0) -- (0,0);
		\draw[dashed] (2,3) -- (5,1.5);
		\draw[dashed] (3.5,0) -- (5,1.5);
		\node at (2.78,1) {$\Gamma$};
		\node[left] at (0,0) { \small $\f{x}_0$};
		\node[above] at (5.55,1.4) { \small  $\f{x}_1$};
		\node at (2.78,1) {$\Gamma$};
		\draw[fill = red] (0,0) circle(3pt); 
		\draw[fill = red] (5,1.5) circle(3pt); 
		\end{tikzpicture}
 \end{minipage}
		\caption{ \small A non-star-shaped domain which is divided into star-shaped subdomains that have non-curved interfaces. Such multipatch structures are still suitable for a $C^1$ coupling. If the interface between SB-parametrizations with two different scaling centers is a straight line, then w.l.o.g. the patches meet as two bilinear patches and the parametrization is quasi AS-$G^1$, i.e. AS-$G^1$ except at the singular points.}
		\label{Fig:non-star-shaped-domain}
	\end{figure}
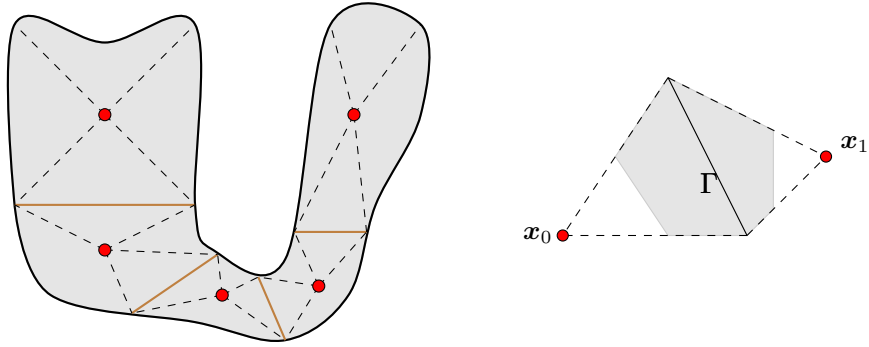

\subsection{Application to trimmed domains}
 Trimming, i.e. the cut off of domain parts utilizing trimming curves or surfaces, is a fundamental operation within CAD and is applied frequently. For more information and a detailed study we refer to \cite{Marussig2018}. But as explained in latter reference the  implementation of general trimming procedures is quite complicated and for  different approaches there arise  several difficulties. Especially the stable integration  over trimmed geometries might be an issue. SB-IGA with its boundary representation is suitable for the consideration of trimming and we want to explain here briefly, how we can handle $C^1$-coupling if  we incorporate trimming. The basic idea is to interpret the trimmed domain as a new computational domain which is defined by appropriate new boundary curves and to apply the coupling approach from above.  \\
  First, we explain the procedure by means of a simple example and concerning more complicated situations we add some remarks later.\\
  Let  the boundary curves $\gamma^{(m)}, \ m=1,\dots,n$ of the untrimmed domain $\Omega$ be given; see the blue square boundary in Fig. \ref{Fig:Trimming_1} (a).
  Further, let  a trimming curve $\gamma_{T}$ be given  such that there are two intersections with the boundary  $\partial \o$, e.g.  $\gamma^{(1)}(\zeta^{(1)})= \gamma_T(s_{1})$ and $\gamma^{(2)}(\zeta^{(2)})= \gamma_T(s_2)$ for proper $\zeta^{(k)} \in [0,1]$. Then, we have to check which curve segments belong to the boundary of the wanted trimmed domain $\Omega_T$. For the example  in Fig. \ref{Fig:Trimming_1} (a) we assume that the curve segments  $$
      \tilde{\gamma}_T,  \  \tilde{\gamma}^{(1)} , \  \tilde{\gamma}^{(2)} \  \ \textup{and} \ \ {\gamma}^{(3)}, \ {\gamma}^{(4)}  % & \coloneqq \{ \gamma_T(t) \ | \   t \in [\zeta_{1}, \zeta_{2}]  \} 
$$ define the new boundary of $\Omega_T$ .   \\
If we parameterize now the modified boundary parts $\tilde{\gamma}_T,  \  \tilde{\gamma}^{(1)} , \  \tilde{\gamma}^{(2)}$ utilizing standard NURBS or B-splines, i.e. e.g.  $\tilde{\gamma}_T \in (N_p^r)^2$ we are again in the situation of classical untrimmed SB-IGA provided that we have a suitable scaling center for the trimmed domain; see Remark \ref{remark:trimming_1} below. The exact boundary representation of the trimmed domain is always possible which can be seen as follows.\\
Let $\gamma \colon [0,1] \rightarrow \mathbb{R}^2$ be a NURBS curve in $(N_p^r)^2$ and $0 \leq \zeta^{(1)} < \zeta^{(2)} \leq 1$.  Inserting knots at $\zeta^{(1)}, \zeta^{(2)}$ s.t. the multiplicity of the knots $\zeta^{(1)},\zeta^{(2)}$ is $p+1$,  we  obtain   NURBS with   discontinuities  at the mentioned two knots  that represent  the original curve $\gamma$; cf. \cite{IGA1}. Consequently, each new NURBS basis function is non-zero only in one of the three intervals $$I_1 = [0, \zeta^{(1)}], \ \ I_2 = [\zeta^{(1)}, \zeta^{(2)}], \ \ I_3 = [ \zeta^{(2)},0], \ \ \textup{see Fig. \ref{Fig:Trimming_4}} . $$  Hence, choosing the appropriate NURBS and after a simple re-scaling  it is easy to see that  we can describe  each of the three curve segments $\gamma_{|I_i}$   with NURBS curves such that they start and end in points in $\{ \gamma(0), \ \gamma(\zeta^{(1)}), \ \gamma(\zeta^{(2)}), \ \gamma(1) \}$. In other words, if the intersection  parameter values in the trimming example above are known, we just have to apply some knot insertion steps to find exact parametrizations $\tilde{\gamma}^{(m)} \colon [0,1] \rightarrow \partial \Omega_T$ of the trimmed domain boundary curves. 

Clearly, the trimming example in Fig. \ref{Fig:Trimming_1} is very simple and in general situations some issues might occur. However, most problems can be handled as the next remarks indicate.

\begin{remark}
\label{remark:trimming_1}
  On the one hand, if we have the boundary curves of the trimmed domain, we need to compute a suitable scaling center, but the trimmed domain may loose its star-shape structure. Nevertheless, the trimmed domain can be partitioned into smaller star-shaped regions. 
  We propose a naive ansatz, namely we divide the trimmed domain via straight cut lines into smaller star-shaped blocks in order to maintain the analysis-suitability. In more detail, we repeat trimming steps with straight lines but incorporate all the subdomains arising from the trimming with that straight lines. This simple partition approach is illustrated in Fig. \ref{Fig:Trimming_2}. There we have the square domain with a trimming curve that leads to a non-star-shaped new domain. But the application of an division step with a proper straight line leads to a two star-domains for which scaling centers can be chosen, see Fig. \ref{Fig:Trimming_2} (b).
\end{remark}

  \begin{remark}
      On the other hand, there could be  more than $2$ intersection points between trimming curve and original boundary. Then, the overall trimming still can be done in an iterative manner, i.e. we cut off successively simple  regions. Also the situation of a trimming curve completely defined in the interior of the domain is manageable. One  partitions  the  trimmed domain via straight lines into star-shaped blocks and handles each of this subdomains with SB-IGA. We refer to Fig. \ref{Fig:Trimming_3} for an illustration.
  
  \end{remark}

  \begin{remark}
      We assume always that the SB-mesh boundary is given by NURBS curves, i.e. curves of the form $\gamma \colon [0,1] \rightarrow \mathbb{R}^2$. 
      Then, experiments not shown here indicate that if one uses exact boundary representations of the trimmed domain based on knot insertion together with the iso-parametric paradigm one might obtain SB meshes which are not optimal. To be more precise, one might get meshes with very thin elements  which are not optimal in the sense of condition. Maybe if we only  approximate the boundary curves with NURBS, e.g. with a $L^2$ projection ansatz,  we might relax latter issue. 
  \end{remark}

\begin{figure}[H]	
	\begin{tikzpicture}[scale=1.2]    
	    \node (eins) at (0,0) {\includegraphics[width=0.62\linewidth]{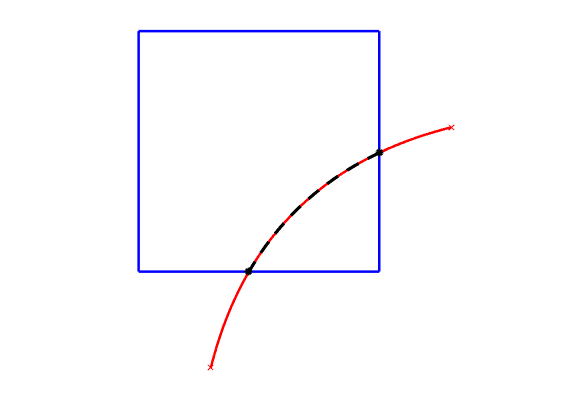}};
     \draw[fill=gray, opacity=0.2,shift={(-1.99,-0.87)}] (-0.02,0) rectangle (3.28,3.28);
         \node[red] at (1.9,1.2) {$\gamma_T$};
	     \node[blue] at (1.6,-0.2) {$\gamma^{(1)}$};
	     \node[blue] at (0.4,-1.15) {$\gamma^{(2)}$};
	     \node[blue] at (-2.3,0.5) {$\gamma^{(3)}$};
	     \node[blue] at (-0.2,2.66) {$\gamma^{(4)}$};
	     \node at (-0.2,0.1) {$\tilde{\gamma}_T$};
	      \draw[line width=0.8mm,dashed, green] (-1.95,-0.89) --(-0.5,-0.89);
	     \draw[line width=0.8mm,dashed, green] (1.132+0.16,0.77) --(1.132+0.16,2.4);
	     \node[green!50!black] at (-1.25,-1.15) {$\tilde{\gamma}^{(2)}$};
	       \node[green!50!black] at (1.7,1.8) {$\tilde{\gamma}^{(1)}$};
         \node (zwei) at (-0.45+6,0.69) {\includegraphics[width=0.437\linewidth]{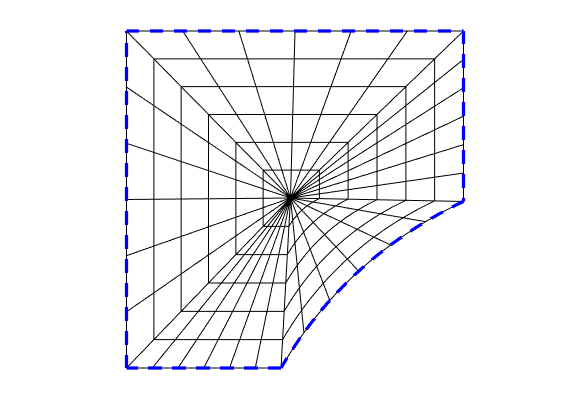}};
         \node at (-1,+3.1) {\footnotesize (a) Original domain with trimming curve $\gamma_T$};
        \node at (6,3.1) {\footnotesize (b) The trimmed domain.};
	\end{tikzpicture}
	\caption{ \small The SB-IGA ansatz is suitable to study trimmed geometries. One computes the intersections and the new boundary curves utilizing a knot insertions to decouple the boundary curves in without changing the geometry. Above the original $4$-patch domain is trimmed and as a result we get a $5$-patch geometry.}
	\label{Fig:Trimming_1}
\end{figure}

\begin{figure}[H]	
\centering
	\begin{tikzpicture}[scale=1.2]    
	    \node (eins) at (-0.5,0.7) {\includegraphics[width=0.4\linewidth]{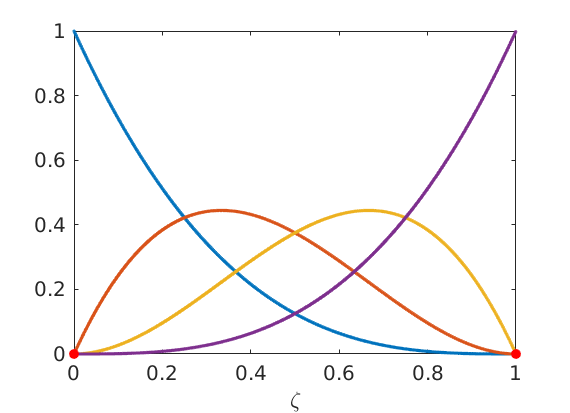}};
         \node (zwei) at (-0.45-0.3+6,0.69) {\includegraphics[width=0.4\linewidth]{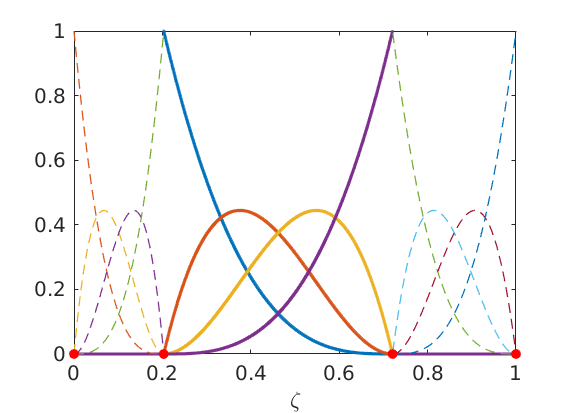}};
         \node at (-0.8,-1.5) {\footnotesize (a) Original basis NURBS};
        \node at (6-0.55,-1.5) {\footnotesize (b) NURBS after knot insertions };
        \draw[fill=black] (4.18,-0.61) circle (1.33pt);
        \draw[fill=black] (4.18+2.0252,-0.61) circle (1.33pt);
	\end{tikzpicture}
  \caption{ \small For the trimming example from Fig. \ref{Fig:Trimming_1} we have the NURBS on the left to describe the trimming curve $\gamma_T$. The intersections of $\gamma_T$ with the boundary curves correspond to the parameter values indicated with the black dots on the right. If we insert knots at these values $\zeta^{(1)}, \ \zeta^{(2)}$ we introduce new NURBS which are suitable for a decoupling and the extraction of the relevant part of the trimming curve that is related to the highlighted NURBS in the right figure.}
  \label{Fig:Trimming_4}
\end{figure}

\begin{figure}[H]	
	\begin{tikzpicture}[scale=1.2]    
	    \node (eins) at (0.3,0.7) {\includegraphics[width=0.49\linewidth]{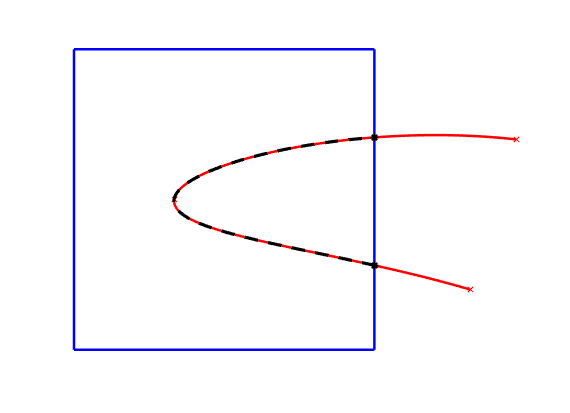}};
     \draw[fill=gray, opacity=0.2,shift={(-1.99,-0.87)}] (0,0) rectangle (3.28,3.28);
         \node (zwei) at (-0.45-0.3+6,0.69) {\includegraphics[width=0.4\linewidth]{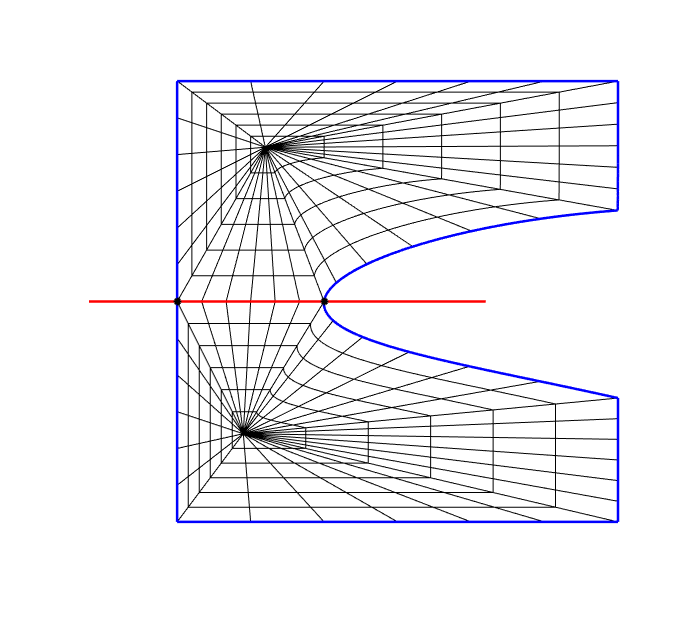}};
         \node at (-1,-1.5) {\footnotesize (a) Original domain with trimming curve $\gamma_T$};
        \node at (6,-1.5) {\footnotesize (b) Trimmed domain with cut line.};
	\end{tikzpicture}
\caption{ \small Here we see a  trimming curve that would lead to a non-star domain. We can divide the trimmed domain into two star-shaped parts using a straight cut line.   }
	\label{Fig:Trimming_2}
\end{figure}

\begin{figure}[H]
\centering
\begin{minipage}{0.32\textwidth}
	\begin{figure}[H]
		\centering
		\includegraphics[width=0.8\textwidth]{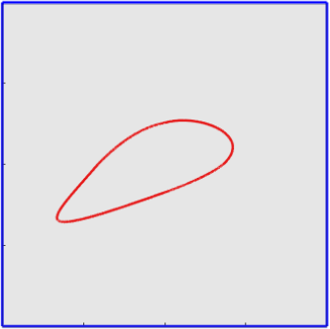}
		\caption*{\footnotesize{(a) Original domain with trimming curve $\gamma_T$}}
	\end{figure}
\end{minipage}
\begin{minipage}{0.32\textwidth}
	\begin{figure}[H]
		\centering
		\includegraphics[width=0.8\textwidth]{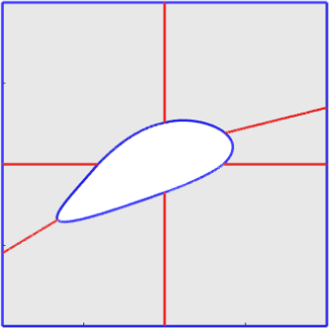}
		\caption*{\footnotesize{(b) Trimmed domain with cut lines.}}
	\end{figure}
\end{minipage}
\begin{minipage}{0.32\textwidth}
	\begin{figure}[H]
		\centering
		\includegraphics[width=0.8\textwidth]{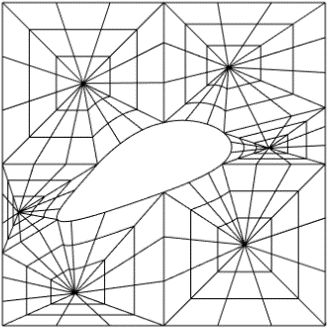}
		\caption*{\footnotesize{(c) Possible SB-mesh after the trimming.}}
	\end{figure}
\end{minipage}
    \caption{ \small Here we see a  trimming curve that would lead to a non-star domain. We can divide the trimmed domain into two star-shaped parts using a straight cut line.  }
    \label{Fig:Trimming_3}
\end{figure}

\section{Numerical examples}
\label{section:numerical_examples}
In this section, the proposed method is investigated regarding its performance by means of several numerical experiments. There are different application examples and methods where $C^1$ regularity of test and ansatz functions is required for the numerical calculations. Here we want to study the application in structural mechanics of Kirchhoff plate theory for which we can exploit the increased inter-patch regularity ansatz from above. The approach is based on the assumptions made by Kirchhoff  who states that "planes perpendicular to the mid surface will remain plane and perpendicular to the deformed mid surface".  The chosen examples demonstrate the power of the theory in the context of scaled boundary isogeometric analysis and contain untrimmed and trimmed examples, both. The approach has been implemented utilizing MATLAB \cite{MATLAB:2022} in combination with the open source package GeoPDEs \cite{geopdesv3}. It is especially designed for the solution of partial differential equations in the framework of isogeometric analysis. 
At first, the Kirchhoff plate formulation is stated briefly. Afterwards, the proposed method is checked on various examples to outline the features and characteristics of SB-IGA plates. The reference solutions are either obtained by the analytical solution, if available, or by results from literature. 
 %\textcolor{red}{Here comes the motivation to use Kirchhoff plates}
\subsection{Kirchhoff plate theory}
Kirchhoff plates determine a two-dimensional fourth-order boundary value problem governed by the bi-Laplace operator. The description of the formulation below follows the notation and derivation presented in \cite{Reali2015} and \cite{Ciarlet2002}. We consider domains $\Omega \subset \mathbb{R}^2$ that have a sufficient smooth boundary $\partial\Omega = \Gamma$ such that the unit normal vector $\mathbf{n}$ is well-defined. Further, the boundary is partitioned into parts for the specification of the energetically conjugate deflections and shear forces $\Gamma = \overline{\Gamma_u \cup \Gamma_Q}$ and rotations and bending moments $\Gamma = \overline{\Gamma_{\phi} \cup \Gamma_M}$. Further, we suppose  that $\Gamma_u \cap \Gamma_Q = \varnothing$ and $\Gamma_{\phi} \cap \Gamma_M = \varnothing$, respectively. The strong form of the Kirchhoff plate formulation is stated as 
\begin{align}
    \Delta^2 u &  = \frac{g}{D} & &\text{in}  \quad \Omega\\ 
    u & = u_{\Gamma} & &\text{on}  \quad \Gamma_u \label{Defor}\\ 
    - \nabla u \cdot \mathbf{n} & = \phi_{\Gamma}  & &\text{on} \quad \Gamma_{\phi} \label{Rot}\\ 
    D\left(\nabla \left(\Delta u\right) + (1-\nu) \boldsymbol{\Psi}(u) \right) \cdot \mathbf{n} & = Q_{\Gamma} & &\text{on} \quad \Gamma_{Q} \label{ShearForc}\\
    \nu D \Delta u + (1-\nu) D \mathbf{n} \cdot (\nabla \nabla u) \mathbf{n} & = M_{\Gamma} & &\text{on} \quad \Gamma_{M} \label{BendMom}.
\end{align}
with $\nabla(\cdot)$ as the gradient operator, $\Delta(\cdot)$ the Laplace operator and $\boldsymbol{\Psi}(\cdot)$ as the third order differential operator
\begin{equation}
    \boldsymbol{\Psi}(\cdot) = \left[ \partial_{xyy}(\cdot) ,\, \partial_{xxy}(\cdot)\right]^T.
\end{equation}
Moreover, $u$ is the deflection of the plate, $g$ the load per unit area and $u_{\Gamma}$, $\phi_{\Gamma}$, $M_{\Gamma}$ and $Q_{\Gamma}$ the prescribed deflections, rotations, bending moments and shear forces acting on the boundary. $D$ denotes the bending stiffness consisting of the thickness $t$, the Poisson ratio $\nu$ and the Young modulus $E$. As $t$, $\nu$ and  $E$ are assumed to be constant over the domain, yielding an isotropic, homogeneous material, the bending stiffness is defined as
\begin{equation}
    D = \frac{Et^3}{12(1-\nu^2)}.
\end{equation}
The classical primal weak form for the computation of approximate solutions has the form
\begin{equation}
    \text{Find}\quad a(u_h,v_h) = F(v_h), \qquad \forall v_h \in V_h
\end{equation}
for a suitable test function space $V_h$ that depends on the boundary conditions. Moreover, $a(u_h,v_h)$ is a bilinear form defined as
\begin{equation}
    a(u_h,v_h) = \int_{\Omega} D \left[\left(1-\nu \right) \nabla (\nabla v_h) : \nabla (\nabla u_h) + \nu \Delta v_h \Delta u_h \right] \mathrm{d}x \mathrm{d}y.
\end{equation}
Furthermore, the linear functional $F$ is denoted as
\begin{equation}
    F(v_h) = \int_{\Omega} g v_h \, \mathrm{d}x \mathrm{d}y + \int_{\Gamma_M} M_{\Gamma} \frac{\partial v_h}{\partial \mathbf{n}}\mathrm{d}\gamma + \int_{\Gamma_{Q}} Q_{\Gamma} v_h \mathrm{d}\gamma. 
\end{equation}
We remark that in  the numerical tests no external shear forces $Q_{\Gamma}$ and bending moments $M_{\Gamma}$ are applied except for the L-bracket \ref{L-bracket}. Moreover, the boundary conditions for $u_{\Gamma}$ \eqref{Defor} and $\phi_{\Gamma}$ \eqref{Rot} need to be enforced strongly, while the conditions of $Q_{\Gamma}$ \eqref{ShearForc} and $M_{\Gamma}$ \eqref{BendMom} are natural boundary conditions. Provided proper boundary conditions and suitable $V_h \subset H^2(\o)$, the conditions of the Lax-Milgram theorem are satisfied and  the weak form has a unique solution $ u_h \in V_h$. In the examples below, the discretization of the solution field is performed by the previously defined SB-IGA test functions including the coupling approach at the patch boundaries, which ensure $C^1$-continuity within the whole domain that is required for the plate formulation. In other words, the discrete spaces are of the form $V_h \subset {\mathcal{V}}_h^{M,1}$; see \eqref{eq:definition_coupled_multipatch}, since boundary conditions are taken into account. In the following examples, we mean mesh size $h$ if the underlying parametric meshes for each patch have $1/h$ equidistant subdivisions with respect to both parametric coordinate directions.
\subsection{Smooth solution on a square plate}
\label{subsection:smooth_solution}
At first, the plate formulation is checked on its general performance. A square plate of $\Omega = [-0.5,0.5]^2$ is subjected to a smoothly distributed source function $g$ that is chosen such that the exact solution is $ u=\cos(\pi  x)^2 \, \cos(\pi y)^2$. A plate of thickness $t \approx 0.1063$ is considered of elastic material with Young's modulus $E = 10^4$ and Poisson's ratio $\nu = 0.0$, such that the flexural stiffness $D = 1$. We have clamped boundary conditions and the domain is discretized with four SB-IGA patches. The scaling center is placed with an offset of the plate center, namely $x_{off} =-0.15$ and $y_{off} =0.1$, to demonstrate its approximation behaviour for non-symmetric meshes. Fig. \ref{fig:Ex1_Mesh} exemplary shows the mesh for $h=1/4$ and the corresponding deformation plot.
\begin{figure}[H]
\begin{minipage}{0.44\textwidth}
	\begin{figure}[H]
		\centering
        \includegraphics[width=\textwidth]{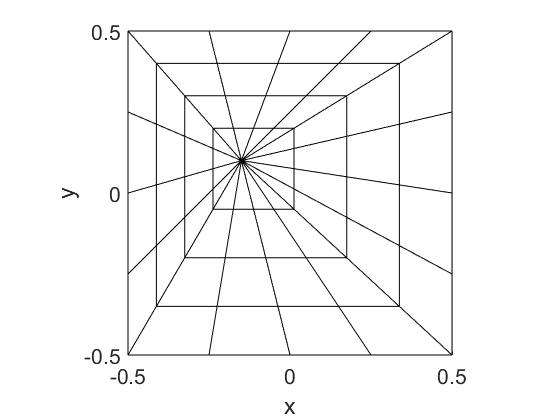}
		\caption*{\footnotesize{(a) Mesh}}
	\end{figure}
\end{minipage}
\begin{minipage}{0.44\textwidth}
	\begin{figure}[H]
		\centering
        \includegraphics[width=\textwidth]{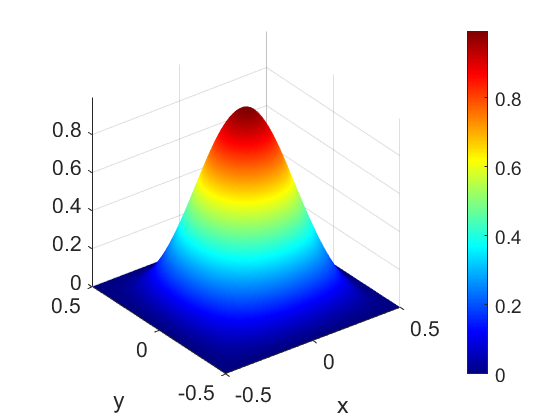}
		\caption*{\footnotesize{(b) Deformation plot}}
	\end{figure}
\end{minipage}
    \centering
    \caption{ \small Example of the smooth solution on a square plate. On the left side, the underlying mesh for $h = 1/4$ is shown. On the right, the deformation plot of the problem for the corresponding mesh is depicted, where $p=3$ and $r=1$.}
    \label{fig:Ex1_Mesh}
\end{figure}
The numerical solution is evaluated with respect to the $H^2$ seminorm and $L^2$ norm that are computed for third, fourth and fifth  order basis functions, where the error norms are defined as
\begin{flalign}
    | u - u_h |_{H^2(\Omega)} \quad \text{for the} \: H^2\text{-seminorm} \\
    || u - u_h ||_{L^2(\Omega)} \quad \text{for the} \: L^2\text{-norm}.
\end{flalign}
For the corresponding convergence rates, we refer to  Fig. \ref{fig:StabilityEx1}.
\begin{figure}[H]
\begin{minipage}{0.47\textwidth}
\begin{tikzpicture}
	\small \begin{loglogaxis} [
		width = \textwidth,
		xlabel={ \footnotesize Mesh size $h$},
		xmin=0.05, xmax=0.5,
		ymin=10^-5, ymax=4*10^0,
		xtick={0.0625,0.1,0.25,0.5},
		%xticklabel style={/pgf/number format/.cd,fixed,fixed zerofill,precision=2,},
		%x tick label style={/pgf/number format/sci}
		%log ticks with fixed point,
		log x ticks with fixed point,
		ytick={10^-8,10^-7,10^-6,10^-5,10^-4,10^-3,10^-2,10^-1,10^0},
		legend pos=south east,
		legend columns = 3,
		xmajorgrids=true,
		ymajorgrids=true,
		grid style=dashed,
		]
		\addplot%[color=blue, mark size=1.5pt,mark=*]
		coordinates {
			%(0.25, 1.399263687109137)%1.40087115929344)
			%(0.125, 0.268423772379134)%0.268455997061335)
			%(0.083333, 0.108930404582348)%0.108934619370518)
			%(0.0625, 0.058724813366949)%0.0587259114634415)
			%(0.05, 0.036690791454583)%0.0366911952478689)
			%(0.041667, 0.025091902185471)%0.0250920953142907)
			%(0.03571429, 0.018240560754154)%0.0182406925095398)
			%(0.03125, 0.013857469838205)%0.0138575279443904)
			% (0.03125, 0.013857469838205)%0.0138575279443904)
           (0.5, 4.88477522305725387)%1.40087115929344)
			(0.25, 1.39926367617716552)%0.268455997061335)
			(0.1666666, 0.524065706167172296)%0.108934619370518)
			(0.125,  0.268423770832067199)%0.0587259114634415)
			(0.1,  0.162553751539400687)%0.0366911952478689)
			(0.08333333, 0.108930404294319091)%0.0250920953142907)
			(1/14, 0.0780851124262938873)%0.0182406925095398)
			 (1/16, 0.0587248132652208252)%0.0138575279443904)
		};
		\addlegendentry{p=3, r=1},
		\addplot%[color=red, mark size=1.5pt,mark=square*]
		coordinates {
		%	(0.25, 0.151528229849199) %0.151561791157947)
		%	(0.125, 0.016457248449627)%0.016458114373246)
		%	(0.083333, 0.004675601235422)%0.00467570000971116)
		%	(0.0625, 0.001931999791387)%0.00193201966097626)
		%	(0.05, 9.762658932761335e-04)%0.000976178590377506)
		%	(0.041667, 5.621679609955389e-04)%0.000563447127672079)
%			(0.03571429, 5.133436625439626e-04)%0.00048837837020767)
%			(0.03125, 2.437554789883678e-04)%0.000320324931765429)
    	(0.5, 1.36342570030805166) %0.151561791157947)
			(0.25,  0.151528227056671821)%0.016458114373246)
			(0.166666666, 0.0408440254721990498)%0.00467570000971116)
			(0.125, 0.0164572484123510186)%0.00193201966097626)
			(0.1, 0.0082132392413997897)%0.000976178590377506)
			(0.08333333, 0.00467560125681835709)%0.000563447127672079)
		};
		\addlegendentry{p=4, r=1},
	    \addplot[color=darkgray,mark=triangle*]%[color=brown, mark=triangle*]%[color=blue, mark size=1.5pt,mark=*]
		coordinates {
		%(0.25		,0.013083013832603)
		%(0.125		,7.991902728115533e-04)
		%(0.083333333,1.527698754171126e-04)
		%(0.0625		,4.744731718453517e-05)
%		(0.05		,6.634559391139615e-04)
%		(0.041666667,7.502001388574731e-04)
%		(0.035714286,0.002505621129507)
%		(0.03125,0.007411430883170)
         (0.5		, 0.185702087549456607)
		(0.25		,0.013083013681589593)
		(0.166666666, 0.00258462255402451676)
		(0.125		,0.000799189017158653467)
        };
		\addlegendentry{p=5, r=1},
		\addplot[color=blue, mark size=1.5pt,mark=none, dashed]
		coordinates {
			(0.5, 0.5^2*10)
			%(0.05, 0.0025*10)
			(1/16, 1/16^2*10)};
		\addplot[color=red, mark size=1.5pt,mark=none, dashed]
		coordinates {
			(0.5, 0.5^3*5)
			%(0.05, 0.000125*5)
			(1/12, 1/12^3*5)};
		\addplot[color=darkgray, mark size=1.5pt,mark=none, dashed]
		coordinates {
			(0.5, 0.5^4*1.8)
			%(0.05, 0.000125*5)
			(1/8, 1/8^4*1.8)};
			\legend{p=3,p=4,p=5,$\mathcal{O}(h^2)$,$\mathcal{O}(h^3)$,$\mathcal{O}(h^4)$}
	\end{loglogaxis}
\end{tikzpicture}
\caption*{\footnotesize{(a) $|u - u_h|_{H^2(\Omega)}$}}
\end{minipage}
\begin{minipage}{0.47\textwidth}
\begin{tikzpicture}
	  \small \begin{loglogaxis} [
		width = \textwidth,
		xlabel={\footnotesize Mesh size $h$},
		xmin=0.05, xmax=0.5,
		ymin=10^-9, ymax=10^-2,
		xtick={0.0625,0.1,0.25,0.5},
		%xticklabel style={/pgf/number format/.cd,fixed,fixed zerofill,precision=2,},
		%x tick label style={/pgf/number format/sci}
		%log ticks with fixed point,
		log x ticks with fixed point,
		ytick={10^-12,10^-11,10^-10,10^-9,10^-8,10^-7,10^-6,10^-5,10^-4,10^-3,10^-2},
		legend pos=south east,
		legend columns = 3,
		xmajorgrids=true,
		ymajorgrids=true,
		grid style=dashed,
		]
		\addplot%[color=blue, mark size=1.5pt,mark=*]
		coordinates {
			%(0.25, 0.005575009395933)%0.00557500938985469)
			%(0.125, 2.047837887545693e-04)%0.000204783432363701 )
			%(0.083333, 2.915809773470997e-05)%0.0000292)
			%(0.0625, 7.620936039873554e-06)%0.0000076)
			%(0.05, 2.698861891367287e-06)%2.72247088440229*10^-6)
			%(0.041667, 1.238083099850577e-06)%9.70242987496485*10^-7)
%			(0.03571429, 1.631339608637220e-06)%1.87596427365346*10^-6)
%			(0.03125, 6.965339258950247e-07)%3.50347189022939*10^-7)
            (0.5, 0.0600617534325911112)
			(0.25, 0.00557450082300144957)
			(0.16666666, 0.00081626308570894267)
			(0.125, 0.00020477866809034384)
			(0.1, 6.98158068287520921e-05)
			(0.083333333,2.91587658032949004e-05)
           (1/14,1.41067562674136511e-05)
           (1/16,7.5153113420398704e-06)
		};
		\addlegendentry{p=3, r=1},
		\addplot%[color=red, mark size=1.5pt,mark=square*]
		coordinates {
			%(0.25, 1.275933013831962e-04)%0.000127593312690345)
			%(0.125, 2.580775057789607e-06)%2.58060876957345*10^-6)
			%(0.083333, 2.847855376544914e-07)%2.83286336064253*10^-7)
			%(0.0625, 8.407587721361716e-08)%6.49883343534805*10^-8)
%			(0.05, 3.105331711875771e-07)%6.57061297503098*10^-8)
%			(0.041667, 1.215035805582758e-06)%1.48584983197029*10^-6)
%			(0.03571429, 8.555379071959696e-06)%7.72334519730736*10^-6)
%			(0.03125, 1.598739795924605e-06)%4.93962751703958*10^-6)
    		(0.5, 0.00522418471570287781)%0.000127593312690345)
			(0.25, 0.000127586616855702214)%2.58060876957345*10^-6)
			(0.1666666666,1.28642355913829204e-05)%2.83286336064253*10^-7)
			(0.125, 2.5797286779233216e-06)%6.49883343534805*10^-8)
     		(0.1, 7.5688957792718852e-07)%6.57061297503098*10^-8)
		(0.0833333333, 2.85256500050647393e-07)%1.48584983197029*10^-6)

		};
		\addlegendentry{p=4, r=1},
	    \addplot[color=darkgray,mark=triangle*]%[color=brown, mark=triangle*]%[color=blue, mark size=1.5pt,mark=*]
		coordinates {
		%(0.25		,5.077253109601121e-06)
		%(0.125		,8.463281324382850e-08)
%		(0.083333333,7.792366353730157e-08)
%		(0.0625		,2.401400206180703e-08)
%		(0.05		,1.513016371265161e-05)
%		(0.041666667,1.701486189060420e-05)
%		(0.035714286,5.721166872093712e-05)
%		(0.03125,1.687706600803239e-04)
% New
		(0.5		,0.000285296000167890211)
		(0.25		,5.07688812116633106e-06)
        (0.166666666,4.05100526676587839e-07)
     	(0.125		,7.87332589216156844e-08)
%		(0.05		,1.513016371265161e-05)
        };
		\addlegendentry{p=5, r=1},
		\addplot[color=blue, mark size=1.5pt,mark=none, dashed]
		coordinates {
			(0.5, 0.5^4*0.29)
			%(0.05, 0.00000625*0.22)
			(1/16, 1/16^4*0.29)};
		\addplot[color=red, mark size=1.5pt,mark=none, dashed]
		coordinates {
			(0.5, 0.5^5*0.044)
			%(0.0625, 0.00000095367*0.165)
			(1/12, 1/12^5*0.044)};
		\addplot[color=darkgray, mark size=1.5pt,mark=none, dashed]
		coordinates {
			(0.5, 0.5^6*0.01)
			%(0.0625, 0.00000095367*0.165)
			(0.125, 0.125^6*0.01)};
					\legend{,,,$\mathcal{O}(h^4)$,$\mathcal{O}(h^5)$,$\mathcal{O}(h^6)$}
	\end{loglogaxis}
\end{tikzpicture}
\caption*{\footnotesize{(b) $||u - u_h||_{L^2(\Omega)}$}}
\end{minipage}
    \caption{ \small Convergence studies of the $H^2$ seminorm and the $L^2$ norm on the example of the
smooth solution on a square plate of orders $p = 3$, $p = 4$ and $p = 5$. For all these cases the regularity $r$ is $1$.}
    \label{fig:StabilityEx1}
\end{figure}
%The model shows that there is a severe discrepancy in the convergence between the stabilized and the unstabilized results. Especially for the $L^2$ norm, the results diverge for fine meshes.  For the stabilized results, the convergence rate coincides for both error estimations with the optimal convergence rate $C \cdot h^{p-1}$ for the $H^2$ seminorm and $C \cdot h^{p+1}$ for the $L^2$ norm. Note that $C$ differs for each order and error estimation.
The convergence rate indicate for both error estimates  optimal convergence rates $\mathcal{O}(h^{p-1})$ for the $H^2$ seminorm and $ \mathcal{O}(h^{p+1})$ for the $L^2$ norm. Especially we do not see a $C^1$ locking effect.

\subsection{Point load on a square plate}
\label{subsection:point_load}
The next example shows the capability of the proposed formulation to consider point loads even in the scaling center. Again, the square plate defined as $\Omega = [-0.5,0.5]^2$ is considered and the boundaries $\partial \Omega$ are simply supported. For the material parameters we choose  $E = 10^6$ and $\nu = 0$ and the thickness is chosen such that $D = 1$. The point load is defined as $F=1$. There is an analytical solution for the displacement of the plate center under a point load, namely using   \cite{reddy2006} one gets:
\begin{equation}
    u_{ex} = \frac{4FL^2}{D\pi^4} \sum_{n=1}^{\infty} \sum_{m=1}^{\infty} \frac{1}{(m^2+n^2)^2}\approx 0.01160,%0.011600839735872,
\end{equation}
where $L$ is the length of plate. The reference solution $u_{ref} \approx u_{ex}$ for the deflection in the center under the load application is obtained from the series above by taking sufficient terms into account. %It is important that the load is applied as a regularization of the Dirac delta, which means that the point load is distributed over the Gauss points close to the point of force application. The stamp-like area is decreasing with the fineness of the mesh. However, the displacement in the center is hence slightly underestimated.
Fig. \ref{fig:Ex_2_Mesh_Defo} shows a mesh exemplary and the corresponding deformation plot.
\begin{figure}[H]
\begin{minipage}{0.48\textwidth}
	\begin{figure}[H]
		\centering
  		\includegraphics[width=\textwidth]{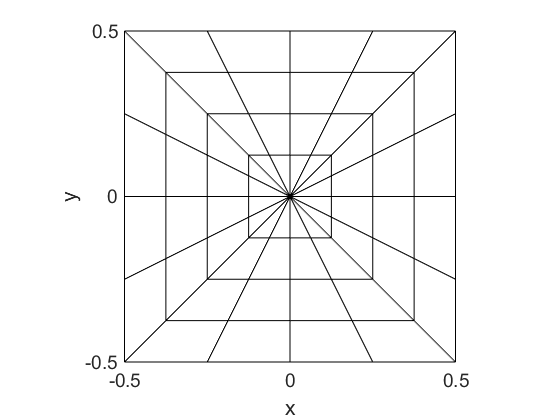}
		\caption*{\footnotesize{(a) Geometry and mesh}}
	\end{figure}
\end{minipage}
\begin{minipage}{0.48\textwidth}
	\begin{figure}[H]
		\centering
		\includegraphics[width=\textwidth]{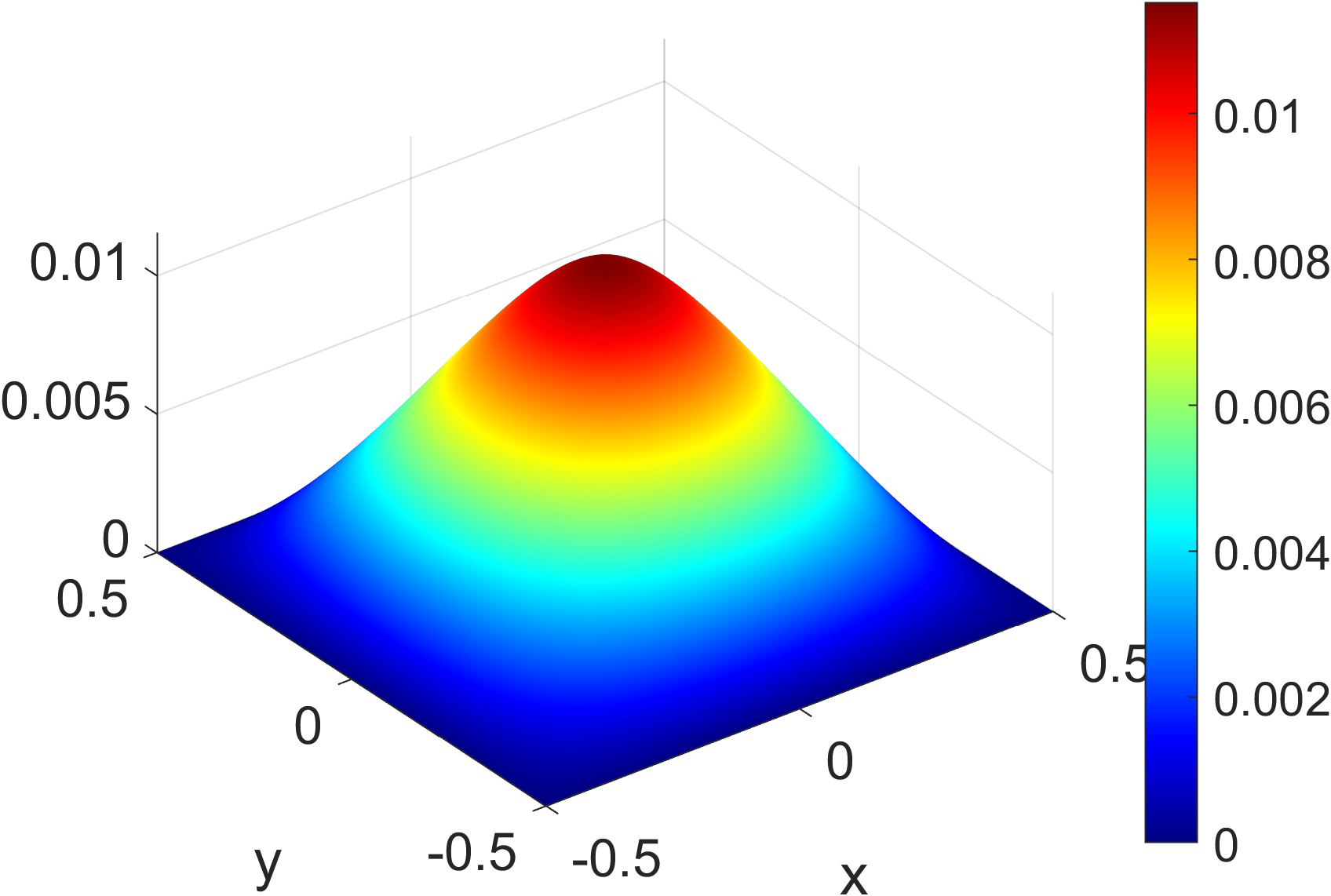}
		\caption*{\footnotesize{(b) Plot of deformation $u$}}
	\end{figure}
\end{minipage}
    \caption{\small Example of the point load on a square plate. On the left side, the underlying mesh for $h = 1/4$ is shown. On the right side the deformation plot of the corresponding mesh with $p = 3$ and $r = 1$ is displayed.}
    \label{fig:Ex_2_Mesh_Defo}
\end{figure}
Besides we look at the  relative errors  calculated by $|1-u/u_{ref}|$, where the convergence study is shown in Fig. \ref{fig:StabilityEx3}. Here, the example points out that the formulation is capable to determine results properly even for loading applied in the singular point. We think that the deviation from the best approximation rates is caused by the lack of regularity for the point load, which is not a function anymore,  strictly speaking.  In the mentioned figure we display also the bending moment $m_{11}=D \, \partial_{xx} u + \nu \, \partial_{yy} u$, a quantity for which the $C^1$ regularity is crucial.
\begin{figure}[H]
\centering
\begin{minipage}{0.47\textwidth}
\begin{tikzpicture}
	\small \begin{loglogaxis} [
		width = \textwidth,
		xlabel={ \footnotesize Mesh size $h$},
		xmin=0.025, xmax=0.5,
		ymin=5*10^-6, ymax=10^-1,
		xtick={0.03,0.05,0.1,0.2,0.3,0.5},
		%xticklabel style={/pgf/number format/.cd,fixed,fixed zerofill,precision=2,},
		%x tick label style={/pgf/number format/sci}
		%log ticks with fixed point,
		log x ticks with fixed point,
		ytick={10^-9,10^-8,10^-7,10^-6,10^-5,10^-4,10^-3,10^-2,10^-1},
		legend pos=south east,
		legend columns = 3,
		xmajorgrids=true,
		ymajorgrids=true,
		grid style=dashed,
		]
		\addplot%[color=blue, mark size=1.5pt,mark=*]
		coordinates {
		(0.5		,0.025201046736056)
		(0.25		,0.006471083270783)
		(0.166666667,0.002866547971774)
		(0.125		,0.001611547426389)
		(0.1		,0.001031227227093)
		(0.083333333,7.161857774627967e-04)
		(0.071428571,5.263625760759671e-04)
		(0.0625		,4.037151890995405e-04)
		(0.055555556,3.221851292611877e-04)
		(0.05		,2.539236560431535e-04)
		(0.045454545,2.134323424185514e-04)
		(0.041666667,1.785722046213634e-04)
		(0.038461538,1.450042205973778e-04)
		(0.035714286,1.340516158990779e-04)
		(0.033333333,1.335976031135999e-04)
        };
		\addlegendentry{p=3},
		\addplot%[color=red, mark size=1.5pt,mark=square*]
		coordinates {
		(0.5		,0.005610350974733)
		(0.25		,0.001387571163495)
		(0.166666667,6.159386840871584e-04)
		(0.125		,3.463683571577336e-04)
		(0.1		,2.216424597545341e-04)
		(0.083333333,1.516842726270173e-04)
		(0.071428571,1.166046563839673e-04)
		(0.0625		,7.446419582757269e-05)
		(0.055555556,8.241074600834342e-05)
        };
		\addlegendentry{p=4},
		\addplot[color=darkgray,mark=triangle*]%[color=brown, mark=triangle*]%[color=blue, mark size=1.5pt,mark=*]
		coordinates {
		(0.5		,0.001884689247931)
		(0.25		,4.618003607584909e-04)
		(0.166666667,2.048264228111663e-04)
		(0.125		,1.141109943840757e-04)
		(0.1		,7.122105911705479e-05)
        };
		\addlegendentry{p=5},
%		\addplot[color=blue, mark size=1.5pt,mark=none, dashed]
%		coordinates {
%			(0.25, 0.00390625*0.22)
			%(0.05, 0.00000625*0.22)
%			(0.03125, 0.0000009536743*0.22)
%			};
%		\addplot[color=red, mark size=1.5pt,mark=none, dashed]
%		coordinates {
%			(0.25, 0.0009765625*0.165)
			%(0.0625, 0.00000095367*0.165)
%			(0.03125, 0.0000000298*0.165)};
	\end{loglogaxis}
\end{tikzpicture}
\caption*{\footnotesize{(a) $|1-u/u_{ref}|$}}
\end{minipage}
\begin{minipage}{0.5\textwidth}
	\begin{figure}[H]
		\centering
	  	\includegraphics[width=\textwidth]{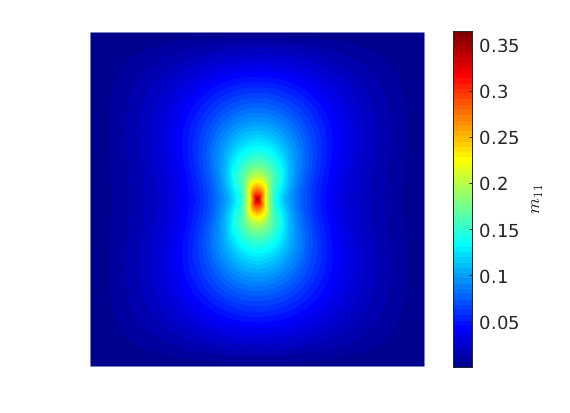}
    \vspace*{0.0cm}
		\caption*{\footnotesize{(b) Bending moment $m_{11}$}}
	\end{figure}
 \vspace*{0.08cm}
\end{minipage}
    \caption{\small On the left, the convergence study of the example of the point load on a square plate and orders of $p = 3$, $p = 4$, and $p = 5$ is shown. For all these cases the regularity is $r=1$. On the right, the bending moment $m_{11}$ is plotted for $p=4$ and $h=1/8$.}
   \label{fig:StabilityEx3}
\end{figure}

% The example emphasizes, that the SB-IGA formulation is capable to even handle loads that are applied in the scaling center.

\subsection{Perforated circular disk}
Having tested the method for the plate formulation in a simply square domain,  trimmed  geometries are now incorporated. Therefore, a simply supported disk of radius $R = 1$ and $t = 0.02$, $ E = 10^7$, $\nu = 0.3$ is trimmed by four holes of diameter $d_{hole}=0.1$, see Fig. \ref{fig:Ex_3_Geo_Mesh} (a). The center of the holes are placed on the $x$ and $y$ axes with a distance of $0.4$ from the origin. The disk is loaded by a constant surface load $g=1$. The reference solution is obtained from \cite{liu2018} as the maximal deflection in the center $u_{ref} =0.008950$, which is a numerically converged solution determined using the commercial software Abaqus with the triangular, linear shell element S3.  
Since the SB-IGA formulation requires star shaped patches, it is not possible to determine the deflection by a single scaling center but several scaling centers are required; see Sec. \ref{subsec:non-star-domains}. In the following, a mesh of $25$ scaling centers is evaluated (see Fig. \ref{fig:Ex_3_Geo_Mesh} (b)) with varying numbers of elements and orders on each patch.
\begin{figure}[H]
\begin{minipage}{0.47\textwidth}
	\begin{figure}[H]
		\centering
		\includegraphics[width=0.8\textwidth]{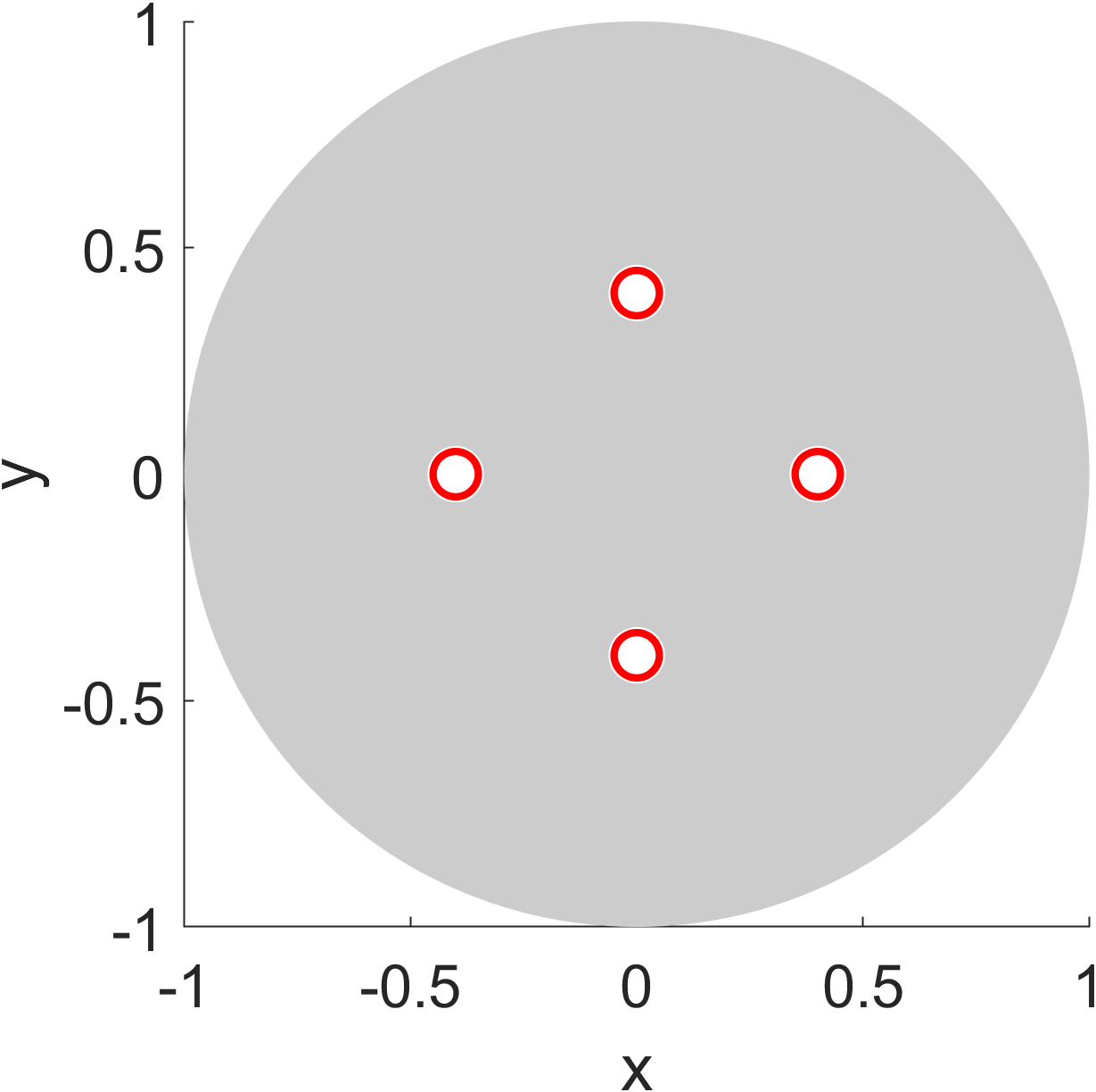}
		\caption*{\footnotesize{(a) Geometry}}
	\end{figure}
\end{minipage}
\begin{minipage}{0.47\textwidth}
	\begin{figure}[H]
		\centering
		\includegraphics[width=0.8\textwidth]{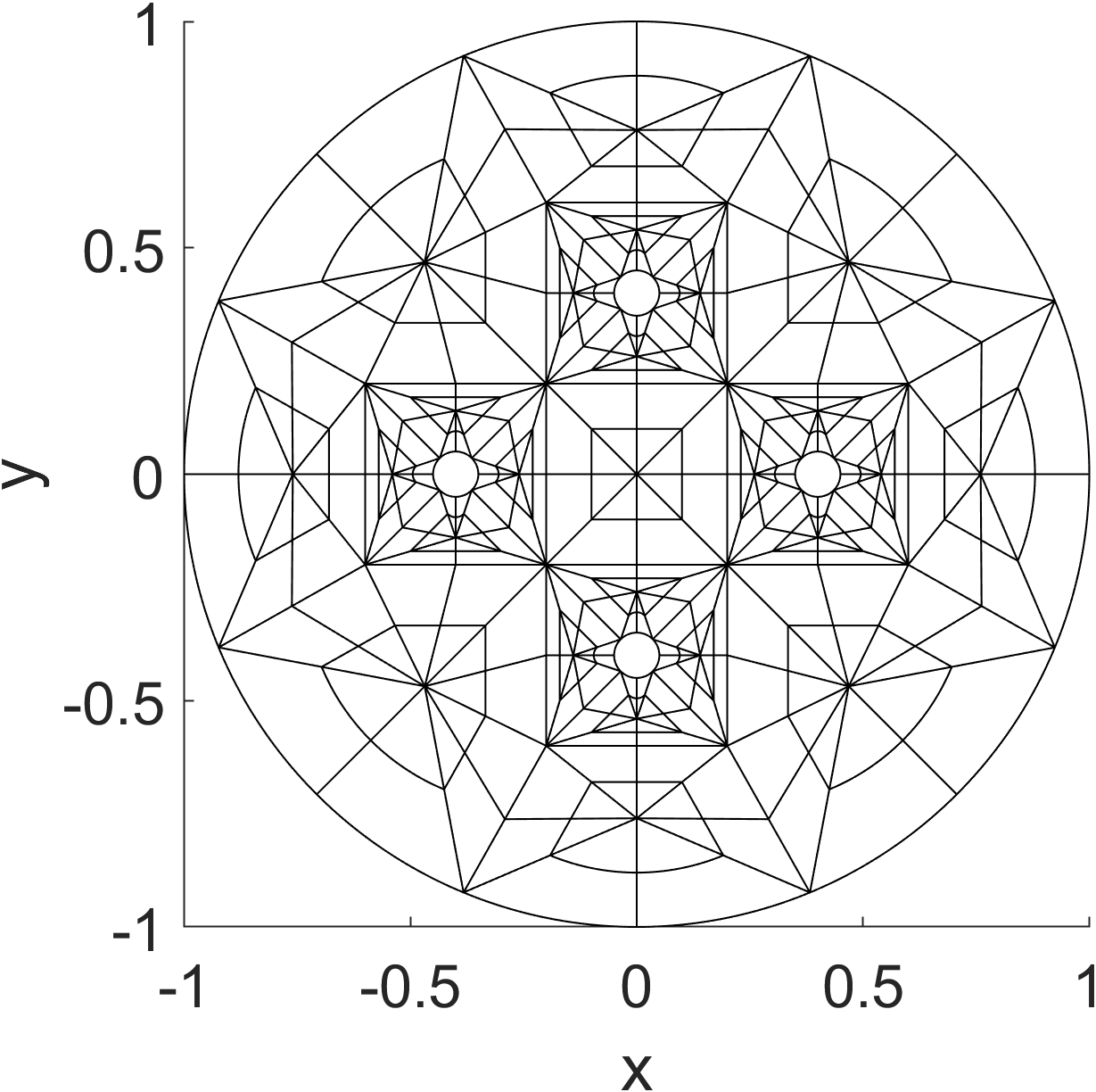}
		\caption*{\footnotesize{(b) Mesh}}
	\end{figure}
\end{minipage}
    \caption{\small Model of the perforated circular disk. On the left side the geometry with the trimming curves is drawn. On the right side, the underlying mesh of $h = 1/2$ is shown.}
    \label{fig:Ex_3_Geo_Mesh}
\end{figure}
\begin{comment}
    \begin{figure}[H]	
\centering
	\begin{tikzpicture}[scale=1.2]    
	    \node (eins) at (0.3,0.7) {\includegraphics[width=0.31\linewidth]{Images_NE/Ex5/Geometry.png}};
         \node (zwei) at (-0.45-0.3+6.6,0.69) {\includegraphics[width=0.31\linewidth]{Images_NE/Ex5/Mesh.png}};
         \node at (-0.7,-2.5) {\footnotesize (a) Geometry};
        \node at (5.4,-2.5) {\footnotesize (b) Mesh structure};
         \node at (1.1,0.7) {\footnotesize $\boxed{\textup{clamped}}$};
          \node at (1.5,1.95) {\footnotesize $\boxed{\textup{line load}}$};
         \draw[->] (0.7,0.91) -- (-0.48,2.85);
         \draw[->] (0.7,0.49) -- (-0.48,-0.7);
         \draw[->, out=0,in=-30] (2.1,1.9) to (2.25,2.65);
         \draw[->,very thick] (2.35,3.4) -- (2.35,2.8); 
	\end{tikzpicture}
 \caption{Example of the L-bracket. On the left side, the geometry is shown. On the right side the underlying mesh of $h = 1/4$ and $p = 3$, $r = 1$ is shown.}
    \label{fig:Ex_5_Geo_Mesh}
\end{figure}
\end{comment}
Fig. \ref{fig:Ex4_Results} (a) shows the deformation plot of the example corresponding to the mesh presented in Fig. \ref{fig:Ex_3_Geo_Mesh}. Furthermore, the convergence is evaluated in \ref{fig:Ex4_Results} (b) for the orders of $p=3$, $1/h=2,3,4,5$  and $p=4$, $1/h =2,3,4$. For a constant surface load  nearly exact results are obtained by at least cubic basis functions. This is fulfilled for both orders.  A comparison to an automatic meshing approach of triangular Bézier spline elements presented in \cite{liu2018} shows that the incorporation of trimming in the structural description of the model as in the context of SB-IGA has advantages over mesh finding algorithms as the meshing algorithm does not converge towards the reference solution even for very fine meshes.

\subsection{L-bracket}\label{L-bracket}
In the second last example, an L-bracket is investigated analogously to \cite{Coradello2021} and \cite{Benzaken2017} with the geometry shown in Fig. \ref{fig:Ex_5_Geo_Mesh}. The bracket has underlying parameters $E= 200\cdot 10^9$, $\nu = 0$ and $t=0.01$. The boundary conditions are clamped conditions on the left holes (upper and lower hole) as well as a constant line load of $f = 100$ on the upper right edge marked in blue. Further, the partitioning of the mesh created with 20 scaling centers is given in Fig. \ref{fig:Ex_5_Geo_Mesh} (b). We emphasize that meshes with less scaling centers could be applied, however, the model is not chosen to demonstrate the convergence, but the general application to complex models. 
\begin{comment}
\begin{figure}[H]
\centering
\begin{minipage}{0.34\textwidth}
	\begin{figure}[H]
		\centering
		\includegraphics[width=\textwidth]{Images_NE/Ex5/Geometry.png}
		\caption*{\footnotesize{(a) Geometry}}
	\end{figure}
\end{minipage}
\hspace{2cm}
\begin{minipage}{0.34\textwidth}
	\begin{figure}[H]
		\centering
		\includegraphics[width=\textwidth]{Images_NE/Ex5/Mesh.png}
		\caption*{\footnotesize{(b) Mesh structure}}
	\end{figure}
\end{minipage}
\end{figure} 

\end{comment}
\begin{figure}[H]
\begin{minipage}{0.5\textwidth}
	\begin{figure}[H]
		\centering
  \begin{tikzpicture}
  \node (eins) at (0.3,0.7) {\includegraphics[width=\textwidth]{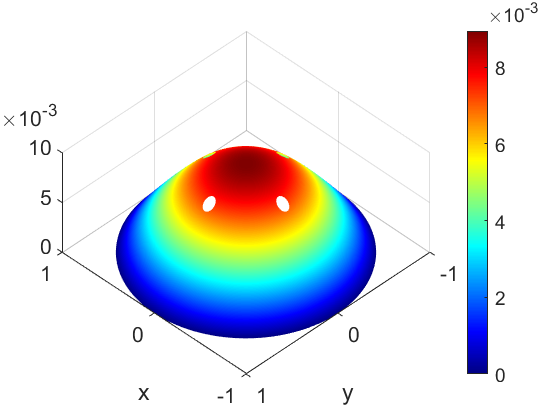}};
  \node at (2.6,2.1) {$u$};
  %\node at (0,0) {}
      %node(node1) at (0,0){\includegraphics[width=\textwidth]{Images_NE/Ex4/(4)_Disp_p3_n2.png}}
  \end{tikzpicture}		
		\caption*{\footnotesize{(a) Deformation plot}}
	\end{figure}
\end{minipage}
\hspace*{0.5cm}
\begin{minipage}{0.47\textwidth}
\begin{tikzpicture}
	\small \begin{semilogxaxis} [
		width = \textwidth,
		xlabel={ \footnotesize Degrees of freedom},
		xmin=2500, xmax=20000,
		ymin=0.4, ymax=1.2,
		xtick={10^3,10^4,10^5},
		%xticklabel style={/pgf/number format/.cd,fixed,fixed zerofill,precision=2,},
		%x tick label style={/pgf/number format/sci}
		%log ticks with fixed point,
		%log x ticks with fixed point,
		%scaled x ticks=false,
		ytick={0, 0.2,0.4,0.6,0.8,1.0, 1.2},
		legend pos=south east,
		legend columns = 3,
		xmajorgrids=true,
		ymajorgrids=true,
		grid style=dashed,
		]
		\addplot%[color=blue, mark size=1.5pt,mark=*]
		coordinates {
		(3780,0.008936832751940/0.008950) % n = 2
		(6720,0.008948664959968/0.008950) % n = 3
		(10500,0.008951409403031/0.008950) % n = 4
		(15120,0.008952339264834/0.008950) % n = 5
		(20580,0.008952279334616/0.008950) % n = 6
		};
		\addlegendentry{p=3},
		\addplot%[color=red, mark size=1.5pt,mark=square*]
		coordinates {
        (6722,0.008951726607923/0.008950) % n = 2
        (12705,0.008952553645941/0.008950) % n = 3
        (20582 , 0.008952323849834/0.008950) % n = 4
        };
		\addlegendentry{p=4},
		\addplot[color=darkgray,mark=triangle*]%[color=brown, mark=triangle*]%[color=red, mark size=1.5pt,mark=square*]
		coordinates {
        (2949,  0.007367/0.008950)
        (5154,  0.008273/0.008950)
        (9219,  0.008608/0.008950)
        (17109, 0.008772/0.008950)
        };
		\addlegendentry{\cite{liu2018}},
%		\addplot[color=blue, mark size=1.5pt,mark=none, dashed]
%		coordinates {
%			(0.25, 0.00390625*0.22)
			%(0.05, 0.00000625*0.22)
%			};
%		\addplot[color=red, mark size=1.5pt,mark=none, dashed]
%		coordinates {
%			(0.25, 0.0009765625*0.165)
			%(0.0625, 0.00000095367*0.165)
%			(0.03125, 0.0000000298*0.165)};
	\end{semilogxaxis}
\end{tikzpicture}
\caption*{\footnotesize{(b) $u/u_{ref}$}}
\end{minipage}
    \caption{ \small On the left side, the corresponding deformation plot to the mesh of Fig. \ref{fig:Ex_3_Geo_Mesh} with $p=3$ and $r=1$ is shown. On the right side, the convergence studies of the displacement in the middle of the perforated disk with orders of $p = 3$, $p = 4$ $(r=1)$ is compared to \cite{liu2018}.}
   \label{fig:Ex4_Results}
\end{figure}
\begin{figure}[H]	
\centering
	\begin{tikzpicture}[scale=1.2]    
	    \node (eins) at (0.3,0.7) {\includegraphics[width=0.31\linewidth]{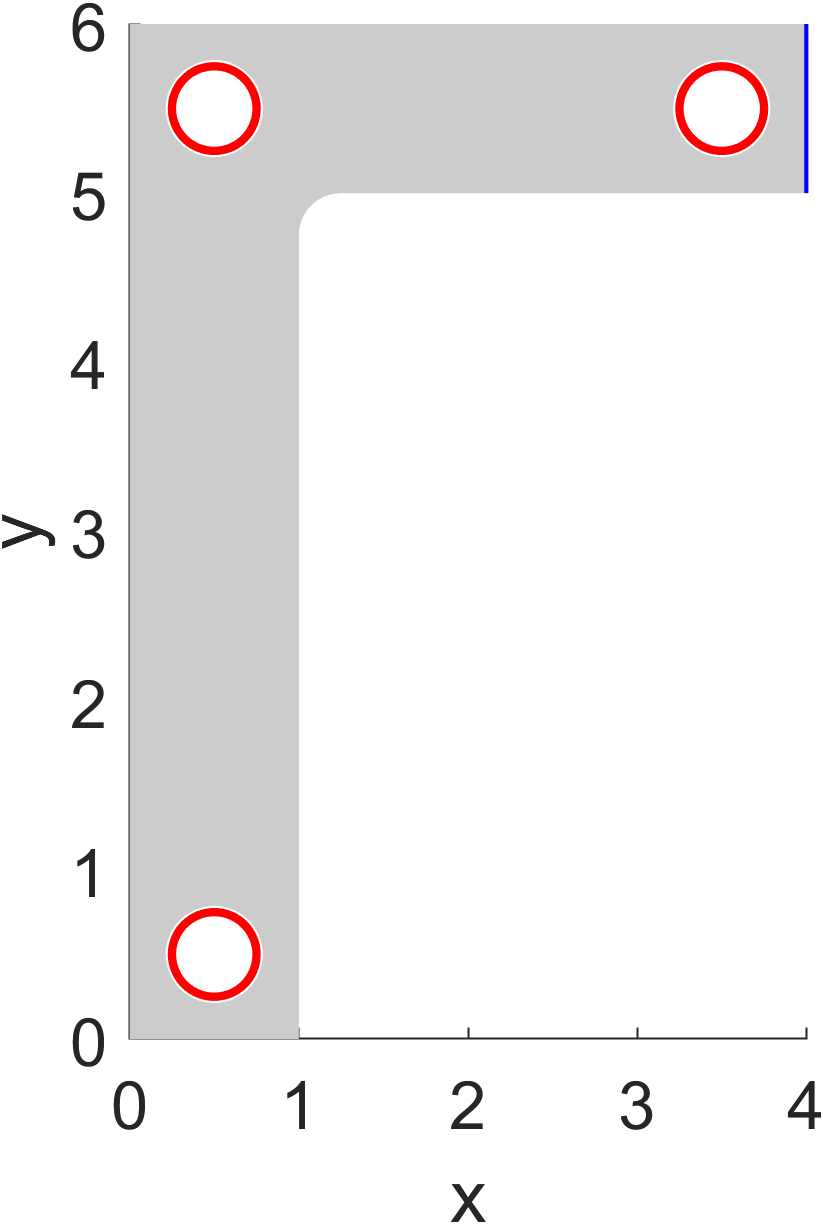}};
         \node (zwei) at (-0.45-0.3+6.6,0.69) {\includegraphics[width=0.31\linewidth]{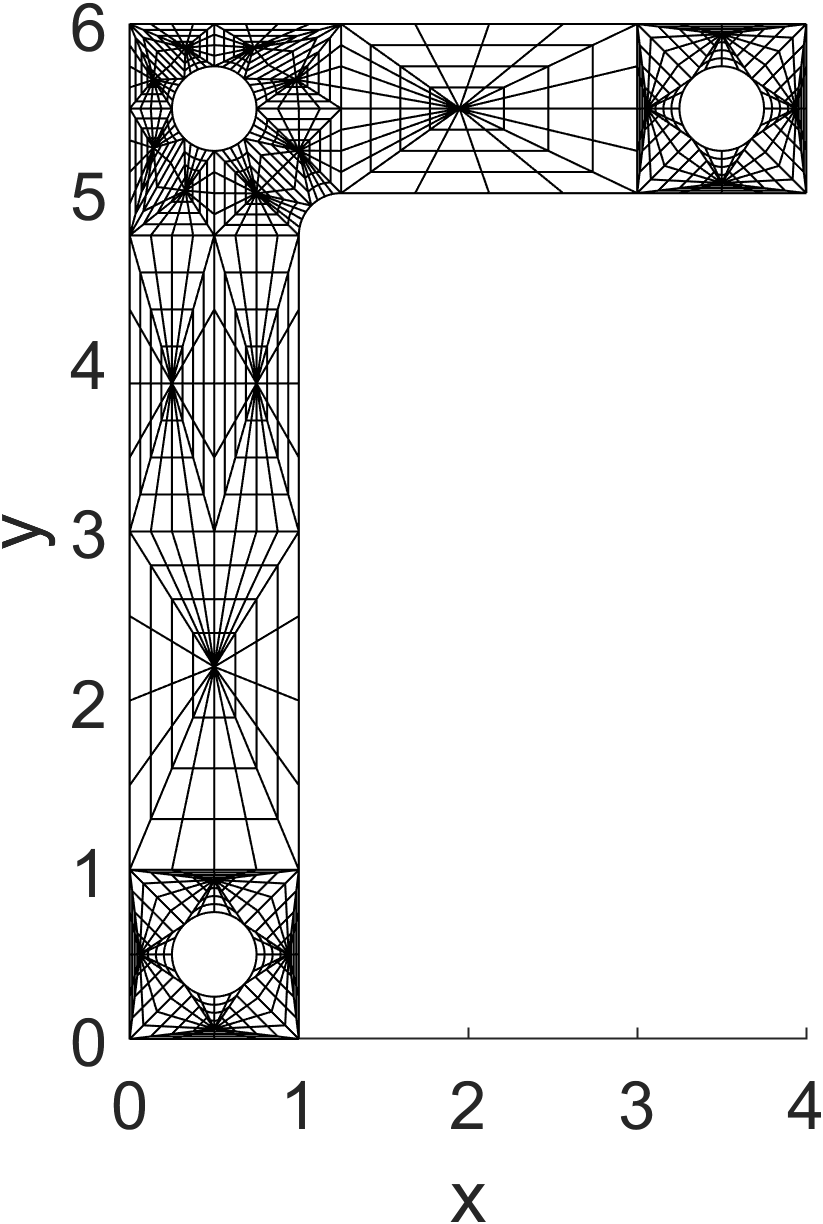}};
         \node at (-1.5,-2.4) {\footnotesize (a) Geometry};
        \node at (5.4,-2.4) {\footnotesize (b) Mesh structure};
         \node at (1.1,0.7) {\footnotesize $\boxed{\textup{clamped}}$};
          \node at (1.5,1.95) {\footnotesize $\boxed{\textup{line load}}$};
         \draw[->] (0.7,0.91) -- (-0.48,2.85);
         \draw[->] (0.7,0.49) -- (-0.48,-0.7);
         \draw[->, out=0,in=-30] (2.1,1.9) to (2.25,2.65);
         \draw[->,very thick] (2.35,3.4) -- (2.35,2.8); 
	\end{tikzpicture}
 \caption{\small The L-bracket. On the left side, the geometry is shown. On the right side the underlying mesh for $h = 1/4$ is given.}
    \label{fig:Ex_5_Geo_Mesh}
\end{figure}
\begin{comment}
\begin{figure}[htp]
\centering
\includegraphics[width=.33\textwidth]{Images_NE/Ex5/Deformation_u_p=3.png}
\hfill
\includegraphics[width=.33\textwidth]{Images_NE/Ex5/Deformation_u_p=4.png}\hfill
\includegraphics[width=.33\textwidth]{Images_NE/Ex5/Coradello_u.PNG}
\caption{Example of the L-bracket with line load. On the left and in the middle, the deformation plots of the proposed approach are shown with mesh of $h = 1/4$ , $r = 1$ and polynomial order of $p=3$ and $p=4$, respectively. On the right side the deformation plot of the results presented in \cite{Coradello2021} with order $p=3$ is shown.}
\label{fig:Ex5_Deform}
\end{figure}
\end{comment}
\begin{figure}[H]
\centering
\begin{minipage}{0.32\textwidth}
    \includegraphics[width=\textwidth]{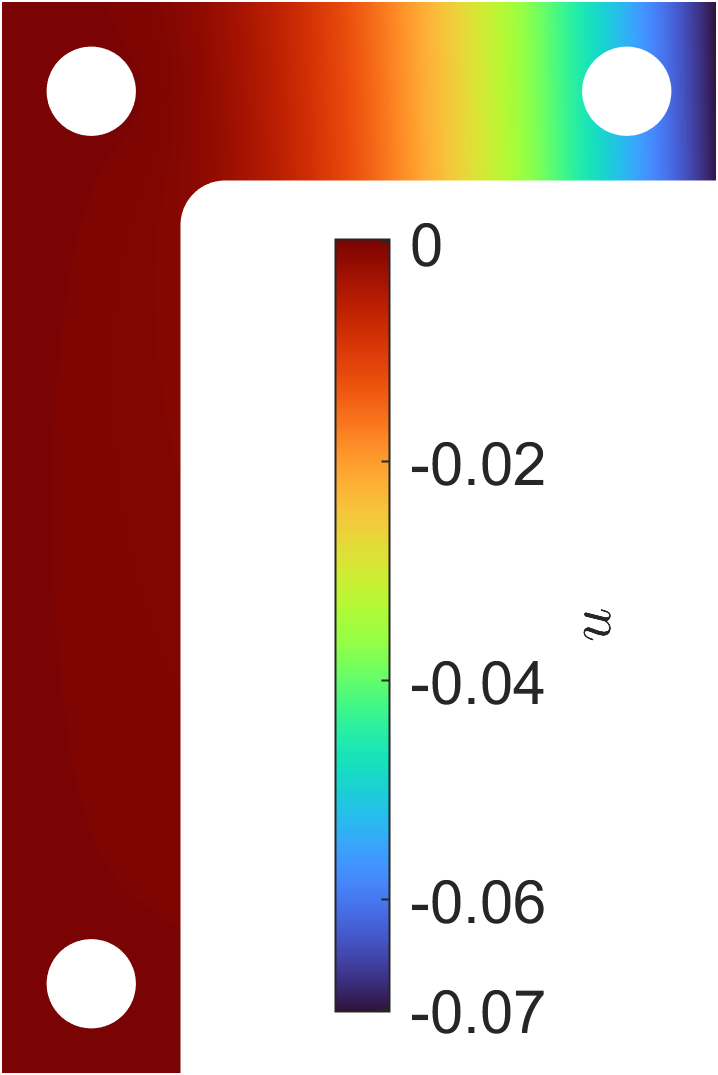}
    \caption*{\footnotesize{$p=3$}}
\end{minipage}
\hspace{2cm}
\begin{minipage}{0.32\textwidth}
    \includegraphics[width=\textwidth]{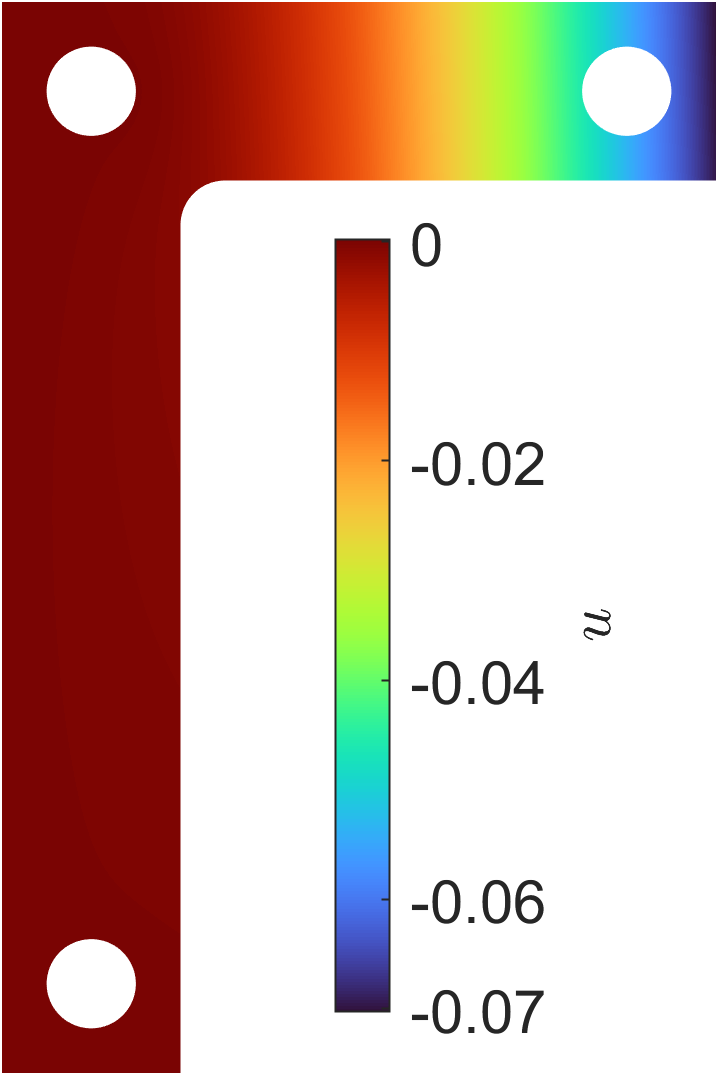}
    \caption*{\footnotesize{$p=4$}}
\end{minipage}
\hfill
%\begin{minipage}{0.32\textwidth}
%    \includegraphics[width=\textwidth]{Images_NE/Ex5/Coradello_u.PNG}
%    \caption*{\footnotesize{\cite{Coradello2021}}}
%\end{minipage}
\caption{ \small Example of the L-bracket with line load. On the left and in the middle, the deformation plots of the proposed approach are shown for the mesh Fig. \ref{fig:Ex_5_Geo_Mesh} (b), $r = 1$ and polynomial order of $p=3$ and $p=4$, respectively.}
\label{fig:Ex5_Deform}
\end{figure}
We present the numerical results in terms of the deformation $u$ and the bending moments $m_{11}$, $m_{22}$ and $m_{12}$. These are computed utilizing the bending stress tensor $\mathbf{m}=(m_{ij})$ defined component-wise as
\begin{equation}
m_{11} = D \, \partial_{xx} u + \nu \, \partial_{yy} u  \qquad m_{22} = D \, \partial_{yy} u + \nu \, \partial_{xx} u \qquad m_{12} = D (1-\nu) \, \partial_{xy} u.
\end{equation}
The plots of the deformation $u$ for mesh $h=1/4$ and orders $p=3$ and $p=4$ respectively are shown in Fig. \ref{fig:Ex5_Deform}. If we look at the latter  one observes a good agreement with results given by \cite[Fig. 14 (b)]{Coradello2021}. Moreover,  the bending moments are presented in Fig. \ref{fig:Ex5Moment}. Compared to \cite[Fig. 15 (d)-(f)]{Coradello2021} we obtain  very similar approximations and   the extremal values at the upper left hole and the vertical part of the bracket fit to the mentioned reference. We note  that the proposed method is capable to use the exact geometries. Further, boundary conditions can be applied at trimmed curves. 
\begin{figure}[H]
\centering
\includegraphics[width=.33\textwidth]{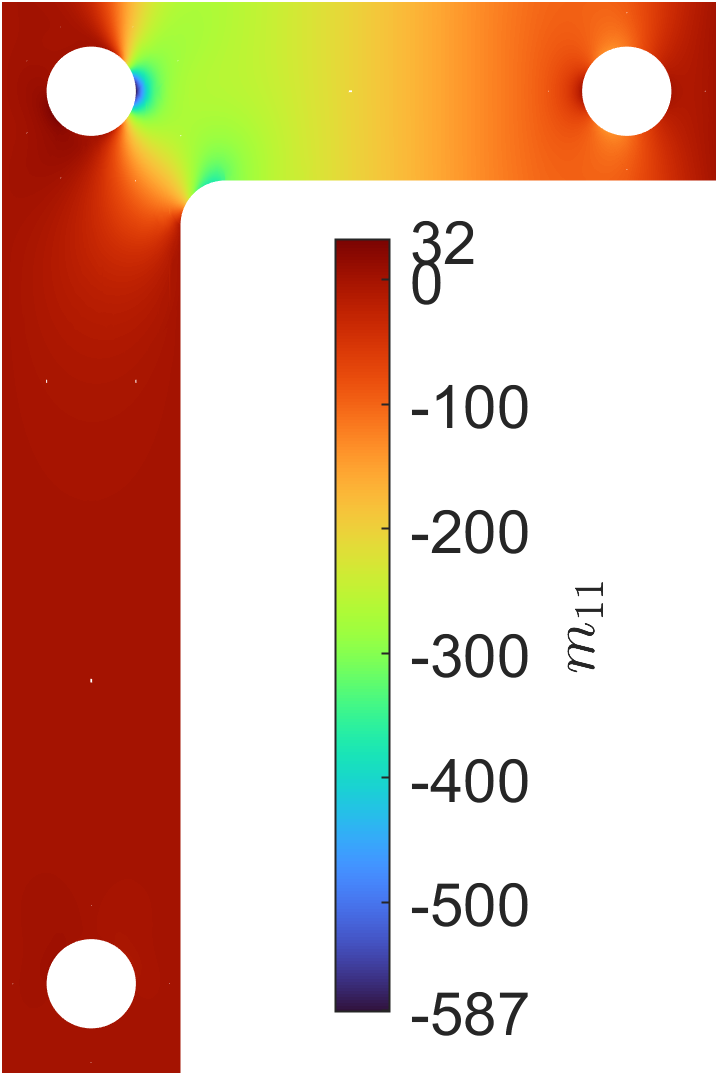}\hfill
\includegraphics[width=.33\textwidth]{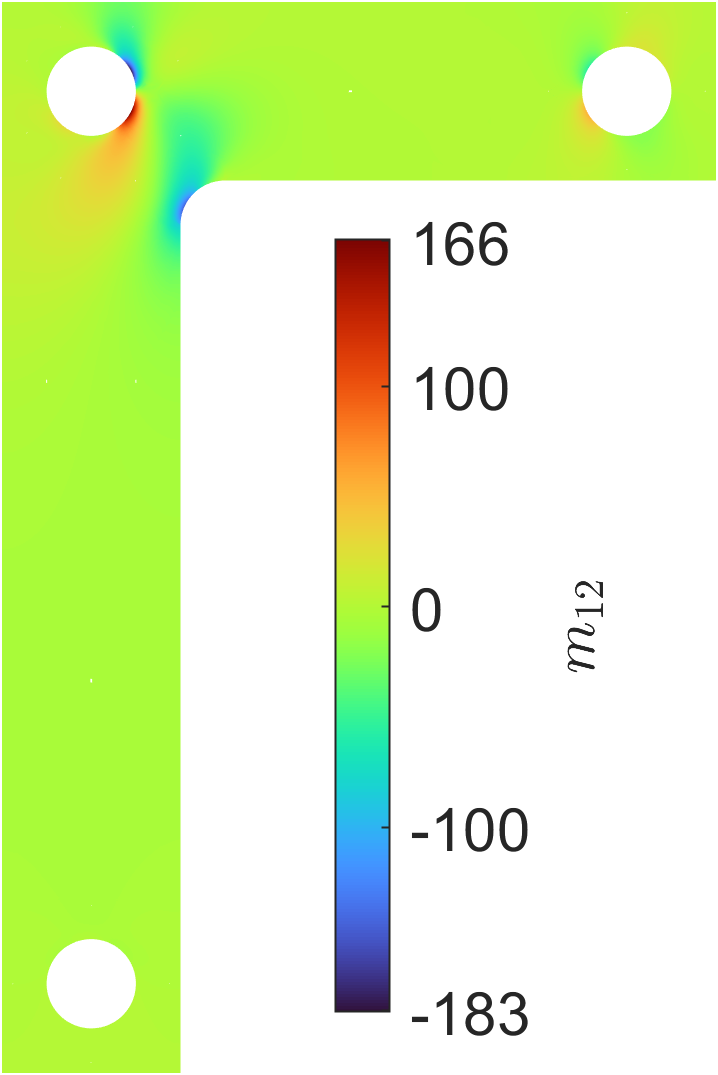}\hfill
\includegraphics[width=.33\textwidth]{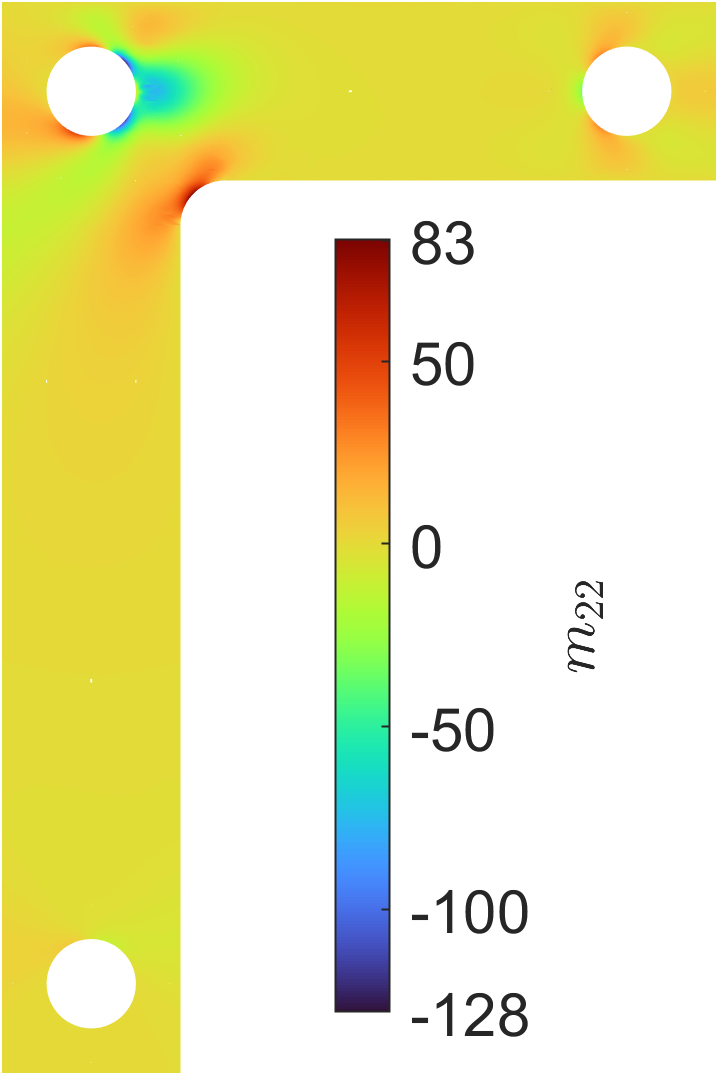}

\caption{ \small Example of the L-bracket with line load. The results of the proposed approach with mesh of $h = 1/4$ and $p = 4$, $r = 1$ are shown.}
\label{fig:Ex5Moment}

\end{figure}

\subsection{A clamped violin}
In this  example, we demonstrate the applicability of the coupling procedure in the context of more complicated geometries involving non-trivial trimmed parts. In more detail, we look at a KL-plate which has the shape given in Fig. \ref{fig:Ex_6_Geo_Mesh} (a). This geometry is inspired by the violin example in \cite{Coradello2020hierarchically}, however, here the setting as  well as the geometry  differs. We note that the boundaries of Fig. \ref{fig:Ex_6_Geo_Mesh} (a) are represented by degree $3$ B-splines.
A decomposition of the  domain into star-shaped blocks with straight interfaces is possible and for the computations we use the mesh structure given in  Fig. \ref{fig:Ex_6_Geo_Mesh} (b).  We consider the case of a uniform load function $f= -0.01$ and as boundary conditions we require the clamping of the inner boundaries that define the f-holes. All the other boundaries can move freely. Further we choose the material parameters $E=10^5, \ \nu=0.1, \ t=0.2$. Using B-splines of degree $3$ and regularity $1$ for the determination of the coupled basis functions we obtain as deformation the result in  Fig. \ref{fig:Ex_6_Deform}. Although a reference solution is missing, we regard the numerical deformations as reasonable and the coupling ansatz seems to work for this advanced setting.
\begin{figure}[H]	
\hspace*{-1cm}
	\begin{tikzpicture}[scale=1.1]    
	    \node (eins) at (0.3-1.5,0.7) {\includegraphics[width=0.57\linewidth]{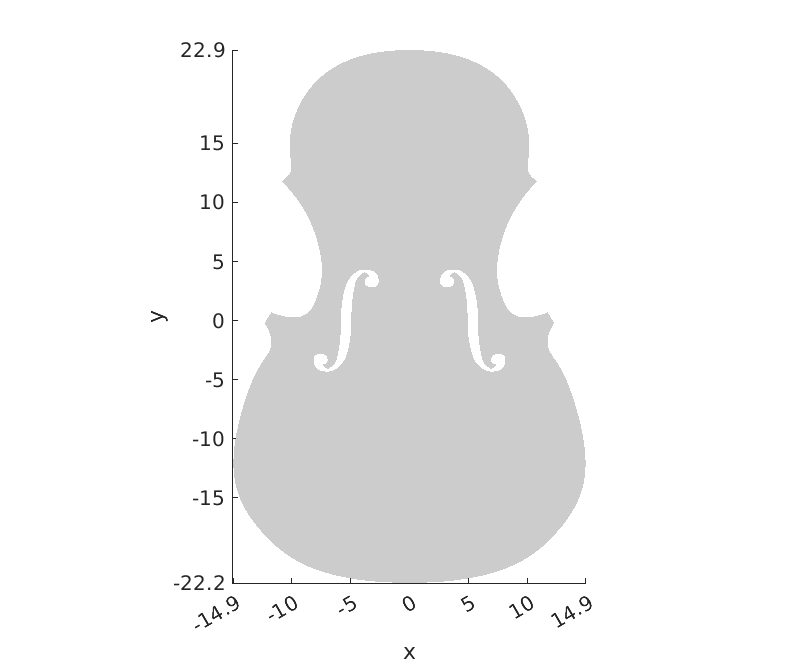}};
         %\node (zwei) at (-0.45-0.3+6.6,0.69) {\includegraphics[width=0.57\linewidth]{Images_NE/Ex6/Plate_Violin_Mesh.png}};
         \node (zwei) at (-0.45-0.3+6.6,0.69) {\includegraphics[width=0.53\linewidth]{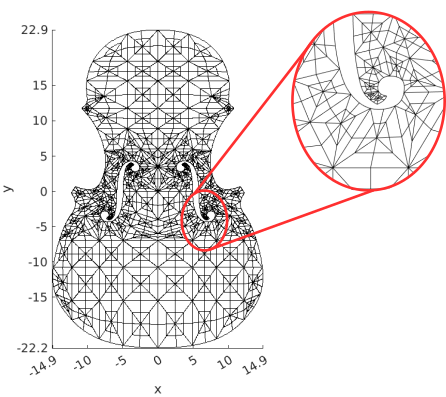}};
         \node at (-2,-3.3) {\footnotesize (a) Geometry};
        \node at (5.4,-3.3) {\footnotesize (b) Mesh structure};
         \node at (-1.07,2.73) {\footnotesize $\boxed{\textup{clamped}}$};
         \draw[->] (-1-0.02-0.05,2.5) -- (-1.5+0.1-0.05,1.45);
         \draw[->] (-1-0.02-0.05,2.5) -- (-0.5-0.1-0.05,1.45);
	\end{tikzpicture}
 \caption{ \small The violin example. On the left side, the computational domain is illustrated. On the right side the underlying mesh is shown.}
    \label{fig:Ex_6_Geo_Mesh}
\end{figure}

\begin{figure}[H]	
\hspace*{-1cm}
	\begin{tikzpicture}[scale=1.1]    
	    \node (eins) at (0.3-1.5,0.7) {\includegraphics[width=0.57\linewidth]{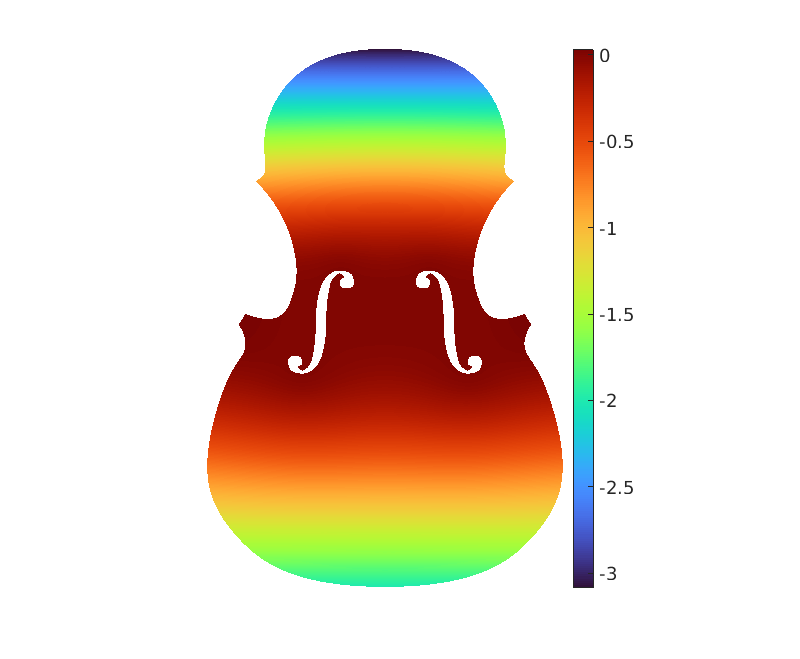}};
         \node(zwei) at (-0.45-0.3+6,0.69) {\includegraphics[width=0.48\linewidth]{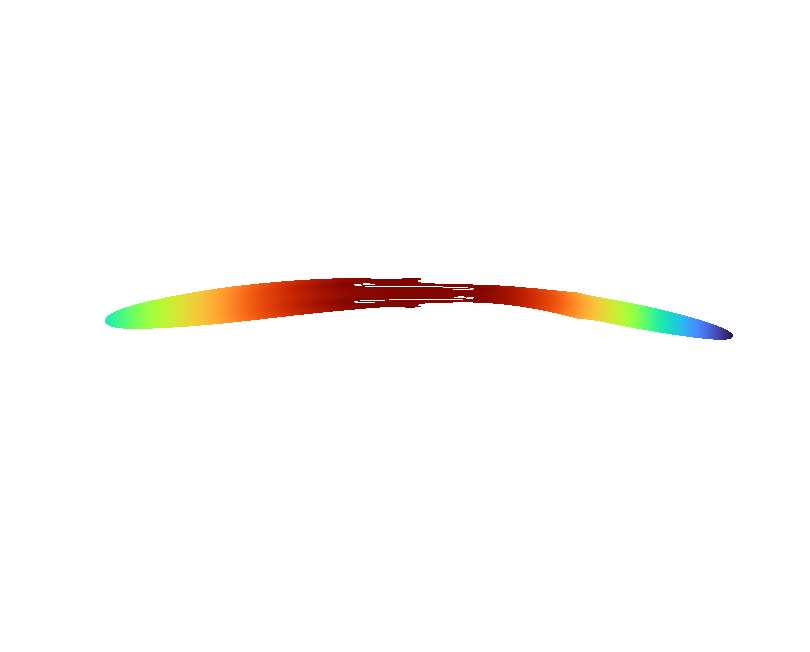}};
     %    \node at (-2,-3.3) {\footnotesize (a) Deformation};
      %  \node at (5.4,-3.3) {\footnotesize (b) Deformation};
      \node at (0.35,1) {$u$};
        % \node at (-1.02,2.7) {\footnotesize $\boxed{\textup{clamped}}$};
         %\draw[->] (-1-0.02,2.5) -- (-1.5-0.02,1.45);
         %\draw[->] (-1-0.02,2.5) -- (-0.5-0.02,1.45);
	\end{tikzpicture}
 \caption{ \small The violin example. On the left side, we show the deformation. On the right side the deformation is illustrated utilizing a side view.}
    \label{fig:Ex_6_Deform}
\end{figure}

\section{Remarks on  stability}
\label{Sec:Stability}
The examples from above confirm that we can apply the $C^1$-coupled basis functions in the context of the Kirchhoff plate. In particular, it is still possible to handle the high-order problem although we have a singular parametrization mapping. Certainly, this singularity comes along with some difficulties. The main bottleneck is the occurrence of large condition numbers for the underlying linear systems, especially if the mesh sizes gets small. The latter causes problems in the sense of stability of the approximate solution and below we  discuss this briefly.\\
First, we bring to mind that the definition of the $C^1$ basis functions implies the boundedeness of the  $H^2$-norm, i.e. $\mathcal{V}_h^{M,1} \subset H^2(\o)$ if we have B-spline ansatz spaces and B-spline parametrizations. 
By construction, we have the smoothness of each basis function in the interior of each mesh element and as already mentioned the global continuity of the first derivatives. Then, if we look at 'Case I' from in Sec. 4. of \cite{takacs2012}, we get with Theorem 4.4. from latter article that the coupled  SB-IGA test functions are in $H^2(\o_m)$ for each patch. To see this, we observe that we use for the coupling method only such B-splines $\widehat{B}_{i,p}^r \cdot \widehat{B}_{j,p}^r$ with $j>2$ or the additional scaling center test functions, which are obviously smooth in a whole neighborhood of the scaling center. But, as piece-wise $H^2$ mappings which are globally $C^1$ we also get the $H^2$ property in the whole domain. %the application of Theorem 5.2. in \cite[Chapter 2]{Braess}, leads indeed to the relation $\mathcal{V}_h^{M,1}\subset H^2(\o)$.  Hence, the discrete weak formulations we exploited for the numerical formulations are well-defined.
Since we use a quadrature formula with discrete evaluation points away from the  scaling center, we obtain  well-defined  integral values and one does not see directly the singularity within the computations of system matrix entries. However,  we can observe large second derivative entries caused by the mesh degeneration.  This comes along
with instability problems for the discrete problem in the case of fine meshes. In other words,
%for small $h$ we get large system matrix entries for the Kirchhoff-plate model and consequently, 
the condition number increases strongly. This instability issue can be seen in figures Fig. \ref{fig:StabilityEx1_add} (a)-(b) and Fig.  \ref{fig:StabilityEx3_add} (a). Here we have to note that   we use for all examples the standard \emph{mldivide()} function from MATLAB \cite{MATLAB:2022} to solve appearing linear systems. \\
A  naive idea to relax the problem of evaluations near the singular point is to reduce the number of basis functions near the scaling center. An easy  way  without much additional effort, is to combine the basis functions of two meshes. One coarse mesh, for the basis functions with support near the scaling center and one refined mesh, for the areas away from the singular point. In more detail, one could proceed as follows. We have given two SB-meshes for  some domain. One fine mesh with   subdivisions w.r.t. both parametric coordinates (see Fig. \ref{Fig:two_mesh} (a)) and secondly we have the analogous SB-mesh of the same domain, but which is only refined in the radial direction, compare Fig. \ref{Fig:two_mesh}. Clearly, the coarse mesh is not suitable to   highly accurate approximations. However,   the mesh elements adjacent to the scaling center have diameter $\mathcal{O}(h)$ and the mesh from Fig. \ref{Fig:two_mesh} (b) is enough for an approximation in the neighborhood of $\f{x}_0$.  Hence, in view of Fig. \ref{Fig:two_mesh}(a) and Fig. \ref{Fig:two_mesh} (b) we consider only those basis functions   in the fine mesh which have a support outside \color{red} $A$ \color{black} which stands for  the union of  mesh elements adjacent to the singular point. This means we consider $$\mathcal{V}_h^{\textup{fine}} \coloneqq \textup{span}\{ \phi \in \mathcal{V}_h^{(1)} \ |  \ \textup{supp}(\phi) \subset \o \backslash \overline{A} \},$$ where $\mathcal{V}_h^{(1)} $ is the (un-coupled) SB-IGA space for the fine mesh.
From the SB-IGA space $\mathcal{V}_h^{(2)}$ corresponding to mesh Fig. \ref{Fig:two_mesh} (b) we add the space $$\mathcal{V}_h^{\textup{coarse}} \coloneqq \textup{span}\{ \phi \in \mathcal{V}_h^{(2)} \ |  \ \textup{supp}(\phi) \cap  \overset{\circ}{A}\neq \emptyset \}.$$  Consequently,
our modified (uncoupled) SB-IGA test function space is $$\mathcal{V}_h^{\textup{combined}} \coloneqq \mathcal{V}_h^{\textup{fine}} \oplus \mathcal{V}_h^{\textup{coarse}}.$$
After this we can  go through similar steps like in the sections before to define $C^1$-regular  global basis functions. This idea can straightforwardly generalized to more complex geometries.
\begin{remark}
    By the definition of the two spaces utilizing the support of the basis functions it is clear that the  basis  from $\mathcal{V}_h^{\textup{fine}} $ together with the basis functions from $\mathcal{V}_h^{\textup{coarse}} $ directly determines a basis for $\mathcal{V}_h^{\textup{combined}} $.
\end{remark}
Although such a combination of two meshes is certainly not the  only method to remove problematic basis functions near the singularity, it is easy to implement and it is helpful to demonstrate the stabilizing influence of a such mesh coarsening near the center. Therefore, we repeat below the convergence test from Sec. \ref{subsection:smooth_solution} as well as the example of the plate with point load from Sec. \ref{subsection:point_load}, but now with the above explained two-mesh ansatz.
If we look at the results in Fig. \ref{fig:StabilityEx1} and \ref{fig:StabilityEx3}, we see an obviously more stable decay behavior compared to the original unstabilized computations. %We get a more stable  error decay for the two-mesh approach although we have for the latter a  reduced number of degrees of freedom. \\

\begin{figure}[h!]	
\centering
\begin{tikzpicture}[scale=0.76]
      \node (fi1) at (-0.015,0) {\includegraphics[width=0.38\linewidth]{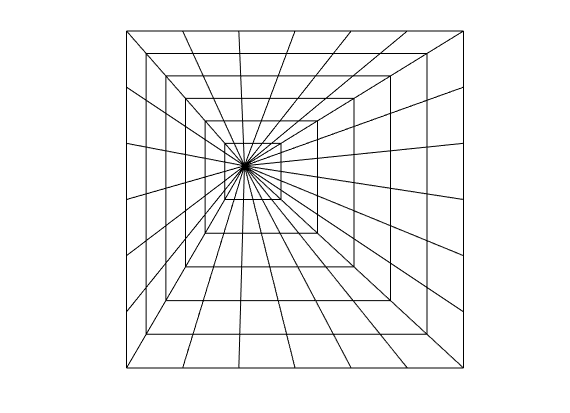}};
      \node (fi2) at (7+1,0) {\includegraphics[width=0.38\linewidth]{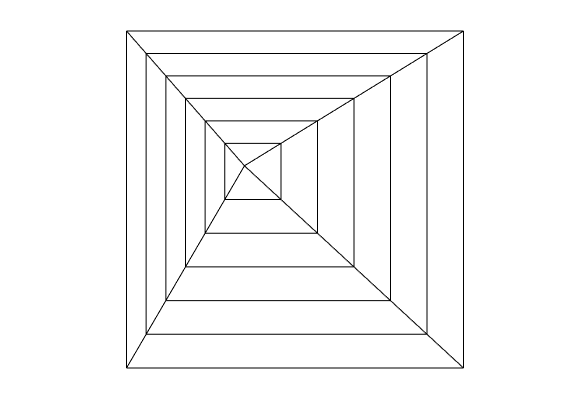}};
      	\draw[fill=blue,opacity=0.4,scale = 4.45, shift={(0.025,0.02)}] (-0.20833,0.166666667) -- (-0.20833,0) -- (-0.5,-0.5) -- (-0.5,0.5)-- (-0.20833,0.166666667);      
       \draw[fill=blue,opacity=0.4,scale=4.45, shift={(0.025,0.02)}] (-0.20833,0) -- (-0.0416667,0) -- (0.5,-0.5) -- (-0.5,-0.5)-- (-0.20833,0);
       		\draw[fill=blue,opacity=0.4,scale=4.45, shift={(0.025,0.02)}]  (-0.0416667,0) -- (-0.0416667,0.166666667) -- (0.5,0.5) -- (0.5,-0.5)--  (-0.0416667,0);
         		\draw[fill=blue,opacity=0.4,scale=4.45,shift={(0.025,0.02)}] (-0.20833,0.166666667) -- (-0.0416667,0.1666667) -- (0.5,0.5) -- (-0.5,0.5)-- (-0.20833,0.166666667);
           \draw[fill=red,opacity=0.4,shift={(6.225+1,0.1)},scale=0.72] (0,0) -- (1,0) -- (1,1) -- (0,1) --(0,0);
           \node at (0,-2.8) { \footnotesize (a) Fine mesh};
            \node at (7+1,-2.8) { \footnotesize (b) Coarse mesh};
             \node[color = blue] at (0,2.7) {$\o \backslash A$};
              \node[color = red] at (7.12+1,2.7) {$A$};
              \draw[<-,dashed,very thick, out=180,color=red, in = 90] (6.93+1,2.6) to (6.33+1,0.8);
  \end{tikzpicture}
  \caption{ \small To demonstrate the influence of the basis functions near the scaling center, we combine the basis functions from two meshes. The right mesh with refinement only in radial direction is exploited for the approximation near the scaling center, whereas the fully refined mesh on the left is only used for the basis functions with support outside the mesh elements adjacent to the singular point. }
   \label{Fig:two_mesh}
\end{figure}

\begin{figure}
\begin{minipage}{0.47\textwidth}
\begin{tikzpicture}
	\small \begin{loglogaxis} [
		width = \textwidth,
		xlabel={ \footnotesize Mesh size $h$},
		xmin=0.025, xmax=0.5,
		ymin=10^-7, ymax=2*10^0,
		xtick={0.03,0.1,0.3,0.5},
		%xticklabel style={/pgf/number format/.cd,fixed,fixed zerofill,precision=2,},
		%x tick label style={/pgf/number format/sci}
		%log ticks with fixed point,
		log x ticks with fixed point,
		ytick={10^-8,10^-7,10^-6,10^-5,10^-4,10^-3,10^-2,10^-1,10^0},
		legend pos=south east,
		legend columns = 3,
		xmajorgrids=true,
		ymajorgrids=true,
		grid style=dashed,
		]
		\addplot%[color=blue, mark size=1.5pt,mark=*]
		coordinates {
			%(0.25, 1.399263687109137)%1.40087115929344)
			%(0.125, 0.268423772379134)%0.268455997061335)
			%(0.083333, 0.108930404582348)%0.108934619370518)
			%(0.0625, 0.058724813366949)%0.0587259114634415)
			%(0.05, 0.036690791454583)%0.0366911952478689)
			%(0.041667, 0.025091902185471)%0.0250920953142907)
			%(0.03571429, 0.018240560754154)%0.0182406925095398)
%			(0.03125, 0.013857469838205)%0.0138575279443904)
    	(0.25,  1.39926367617716552)
			(0.125, 0.268423770832067199)
			(0.083333, 0.108930404294319091)
			(0.0625, 0.0587248132652208252)
			(0.05, 0.0366907951962423262)
			(0.041667,0.0250919024043489043)
			(0.03571429,0.0182408593696157875)
		};
		\addlegendentry{p=3, r=1},
		\addplot%[color=red, mark size=1.5pt,mark=square*]
		coordinates {
			%(0.25, 0.151528229849199) %0.151561791157947)
			%(0.125, 0.016457248449627)%0.016458114373246)
			%(0.083333, 0.004675601235422)%0.00467570000971116)
			%(0.0625, 0.001931999791387)%0.00193201966097626)
			%(0.05, 9.762658932761335e-04)%0.000976178590377506)
			%(0.041667, 5.621679609955389e-04)%0.000563447127672079)
			%(0.03571429, 5.133436625439626e-04)%0.00048837837020767)
%			%(0.03125, 2.437554789883678e-04)%0.000320324931765429)
    (0.25,  0.151528227056671821) %0.151561791157947)
			(0.125, 0.0164572484123510186)%0.016458114373246)
			(0.083333, 0.00467560125681835709)%0.00467570000971116)
			(0.0625, 0.00193210975067249486)%0.00193201966097626)
			(0.05, 0.000976476203762964936)%0.000976178590377506)
			(0.041667, 0.000574792596073049286)%0.000563447127672079)
			(0.03571429, 0.000604918606788243211)%0.00048837837020767)
		};
		\addlegendentry{p=4, r=1},
	    \addplot[color=darkgray,mark=triangle*]%[color=brown, mark=triangle*]%[color=blue, mark size=1.5pt,mark=*]
		coordinates {
		(0.25		,0.013083013832603)
		(0.125		,7.991902728115533e-04)
		(0.083333333,1.527698754171126e-04)
		(0.0625		,4.744731718453517e-05)
		(0.05		,6.634559391139615e-04)
		(0.041666667,7.502001388574731e-04)
		(0.035714286,0.002505621129507)
%		(0.03125,0.007411430883170)
        };
		\addlegendentry{p=5, r=1},
		\addplot[color=blue, mark size=1.5pt,mark=none, dashed]
		coordinates {
			(0.25, 0.0625*10)
			%(0.05, 0.0025*10)
			(0.035714286, 0.035714286^2*10)};
		\addplot[color=red, mark size=1.5pt,mark=none, dashed]
		coordinates {
			(0.25, 0.015625*5)
			%(0.05, 0.000125*5)
			(0.035714286, 0.035714286^3*5)};
		\addplot[color=darkgray, mark size=1.5pt,mark=none, dashed]
		coordinates {
			(0.25, 0.25^4*1.4)
			%(0.05, 0.000125*5)
			(0.0625, 0.0625^4*1.4)};
			\legend{p=3,p=4,p=5,$\mathcal{O}(h^2)$,$\mathcal{O}(h^3)$,$\mathcal{O}(h^4)$}
	\end{loglogaxis}
\end{tikzpicture}
\caption*{\footnotesize{(a) $|u - u_h|_{H^2(\Omega)}$ - unstabilized}}
\end{minipage}
\begin{minipage}{0.47\textwidth}
\begin{tikzpicture}
	\small \begin{loglogaxis} [
		width = \textwidth,
		xlabel={\footnotesize Mesh size $h$},
		xmin=0.025, xmax=0.5,
		ymin=10^-10, ymax=10^-1,
		xtick={0.03,0.1,0.3,0.5},
		%xticklabel style={/pgf/number format/.cd,fixed,fixed zerofill,precision=2,},
		%x tick label style={/pgf/number format/sci}
		%log ticks with fixed point,
		log x ticks with fixed point,
		ytick={10^-12,10^-11,10^-10,10^-9,10^-8,10^-7,10^-6,10^-5,10^-4,10^-3,10^-2},
		legend pos=south east,
		legend columns = 3,
		xmajorgrids=true,
		ymajorgrids=true,
		grid style=dashed,
		]
		\addplot%[color=blue, mark size=1.5pt,mark=*]
		coordinates {
			%(0.25, 0.005575009395933)%0.00557500938985469)
			%(0.125, 2.047837887545693e-04)%0.000204783432363701 )
			%(0.083333, 2.915809773470997e-05)%0.0000292)
			%(0.0625, 7.620936039873554e-06)%0.0000076)
			%(0.05, 2.698861891367287e-06)%2.72247088440229*10^-6)
			%(0.041667, 1.238083099850577e-06)%9.70242987496485*10^-7)
			%(0.03571429, 1.631339608637220e-06)%1.87596427365346*10^-6)
%			(0.03125, 6.965339258950247e-07)%3.50347189022939*10^-7)
    (0.25, 0.00557450082300144957)%0.00557500938985469)
			(0.125, 0.00020477866809034384)%0.000204783432363701 )
			(0.083333, 2.91587658032949004e-05)%0.0000292)
			(0.0625, 7.5153113420398704e-06)%0.0000076)
			(0.05, 3.03852143124313459e-06)%2.72247088440229*10^-6)
			(0.041667, 1.29630910171996071e-06)%9.70242987496485*10^-7)
			(0.03571429, 3.1072362669036115e-06)%1.87596427365346*10^-6)
		};
		\addlegendentry{p=3, r=1},
		\addplot%[color=red, mark size=1.5pt,mark=square*]
		coordinates {
			%(0.25, 1.275933013831962e-04)%0.000127593312690345)
			%(0.125, 2.580775057789607e-06)%2.58060876957345*10^-6)
			%(0.083333, 2.847855376544914e-07)%2.83286336064253*10^-7)
			%(0.0625, 8.407587721361716e-08)%6.49883343534805*10^-8)
			%(0.05, 3.105331711875771e-07)%6.57061297503098*10^-8)
			%(0.041667, 1.215035805582758e-06)%1.48584983197029*10^-6)
			%(0.03571429, 8.555379071959696e-06)%7.72334519730736*10^-6)
%			(0.03125, 1.598739795924605e-06)%4.93962751703958*10^-6)
    (0.25, 0.000127586616855702214)%0.000127593312690345)
			(0.125, 2.5797286779233216e-06)%2.58060876957345*10^-6)
			(0.083333, 2.85256500050647393e-07)%2.83286336064253*10^-7)
			(0.0625,4.75358633139860764e-07)%6.49883343534805*10^-8)
			(0.05, 5.63566610581946993e-07)%6.57061297503098*10^-8)
			(0.041667, 3.00975786520166108e-06)%1.48584983197029*10^-6)
			(0.03571429,1.12739648615629765e-05)%7.72334519730736*10^-6)
		};
		\addlegendentry{p=4, r=1},
	    \addplot[color=darkgray,mark=triangle*]%[color=brown, mark=triangle*]%[color=blue, mark size=1.5pt,mark=*]
		coordinates {
		%(0.25		,5.077253109601121e-06)
		%(0.125		,8.463281324382850e-08)
		%(0.083333333,7.792366353730157e-08)
		%(0.0625		,2.401400206180703e-08)
		%(0.05		,1.513016371265161e-05)
		%(0.041666667,1.701486189060420e-05)
		%(0.035714286,5.721166872093712e-05)
%		(0.03125,1.687706600803239e-04)
         (0.25		,5.07688812116633106e-06)
		(0.125		,7.87332589216156844e-08)
		(0.083333333,1.23898538775868884e-06)
		(0.0625		,2.89932358364034338e-06)
		(0.05		,1.83433599487904875e-06)
		(0.041666667,2.27460416618543647e-05)
		(0.035714286,4.76655138278815003e-05)
        };
		\addlegendentry{p=5, r=1},
		\addplot[color=blue, mark size=1.5pt,mark=none, dashed]
		coordinates {
			(0.25, 0.00390625*0.22)
			%(0.05, 0.00000625*0.22)
			(0.035714286, 0.035714286^4*0.22)};
		\addplot[color=red, mark size=1.5pt,mark=none, dashed]
		coordinates {
			(0.25, 0.25^5*0.03)
			%(0.0625, 0.00000095367*0.165)
			(0.0625, 0.0625^5*0.03)};
		\addplot[color=darkgray, mark size=1.5pt,mark=none, dashed]
		coordinates {
			(0.25, 0.25^6*0.008)
			%(0.0625, 0.00000095367*0.165)
			(0.125, 0.125^6*0.008)};
					\legend{,,,$\mathcal{O}(h^4)$,$\mathcal{O}(h^5)$,$\mathcal{O}(h^6)$}
	\end{loglogaxis}
\end{tikzpicture}
\caption*{\footnotesize{(b) $||u - u_h||_{L^2(\Omega)}$ - unstabilized}}
\end{minipage}

\vspace{0.5cm}

\begin{minipage}{0.47\textwidth}
\begin{tikzpicture}
	\small \begin{loglogaxis} [
		width = \textwidth,
		xlabel={\footnotesize Mesh size $h$},
		xmin=0.025, xmax=0.5,
		ymin=10^-8, ymax=2*10^0,
		xtick={0.03,0.1,0.3,0.5},
		%xticklabel style={/pgf/number format/.cd,fixed,fixed zerofill,precision=2,},
		%x tick label style={/pgf/number format/sci}
		%log ticks with fixed point,
		log x ticks with fixed point,
		ytick={10^-8,10^-7,10^-6,10^-5,10^-4,10^-3,10^-2,10^-1,10^0},
		legend pos=south east,
		legend columns = 3,
		xmajorgrids=true,
		ymajorgrids=true,
		grid style=dashed,
		]
		\addplot%[color=blue, mark size=1.5pt,mark=*]
		coordinates {
			(0.25, 1.492521055221044)
			(0.125, 0.274629385903966)
			(0.083333, 0.110235188747572)
			(0.0625, 0.059151912186954)
			(0.05, 0.036869303477144)
			(0.041667, 0.025179134888995)
			(0.03571429, 0.018288017855597)
%			(0.03125, 0.013885508527461)
		};
		\addlegendentry{p=3, r=1},
		\addplot%[color=red, mark size=1.5pt,mark=square*]
		coordinates {
			(0.25, 0.157284644790817)
			(0.125, 0.016784154979504)
			(0.083333, 0.004727548413516)
			(0.0625, 0.001945385256940)
			(0.05, 9.807539889703056e-04)
			(0.041667, 5.614899666114134e-04)
			(0.03571429, 3.507753123169701e-04)
		};
		\addlegendentry{p=4, r=1},
	    \addplot[color=darkgray,mark=triangle*]%[color=brown, mark=triangle*]%[color=blue, mark size=1.5pt,mark=*]
		coordinates {
		(0.25		,0.013841207546648)
		(0.125		,8.115371122234539e-04)
		(0.083333333,1.538438058085677e-04)
		(0.0625		,4.763267092382201e-05)
		(0.05		,1.926029983518615e-05)
		(0.041666667,9.210240400518621e-06)
		(0.035714286,4.943177187758049e-06)
%		(0.03125    ,2.890810688372081e-06)
        };
		\addlegendentry{p=5, r=1},		
		\addplot[color=blue, mark size=1.5pt,mark=none, dashed]
		coordinates {
			(0.25, 0.0625*10)
			%(0.05, 0.0025*10)
			(0.03125, 0.0009765625*10)};
		\addplot[color=red, mark size=1.5pt,mark=none, dashed]
		coordinates {
			(0.25, 0.015625*5)
			%(0.05, 0.000125*5)
			(0.03571429, 0.00004555395226*5)};
		\addplot[color=darkgray, mark size=1.5pt,mark=none, dashed]
		coordinates {
			(0.25, 0.25^4*1.4)
			%(0.05, 0.000125*5)
			(0.03125, 0.03125^4*1.4)};
			\legend{p=3,p=4,p=5,$\mathcal{O}(h^2)$,$\mathcal{O}(h^3)$,$\mathcal{O}(h^4)$}
	\end{loglogaxis}
\end{tikzpicture}
\caption*{\footnotesize{(c) $|u - u_h|_{H^2(\Omega)}$ - stabilized}}
\end{minipage}
\begin{minipage}{0.47\textwidth}
\begin{tikzpicture}
	\small \begin{loglogaxis} [
		width = \textwidth,
		xlabel={\footnotesize Mesh size $h$},
		xmin=0.025, xmax=0.5,
		ymin=10^-12, ymax=10^-1,
		xtick={0.03,0.1,0.3,0.5},
		%xticklabel style={/pgf/number format/.cd,fixed,fixed zerofill,precision=2,},
		%x tick label style={/pgf/number format/sci}
		%log ticks with fixed point,
		log x ticks with fixed point,
		ytick={10^-12,10^-11,10^-10,10^-9,10^-8,10^-7,10^-6,10^-5,10^-4,10^-3,10^-2},
		legend pos=south east,
		legend columns = 3,
		xmajorgrids=true,
		ymajorgrids=true,
		grid style=dashed,
		]
		\addplot%[color=blue, mark size=1.5pt,mark=*]
		coordinates {
			(0.25, 0.005669154686944)
			(0.125, 2.079767155047321e-04)
			(0.083333, 2.948389942374852e-05)
			(0.0625, 7.629685747928446e-06)
			(0.05, 2.774588205867356e-06)
			(0.041667, 1.247498377509704e-06)
			(0.03571429, 6.457849848929466e-07)
%			(0.03125, 3.691769056465477e-07)
		};
		\addlegendentry{p=3, r=1},
		\addplot%[color=red, mark size=1.5pt,mark=square*]
		coordinates {
			(0.25, 1.307380204737373e-04)
			(0.125, 2.650613783670587e-06)
			(0.083333, 2.909131750288367e-07)
			(0.0625, 6.239157843619700e-08)
			(0.05, 1.909179833045795e-08)
			(0.041667, 7.281615081381058e-09)
			(0.03571429, 3.363420992606157e-09)
		};
		\addlegendentry{p=4, r=1},
		\addplot[color=darkgray,mark=triangle*]%[color=brown, mark=triangle*]%[color=blue, mark size=1.5pt,mark=*]
		coordinates {
		(0.25		,5.470907662596484e-06)
		(0.125		,6.631132568677453e-08)
		(0.083333333,4.957196344074540e-09)
		(0.0625		,8.841338701536581e-10)
		(0.05		,9.967740108032694e-10)
		(0.041666667,4.308518131273852e-10)
		(0.035714286,2.515677524295298e-09)
%		(0.03125    ,4.269447481049528e-09)
        };
		\addlegendentry{p=5, r=1},		
		\addplot[color=blue, mark size=1.5pt,mark=none, dashed]
		coordinates {
			(0.25, 0.25^4*0.22)
			%(0.05, 0.00000625*0.22)
			(0.035714286, 0.035714286^4*0.22)
			};
		\addplot[color=red, mark size=1.5pt,mark=none, dashed]
		coordinates {
			(0.25, 0.0009765625*0.03)
			%(0.0625, 0.00000095367*0.165)
			(0.03571429, 0.03571429^5*0.03)};
		\addplot[color=darkgray, mark size=1.5pt,mark=none, dashed]
		coordinates {
			(0.25, 0.25^6*0.008)
			%(0.0625, 0.00000095367*0.165)
			(0.0625, 0.0625^6*0.008)};
		\legend{,,,$\mathcal{O}(h^4)$,$\mathcal{O}(h^5)$,$\mathcal{O}(h^6)$}
	\end{loglogaxis}
\end{tikzpicture}
\caption*{\footnotesize{(d) $||u - u_h||_{L^2(\Omega)}$ - stabilized}}
\end{minipage}
    \caption{\small Convergence studies of the $H^2$ seminorm and the $L^2$ norm on the example of the
smooth solution on a square plate with and without the stabilization ansatz  and orders of $p = 3$, $p = 4$ and $p = 5$ $(r=1)$.}
    \label{fig:StabilityEx1_add}
\end{figure}

\begin{figure}%[H]
\begin{minipage}{0.47\textwidth}
\begin{tikzpicture}
	\small \begin{loglogaxis} [
		width = \textwidth,
		xlabel={\footnotesize Mesh size $h$},
		xmin=0.025, xmax=0.5,
		ymin=10^-6, ymax=10^-1,
		xtick={0.03,0.1,0.3,0.5},
		%xticklabel style={/pgf/number format/.cd,fixed,fixed zerofill,precision=2,},
		%x tick label style={/pgf/number format/sci}
		%log ticks with fixed point,
		log x ticks with fixed point,
		ytick={10^-9,10^-8,10^-7,10^-6,10^-5,10^-4,10^-3,10^-2,10^-1},
		legend pos=south east,
		legend columns = 3,
		xmajorgrids=true,
		ymajorgrids=true,
		grid style=dashed,
		]
		\addplot%[color=blue, mark size=1.5pt,mark=*]
		coordinates {
		(0.25		,0.006471083270783)
		(0.125		,0.001611547426389)
		(0.083333333,7.161857774627967e-04)
		(0.0625		,4.037151890995405e-04)
		(0.05		,2.539236560431535e-04)
		(0.041666667,1.785722046213634e-04)
		(0.035714286,1.340516158990779e-04)
        };
		\addlegendentry{p=3},
		\addplot%[color=red, mark size=1.5pt,mark=square*]
		coordinates {
		(0.25		,0.001387571163495)
		(0.125		,3.463683571577336e-04)
		(0.083333333,1.516842726270173e-04)
		(0.0625		,7.446419582757269e-05)
		(0.05		,4.022374767553849e-05)
		(0.041666667,9.987880674855454e-05)
		(0.035714286,2.988404887172003e-05)
        };
		\addlegendentry{p=4},
		\addplot[color=darkgray,mark=triangle*]%[color=brown, mark=triangle*]%[color=blue, mark size=1.5pt,mark=*]
		coordinates {
%		(0.5		,0.00188472 )
		(0.25		,4.618003607584909e-04)
%		(0.166666667,0.000205249)
		(0.125		,1.141109943840757e-04)
%		(0.1		,7.06776E-05)
		(0.083333333,5.013293255384799e-05)
%		(0.071428571,2.11378E-05)
		(0.0625		,7.009795122137064e-08)
%		(0.055555556,2.57362E-05)
		(0.05		,5.700075766768009e-05)
%		(0.045454545,8.3295E-05 )
		(0.041666667,3.370903148769422e-05)
%		(0.038461538,0.000152502)
		(0.035714286,0.001397058960322)
%		(0.033333333,9.38906E-05)
        };
		\addlegendentry{p=5},
%		\addplot[color=blue, mark size=1.5pt,mark=none, dashed]
%		coordinates {
%			(0.25, 0.00390625*0.22)
			%(0.05, 0.00000625*0.22)
%			(0.03125, 0.0000009536743*0.22)
%			};
%		\addplot[color=red, mark size=1.5pt,mark=none, dashed]
%		coordinates {
%			(0.25, 0.0009765625*0.165)
			%(0.0625, 0.00000095367*0.165)
%			(0.03125, 0.0000000298*0.165)};
	\end{loglogaxis}
\end{tikzpicture}
\caption*{\footnotesize{(a) $|1-u/u_{ref}|$ - unstabilized}}
\end{minipage}
\begin{minipage}{0.47\textwidth}
\begin{tikzpicture}
	\small \begin{loglogaxis} [
		width = \textwidth,
		xlabel={\footnotesize Mesh size $h$},
		xmin=0.025, xmax=0.5,
		ymin=10^-6, ymax=10^-1,
		xtick={0.03,0.05,0.1,0.2,0.3,0.5},
		%xticklabel style={/pgf/number format/.cd,fixed,fixed zerofill,precision=2,},
		%x tick label style={/pgf/number format/sci}
		%log ticks with fixed point,
		log x ticks with fixed point,
		ytick={10^-9,10^-8,10^-7,10^-6,10^-5,10^-4,10^-3,10^-2,10^-1},
		legend pos=south east,
		legend columns = 3,
		xmajorgrids=true,
		ymajorgrids=true,
		grid style=dashed,
		]
		\addplot%[color=blue, mark size=1.5pt,mark=*]
		coordinates {
%		(0.5		,9.16E-01)
        (0.25		,2.09E-02)
%        (0.166666667,9.68E-03)
        (0.125		,5.53E-03)
%        (0.1		,3.57E-03)
        (0.083333333,2.49E-03)
%        (0.071428571,1.83E-03)
        (0.0625		,1.40E-03)
%        (0.055555556,1.11E-03)
        (0.05		,9.01E-04)
%        (0.045454545,7.45E-04)
        (0.041666667,6.26E-04)
%        (0.038461538,5.34E-04)
        (0.035714286,4.60E-04)
%        (0.033333333,4.01E-04)
%        (0.03125	,3.53E-04)
%        (0.029411765,3.12E-04)
%        (0.027777778,2.79E-04)
%        (0.026315789,2.50E-04)
%        (0.025		,2.26E-04)
%        (0.023809524,2.05E-04)
%        (0.022727273,1.87E-04)
%        (0.02173913	,1.71E-04)
%        (0.020833333,1.57E-04)
%        (0.02		,1.45E-04)
%        (0.019230769,1.34E-04)
%        (0.018518519,1.24E-04)
%        (0.017857143,1.15E-04)
%        (0.017241379,1.07E-04)
%        (0.016666667,1.00E-04)
%        (0.016129032,9.41E-05)
%        (0.015625	,8.85E-05)
%        (0.015151515,8.37E-05)
%        (0.014705882,7.83E-05)
%        (0.014285714,7.36E-05)
%        (0.013888889,7.02E-05)
%        (0.013513514,6.48E-05)
%        (0.013157895,6.20E-05)
		};
		\addlegendentry{p=3},
		\addplot%[color=red, mark size=1.5pt,mark=square*]
		coordinates {
%		(0.5		,0.054397889)
		(0.25		,0.001754921)
%		(0.166666667,0.000780549)
		(0.125		,0.000439081)
%		(0.1		,0.000281011)
		(0.083333333,0.000195145)
%		(0.071428571,0.000143372)
		(0.0625		,0.000109768)
%		(0.055555556,8.67321E-05)
		(0.05		,7.02494E-05)
%		(0.045454545,5.80554E-05)
		(0.041666667,4.87813E-05)
%		(0.038461538,4.15667E-05)
		(0.035714286,3.58335E-05)
%		(0.033333333,3.12225E-05)
%		(0.03125	,2.7438E-05 )
%		(0.029411765,2.42978E-05)
%		(0.027777778,2.16511E-05)
%		(0.026315789,1.94596E-05)
%		(0.025		,1.76E-05   )
%		(0.023809524,1.59144E-05)
%		(0.022727273,1.44526E-05)
        };
		\addlegendentry{p=4},
		\addplot[color=darkgray,mark=triangle*]%[color=brown, mark=triangle*]%[color=red, mark size=1.5pt,mark=square*]
		coordinates {
%		(0.5		,0.002577163)
		(0.25		,0.000638428)
%		(0.166666667,0.00028356 )
		(0.125		,0.000159483)
%		(0.1		,0.000102066)
		(0.083333333,7.08814E-05)
%		(0.071428571,5.20758E-05)
		(0.0625		,3.98724E-05)
%		(0.055555556,3.14943E-05)
		(0.05		,2.55068E-05)
%		(0.045454545,2.10799E-05)
		(0.041666667,1.77162E-05)
%		(0.038461538,1.51047E-05)
		(0.035714286,1.30192E-05)
%		(0.033333333,1.13988E-05)
};
		\addlegendentry{p=5},
%		\addplot[color=blue, mark size=1.5pt,mark=none, dashed]
%		coordinates {
%			(0.25, 0.00390625*0.22)
			%(0.05, 0.00000625*0.22)
%			};
%		\addplot[color=red, mark size=1.5pt,mark=none, dashed]
%		coordinates {
%			(0.25, 0.0009765625*0.165)
			%(0.0625, 0.00000095367*0.165)
%			(0.03125, 0.0000000298*0.165)};
	\end{loglogaxis}
\end{tikzpicture}
\caption*{\footnotesize{(b) $|1-u/u_{ref}|$ - stabilized }}
\end{minipage}
    \caption{ \small Convergence studies of the unstabilized and stabilized results on the example of the point load on a square plate and orders of $p = 3$, $p = 4$ and $p = 5$ $(r=1)$.}
   \label{fig:StabilityEx3_add}
\end{figure}

These first experiments with modified spaces  reveal the importance of a good mesh choice, at least when dealing with high order problems. For sure, different refinement strategies should be discussed, too,  maybe also in combination with some adaptive scheme. And the application of the naive two mesh ansatz in the context of more complicated test cases is advisable. But, it is not our goal of this article to  study the stability issue in detail. This seems a reasonable object of investigations with another publication.\\
We want to conclude this subsection with the following summarizing words.
The SB-IGA space with $C^1$ coupled basis functions suffers from bad conditions numbers, especially if we look at high order problems combined with fine meshes. 
Nevertheless, there is hope to alleviate this problem. We think that removing basis functions near the scaling center leads to increased stability in the computations without reducing the overall convergence behaviour significantly.

\section{Conclusion}
\label{section:conclusion}
In this contribution, an approach of the incorporation of SB-IGA patches and $C^1$ coupling in terms of analysis suitable $G^1$ parametrization was presented. This implies a special consideration of the basis functions at the scaling center. The method was tested in the context of the Kirchhoff plate formulation. It is especially suitable for trimmed models since the boundary representation can easily replace the existing boundary by the trimming curve, no matter if the trimming curve is entirely inside the domain or partially outside. The proposed approach was tested in various ways such as the $L^2$ norm and the $H^2$ seminorm against the optimal convergence rates, the bending moments and the deflection at specific points on the accuracy per degree. The stability issues arising at the scaling center for fine meshes are discussed and a stabilization remedy is presented that yields improved results. Moreover, the results were compared to other approaches considering meshing for unfavorably shaped problems and to results from other coupling approaches in isogeometric analysis.  \\
In conclusion, we presented the feasibility of SB-IGA with $C^1$ coupling in terms of Kirchhoff plates that is especially powerful for sophisticated structures.

\paragraph{Acknowledgments}
%The financial support of the DFG (German Research Foundation) under Grant No. KL1345/10-2 (project number: 667493) is gratefully acknowledged.
The financial support of the DFG (German Research Foundation) under Grant No. KL1345/10-2 (project number: 667493) and Grant No. SI756/5-2 (project number: 667494) is gratefully acknowledged.

 %\nocite{*}
%\bibliographystyle{alphadin}
%\bibliography{Literatur}
%\bibliographystyle{alphadin}
\bibliographystyle{IEEEtran}
\bibliography{manuscript}

\begin{thebibliography}{10}
\providecommand{\url}[1]{#1}
\csname url@samestyle\endcsname
\providecommand{\newblock}{\relax}
\providecommand{\bibinfo}[2]{#2}
\providecommand{\BIBentrySTDinterwordspacing}{\spaceskip=0pt\relax}
\providecommand{\BIBentryALTinterwordstretchfactor}{4}
\providecommand{\BIBentryALTinterwordspacing}{\spaceskip=\fontdimen2\font plus
\BIBentryALTinterwordstretchfactor\fontdimen3\font minus
  \fontdimen4\font\relax}
\providecommand{\BIBforeignlanguage}[2]{{%
\expandafter\ifx\csname l@#1\endcsname\relax
\typeout{** WARNING: IEEEtran.bst: No hyphenation pattern has been}%
\typeout{** loaded for the language `#1'. Using the pattern for}%
\typeout{** the default language instead.}%
\else
\language=\csname l@#1\endcsname
\fi
#2}}
\providecommand{\BIBdecl}{\relax}
\BIBdecl

\bibitem{IGA2}
T.~J.~R. Hughes, J.~Cottrell, and Y.~Bazilevs, ``{Isogeometric Analysis: CAD,
  finite elements, NURBS, exact geometry and mesh refinement},'' \emph{Computer
  Methods in Applied Mechanics and Engineering}, vol. 194, pp. 4135--4195, 10
  2005.

\bibitem{IGA1}
A.~Buffa and G.~Sangalli, \emph{Iso{G}eometric {A}nalysis: {A} {N}ew {P}aradigm
  in the {N}umerical {A}pproximation of {PDE}s: {C}etraro, {I}taly 2012}.\hskip
  1em plus 0.5em minus 0.4em\relax Springer, 2016, vol. 2161.

\bibitem{IGA3}
Y.~Bazilevs, L.~Veiga, J.~Cottrell, T.~Hughes, and G.~Sangalli,
  ``{I}sogeometric {A}nalysis: {A}pproximation, stability and error estimates
  for h-refined meshes,'' \emph{Mathematical Models and Methods in Applied
  Sciences}, vol.~16, no.~07, pp. 1031--1090, 2006.

\bibitem{Collin2016AnalysissuitableGM}
A.~Collin, G.~Sangalli, and T.~Takacs, ``Analysis-suitable {G}$^1$ multi-patch
  parametrizations for {C}$^1$ isogeometric spaces,'' \emph{Computer Aided
  Geometric Design}, vol.~47, pp. 93--113, 2016.

\bibitem{Dittmann2019}
M.~Dittmann, S.~Schuß, B.~Wohlmuth, and C.~Hesch, ``Weak {C}$^n$ coupling for
  multipatch isogeometric analysis in solid mechanics,'' \emph{International
  Journal for Numerical Methods in Engineering}, vol. 118, pp. 678--699, 6
  2019.

\bibitem{Dornisch2015}
W.~Dornisch, G.~Vitucci, and S.~Klinkel, ``The weak substitution method--an
  application of the mortar method for patch coupling in {NURBS}-based
  isogeometric analysis,'' \emph{International Journal for Numerical Methods in
  Engineering}, vol. 103, no.~3, pp. 205--234, 2015.

\bibitem{Chasapi2020}
M.~Chasapi, W.~Dornisch, and S.~Klinkel, ``Patch coupling in isogeometric
  analysis of solids in boundary representation using a mortar approach,''
  \emph{International Journal for Numerical Methods in Engineering}, vol. 121,
  no.~14, pp. 3206--3226, 2020.

\bibitem{toshniwal2017}
D.~Toshniwal, H.~Speleers, R.~R. Hiemstra, and T.~J.~R. Hughes, ``Multi-degree
  smooth polar splines: {A} framework for geometric modeling and isogeometric
  analysis,'' \emph{Computer Methods in Applied Mechanics and Engineering},
  vol. 316, pp. 1005--1061, 2017.

\bibitem{Arioli2019}
C.~Arioli, A.~Shamanskiy, S.~Klinkel, and B.~Simeon, ``Scaled boundary
  parametrizations in isogeometric analysis,'' \emph{Computer Methods in
  Applied Mechanics and Engineering}, vol. 349, pp. 576--594, 2019.

\bibitem{Klinkel2019}
S.~Klinkel and R.~Reichel, ``A finite element formulation in boundary
  representation for the analysis of nonlinear problems in solid mechanics,''
  \emph{Computer Methods in Applied Mechanics and Engineering}, vol. 347, pp.
  295--315, 2019.

\bibitem{Chasapi2018}
M.~Chasapi and S.~Klinkel, ``A scaled boundary isogeometric formulation for the
  elasto-plastic analysis of solids in boundary representation,''
  \emph{Computer Methods in Applied Mechanics and Engineering}, vol. 333, pp.
  475--496, 2018.

\bibitem{liu2018}
N.~Liu and A.~E. Jeffers, ``A geometrically exact isogeometric {K}irchhoff
  plate: {F}eature-preserving automatic meshing and {C}$^1$ rational triangular
  {B}{\'e}zier spline discretizations,'' \emph{International Journal for
  Numerical Methods in Engineering}, vol. 115, no.~3, pp. 395--409, 2018.

\bibitem{Dieringer2011}
R.~Dieringer, J.~Hebel, and W.~Becker, ``The scaled boundary finite element
  method for plate bending problems,'' \emph{Computer Methods in Mechanics},
  2011.

\bibitem{Marussig2018}
B.~Marussig and T.~J.~R. Hughes, ``A {R}eview of {T}rimming in {I}sogeometric
  {A}nalysis: {C}hallenges, {D}ata {E}xchange and {S}imulation {A}spects,''
  \emph{Archives of Computational Methods in Engineering}, vol.~25, no.~4, pp.
  1059--1127, 2018.

\bibitem{MATLAB:2022}
MATLAB, \emph{version 9.13.0.2049777 (R2022b)}.\hskip 1em plus 0.5em minus
  0.4em\relax Natick, Massachusetts: The MathWorks Inc., 2022.

\bibitem{geopdesv3}
R.~V\'azquez, ``A new design for the implementation of isogeometric analysis in
  {O}ctave and {M}atlab: {G}eo{PDE}s 3.0.''

\bibitem{Reali2015}
A.~Reali and H.~Gomez, ``An isogeometric collocation approach for
  {B}ernoulli--{E}uler beams and {K}irchhoff plates,'' \emph{Computer Methods
  in Applied Mechanics and Engineering}, vol. 284, pp. 623--636, 2015.

\bibitem{Ciarlet2002}
P.~G. Ciarlet, \emph{The finite element method for elliptic problems}, 2002.

\bibitem{reddy2006}
J.~N. Reddy, \emph{Theory and {A}nalysis of {E}lastic {P}lates and
  {S}hells}.\hskip 1em plus 0.5em minus 0.4em\relax CRC press, 2006.

\bibitem{Coradello2021}
L.~Coradello, G.~Loli, and A.~Buffa, ``A projected super-penalty method for the
  $c^{1}$-coupling of multi-patch isogeometric {K}irchhoff plates,''
  \emph{Computational Mechanics}, vol.~67, no.~4, pp. 1133--1153, 2021.

\bibitem{Benzaken2017}
J.~Benzaken, A.~J. Herrema, M.-C. Hsu, and J.~Evans, ``A rapid and efficient
  isogeometric design space exploration framework with application to
  structural mechanics,'' \emph{Computer Methods in Applied Mechanics and
  Engineering}, vol. 316, pp. 1215--1256, 2017.

\bibitem{Coradello2020hierarchically}
L.~Coradello, D.~D’Angella, M.~Carraturo, J.~Kiendl, S.~Kollmannsberger,
  E.~Rank, and A.~Reali, ``Hierarchically refined isogeometric analysis of
  trimmed shells,'' \emph{Computational Mechanics}, vol.~66, no.~2, pp.
  431--447, 2020.

\bibitem{takacs2012}
T.~Takacs and B.~J{\"u}ttler, ``${H^2}$ regularity properties of singular
  parameterizations in isogeometric analysis,'' \emph{Graphical {M}odels},
  vol.~74, no.~6, pp. 361--372, 2012.

\end{thebibliography}
%\printbibliography	

\end{document}